\documentclass[12pt]{article}
\usepackage{epsfig}
\usepackage{amsfonts}
\usepackage{amssymb}
\usepackage{amsmath}
\usepackage{amsthm}
\usepackage{latexsym}
\usepackage{graphicx}
\usepackage{bm}
\usepackage{tabularx}
\usepackage{tikz}
\usetikzlibrary{arrows.meta}
\usetikzlibrary{cd}
\usepackage{booktabs}
\usepackage{enumerate}
\usepackage{caption}
\usepackage[titletoc]{appendix}
\usepackage[numbers]{natbib}
\usepackage{float}
\usepackage{comment}
\usepackage{tabularx}
\usepackage{fancyhdr}
\usepackage{epstopdf}
\usepackage{subcaption}

\catcode`\@=11

\newcommand{\E}[1]{{\mathrm{E}\left[#1\right]}}

\newcommand{\PP}[1]{\mathrm{P}\left( #1 \right)}
\newcommand{\Var}[1]{{\mathrm{Var}\left(#1\right)}}
\newcommand{\Cov}[1]{{\mathrm{Cov}[#1]}}
\newcommand{\T}{{\mathrm{T}}}

\newcommand{\onevec}[1]{{\boldsymbol{\mathrm{v}}_{\hspace{-.01in}#1}\hspace{-.01in}}}

\usepackage[pdftitle={Ephemeral Self-Excitement},
  pdfauthor={DawPender},
  pdffitwindow=true,
  breaklinks=true,
  colorlinks=true,
  urlcolor=blue,
  linkcolor=red,
  citecolor=red,
  anchorcolor=red]{hyperref}

\numberwithin{equation}{section}
\numberwithin{table}{section}
\numberwithin{figure}{section}

\theoremstyle{plain}
\newtheorem{theorem}{Theorem}[section]
\newtheorem{lemma}[theorem]{Lemma}
\newtheorem{corollary}[theorem]{Corollary}
\newtheorem{proposition}[theorem]{Proposition}

\theoremstyle{definition}
\newtheorem{definition}{Definition}[section]

\theoremstyle{remark}

\begin{document}

\title{An Ephemerally Self-Exciting Point Process}
\author{
  Andrew Daw \\ Marshall School of Business, Data Sciences and Operations  \\  University of Southern California
\\ 307A Bridge Hall, Los Angeles, CA 90089 \\  andrew.daw@usc.edu  \\
\and
  Jamol Pender \\ School of Operations Research and Information Engineering \\ Cornell University
\\ 228 Rhodes Hall, Ithaca, NY 14853 \\  jjp274@cornell.edu  \\
 }

\maketitle

\abstract{Across a wide variety of applications, the self-exciting Hawkes process has been used to model phenomena in which the history of events influences future occurrences. However, there may be many situations in which the past events only influence the future as long as they remain active. For example, a person spreads a contagious disease only as long as they are contagious. In this paper, we define a novel generalization of the Hawkes process that we call the \textit{ephemerally self-exciting process}. In this new stochastic process, the excitement from one arrival lasts for a randomly drawn activity duration, hence the ephemerality.
Our study includes exploration of the process itself as well as connections to well-known stochastic models such as branching processes, random walks, epidemics, preferential attachment, and Bayesian mixture models. Furthermore, we prove a batch scaling construction of general, marked Hawkes processes from a general ephemerally self-exciting model, and this novel limit theorem both provides insight into the Hawkes process and motivates the model contained herein as an attractive self-exciting process in its own right.}\\


\section{Introduction}\label{Intro}


\textit{What's past is prologue} -- unavoidably, the present is shaped by what has already occurred. The current state of the world is indebted to our history. Our actions, behaviors, and decisions are both precursory and prescriptive to those that follow, and this can be observed across a variety of different scenarios. For example, the spread of an infectious disease is accelerated as more people become sick and dampened as they recover. In finance, a flurry of recent transactions can prompt new buyers or sellers to enter a market. On social media platforms, as more and more users interact with a post it  can become trending or viral and thus be broadcast to an even larger audience.


Self-exciting processes are an intriguing family of stochastic models in which the history of events influences the future.  \citet{hawkes1971spectra} introduced the concept of self-excitement -- defining what is now known as the Hawkes process, a model in which ``the current intensity of events is determined by events in the past.''  That is, the Hawkes process is a stochastic intensity point process that depends on the history of the point process itself. The rate of new event occurrences increases as each event occurs. As time passes between occurrences, the intensity is governed by a deterministic excitement kernel. Most often, this kernel is specified so that the intensity jumps upward at event epochs and strictly decreases in the interim. In this way, occurrences beget occurrences; hence the term ``self-exciting.'' Unlike the Poisson process, disjoint increments are not independent in sample paths of Hawkes process. Instead, they are positively correlated and, by definition, the events of the former influence the events of the latter. Furthermore, the Hawkes process is known to be over-dispersed -- meaning that its variance is larger than its mean -- which is commonly found in real world data, whereas the Poisson process has equal mean and variance.

Because of the practical relevance of these model features, self-exciting processes have been used in a wide variety of applications, many of which are quite recent additions to the literature. Seismology was among the first domains to incorporate these models, such as in \citet{ogata1988statistical}, as the occurrence of an earthquake increases the risk of subsequent seismic activity in the form of aftershocks. Finance has since followed as a popular application and is now perhaps the most prolific area of work. In these studies, self-excitement is used to capture the often contagious nature of financial activity, see e.g. \citet{errais2010affine,bacry2013some,bacry2014hawkes,da2014hawkes,ait2015modeling,azizpour2016exploring,rambaldi2017role,gao2018transform,wu2019queue}. Similarly, there have been many recent internet and social media scenarios that have been modeled using self-exciting processes, drawing upon the virality of modern web traffic. For example, see \citet{farajtabar2017fake,rizoiu2017expecting,rizoiu2018sir}. Notably, this also includes use of Hawkes processes for constructing data-driven methods in the artificial intelligence and machine learning literatures, such as \citet{du2015dirichlet,mei2017neural,xu2017learning}. In an intriguing area of work, self-excitement has also been used to model inter-personal communication; for example in application to conversation audio recordings in \citet{masuda2013self} or in studying email correspondence in \citet{malmgren2008poissonian,halpin2013modelling}. Hawkes processes have also recently been used to represent arrivals to service systems in queueing models, e.g. in \citet{gao2018functional,gao2018large,koops2017infinite,daw2017queues}. This is of course not an exhaustive list of works in these areas, nor is it a complete account of all the modern applications of self-excitement. Examples of other notable uses include neuroscience \cite{truccolo2005point,krumin2010correlation}, environmental management \cite{gupta2018discrete}, public health \cite{zino2018modeling}, movie trailer generation \cite{xu2015trailer}, energy conservation \cite{li2018energy}, and industrial preventative maintenance \cite{yan2013towards}.

As the variety of uses for self-excitement has continued to grow, the number of Hawkes process generalizations has kept pace. By modifying the definition of the Hawkes process in some way, the works in this generalized self-exciting process literature provide new perspectives on these concepts while also empowering and enriching applications. For example, \citet{bremaud1996stability} introduce a non-linear Hawkes process that adapts the definition of the process intensity to feature a general, non-negative function of the integration over the process history, as opposed to the linear form given originally. Similarly, the quadratic Hawkes process model given by \citet{blanc2017quadratic} allows for excitation kernels that have quadratic dependence on the process history, rather than simply linear. This is also an example of a generalization motivated by application, as the authors seek to capture time reversal asymmetry observed in financial data. As another finance-motivated generalization, \citet{dassios2011dynamic} propose the dynamic contagion process. This model can be thought of as a hybrid between a Hawkes process and a shot-noise process, as the stochastic intensity of the model features both self-excited and externally excited jumps. The authors take motivation from an application in credit risk, in which the dynamics are shaped by both the process history and by exogenous shocks. The affine point processes studied in e.g. \citet{errais2010affine,zhang2009rare,zhang2015affine} are also motivated by credit risk applications. The models in these works combine the self-exciting dynamics of Hawkes process with those of an affine jump-diffusion process, imbedding modeling concepts of feedback and dependency into the process intensity. An exact simulation procedure for the Hawkes process with CIR intensity, a generalization of the Hawkes process that is a special case of the affine point process, is shown in \citet{dassios2017efficient}. In that case, the authors discuss an application to portfolio loss processes.

There have also been several Hawkes process generalizations proposed in social media and data analytics contexts. For example, \citet{rizoiu2018sir} introduces a finite population Hawkes process that couples self-excitement dynamics with those of the Susceptible-Infected-Recovered (SIR)  process. Drawing upon the use of the SIR process for the spread of both disease and ideas, the authors propose this SIR-Hawkes process as a method of studying information cascades. Similarly, \citet{mei2017neural} introduce the neural Hawkes process as a new point process model in the machine learning literature. As the name suggests, this model combines self-excitement with concepts from neural networks. Specifically, a recurrent neural network effectively replaces the excitation kernel, governing the effect of the past events on the rate of future occurrences. In the literature for Bayesian nonparametric models, \citet{du2015dirichlet} present the Dirichlet-Hawkes process for topic clustering in document streams. In this case, the authors combine a Hawkes process and a Dirichlet process, so that the intensity of the stream of new documents is self-exciting while the type of each new document is determined by the Dirichlet process, leading to a preferential attachment structure among the document types.

In this paper, we propose the \textit{ephemerally self-exciting process}, a novel generalization of the Hawkes process. 
Rather than regulating the excitement through the gradual, deterministic decay provided by an excitation kernel function, we instead incorporate randomly timed down-jumps. We will refer to this random length of time as the arrival's activity duration. The down-jumps are equal in size to the up-jumps, and between events the arrival rate does not change. Thus, this process increases in arrival rate upon each occurrence, and these increases are then mirrored some time later by decreases in the arrival rate once the activity duration expires.  In this way, the self-excitement is ephemeral: it is only in effect as long as the excitement is active. Much of the body of this work will discuss how this ephemeral, piece-wise constant self-excitement compares to the eternal but ever-decaying notion from Hawkes's original definition. As we will see in our analysis, this new process is both a promising model of self-excitement and an explanation of its origins in natural phenomena.


\subsection{Practical Relevance}

While this paper will not be focused on any one application, in this subsection we summarize several domain areas in which the models in this work can be applied. A natural example is in public health and the management of epidemics. For example, consider influenza. When a person becomes sick with the flu, she increases the rate of spread of the virus through her contact with others. This creates a self-exciting dynamic of the spread of the virus. However, a person only spreads a disease as long as she is contagious; once she has recovered she no longer has a direct effect on the rate of new infections. From a system-level perspective, the ephemerally self-exciting process can thus be thought of as modeling the arrivals of new infections, capturing the self-exciting and ephemeral nature of sick patients.  This motivates the use of this model as an arrival process to queueing models for healthcare, as the rate of \textit{arrivals to} clinics serving patients with infectious diseases should depend on the number of people currently infected. The health-care service can also be separately modeled, as an infinite server queue may be a fitting representation for the number of infected individuals but the clinic itself likely has limited capacity. This concept of course extends to the modeling and management of any other viral disease, including the novel coronavirus that has caused the COVID-19 pandemic. 

Of course, epidemic models need not be exclusively applied to disease spread. These same ideas can be used for information spread and product adoption, such as in the aforementioned Hawkes-infused  models in \citet{rizoiu2018sir} and \citet{zino2018modeling}. In these contexts, one can think of the duration in system as being the time a person actively promotes a concept or product. A single person only affects the self-excitement of the idea or product spread as long as she is in the system, which distinguishes this model from those in the aforementioned works. Epidemic models have also been used to study social issues, such as the contagious nature of imprisonment demonstrated by \citet{lum2014contagious}. We discuss the relevance of ephemeral self-excitement for epidemics in detail in Subsection~\ref{epidemics} by relating this model to the Susceptible-Infected-Susceptible (SIS) process through a convergence in distribution. In fact, throughout Section~\ref{secRelate} we establish connections from this process to other relevant stochastic models. This includes classical processes such as branching processes and random walks, as well as models popular both in Bayesian nonparametrics and in preferential attachment settings, such as the Dirichlet process and the Chinese restaurant process.

In the context of service systems, self-excitement can be motivated by the same rationale that inspires restaurants to seat customers at the tables by the windows. Potential new customers could choose to dine at the establishment because they can see others already eating there, taking an implicit recommendation from those already being served. This same example also motivates the ephemerality. After a customer seated by the window finishes her dinner and departs, any passing potential patron only sees an empty table; the implicit recommendation vanishes with the departing customer. A similar dynamic can be observed in online streaming platforms. For example on popular music streaming services like Spotify and Apple Music, users can see what songs and albums have been recently played by their friends.  If a user sees that many of her friends have listened to the same album recently, she may be more inclined to listen to it as well. However, this applies only as long as the word ``recently'' does. If her friends don't play the album within a certain amount of time,  the platform will no longer promote the album to her in that fashion. Again, this displays the ephemerality of the underlying self-excitement: the album grows more attractive as more users listen to it, but only as long as those listens are ``recent'' enough.

In finance, limit order books (LOB's) are among the many concepts that have been modeled using Hawkes process, such as in \citet{rambaldi2017role,bacry2016estimation}. LOB's have also been studied through queueing models, where one can model the state of the LOB (or, more specifically, the number of unresolved bids and asks) as the length of a queueing process. Moreover, there has been recent work that models this process as not just a queue, but a queue with Hawkes process arrivals; for example see \citet{guo2015dynamics,gao2018functional}. Conceptually, the self-excitement may arise from traders reacting to the activity of other traders, creating runs of transactions. However, the desire to not act on stale information may mean that this excitement only lasts as long as trades are actively being conducted. In fact, the idea of the self-excitement in LOB models being ``queue-reactive'' has just very recently been considered by \citet{wu2019queue}, a related work to this one.

One can also consider failures in a mechanical system as an application of this model. For example, consider a network of water pipes. When one pipe breaks or bursts, it can place stress on the pipes connected to it. This stress may then cause further failures within the pipe network. However once the pipe is  properly repaired it should no longer place strain on its surrounding components. Thus, the increase in pipe failure rate caused by a failure is only in effect until the repair occurs, inducing ephemeral self-excitement. The self-excitement (albeit without the ephemerality) arising in this scenario was modeled using Hawkes processes in \citet{yan2013towards}, which includes an empirical study. A similar problem for electrical systems is considered in \citet{ertekin2015reactive}. The reactive point process considered in that work is perhaps the model most similar to the ones studied herein, as the rate of new power failures both increases at the prior failure times and decreases upon inspection or repair. However, a key difference is that in \cite{ertekin2015reactive}, the authors treat the inspection times as controlled by management, whereas in this paper the model is fully stochastic and thus the repair durations are random. Regardless, that work is an excellent example of how generalized self-exciting processes  can be used to shape practical policy. Because power outages have significant and wide-reaching consequences, it is critical to understand the inter-dependency between these events and to study the resulting ephemerally self-exciting process that arises in these electrical grid failures.



\subsection{Organization and Contributions of Paper}

%


Let us now detail the remainder of this paper's organization, as well as the contributions therein.
\begin{itemize}
\item In Section~\ref{secModelDef}, we define the ephemerally self-exciting process (ESEP), a Markovian model for self-excitement that lasts for only a finite amount of time. After defining the model, we develop fundamental distributional quantities and compare the ESEP to the Hawkes process.
\item In Section~\ref{secRelate}, we relate the ESEP to many other important and well-known stochastic processes. This includes branching processes, which gives us further comparisons between the Hawkes process and the ESEP, models for preferential attachment and Bayesian statistics, and epidemic models. The lattermost of these motivates the ESEP as a representation for the times of infection within an epidemic, and this also provides a formal link between the conceptually similar concepts of epidemics and self-excitement.
\item In Section~\ref{secGESEP}, we broaden our exploration of ephemeral self-excitement to non-Markovian models with general activity durations and batches of arrivals. In this general setting, we establish a limit theorem providing an alternate construction of general Hawkes processes. This batch scaling limit thus yields intuition for the observed occurrence of self-excitement in natural phenomena and stands as a fundamental cornerstone for studying such processes.
\end{itemize}
In addition to these main avenues of study, we also have extensive auxiliary analysis housed in this paper's appendix. Appendix~\ref{secLemmas} contains lemmas and side results that support our analysis but are outside the main narrative. In Appendix~\ref{secHESEP}, we explore a model that is a hybrid between the ESEP and the Hawkes process, in that it regulates the excitement with both down-jumps and decay. Appendix~\ref{subsecBlock} is devoted to a finite capacity version of the ESEP, in which arrivals that would put the active number in system above the capacity are blocked from occurring. Finally, Appendix~\ref{pgfProof} contains an algebraically cumbersome proof of a result from Section~\ref{secModelDef}.

\section{Modeling Ephemeral Self-Excitement}\label{secModelDef}

We begin this paper by defining our ephemerally self-exciting model and conducting an initial analysis of some fundamental quantities. These quantities include the transient moment generating function and the steady-state distribution. Before doing so though, let us first review the Hawkes process, which is the original self-exciting probability model.

\subsection{Preliminary Models and Concepts}\label{subsecHawkes}

Introduced and pioneered through the series of papers \citet{hawkes1971spectra,hawkes1971point,hawkes_oakes_1974}, the Hawkes process is a stochastic intensity point process in which the current rate of arrivals is dependent on the history of arrival process itself. Formally, this is defined as follows: let $(\lambda_t, N_{t,\lambda})$ be an intensity and counting process pair such that
\begin{align*}
\PP{N_{t+\Delta,\lambda} - N_{t,\lambda}  = 1 \mid \mathcal{F}^N_t} &= \lambda_t \Delta + o(\Delta), \\
\PP{N_{t+\Delta,\lambda} - N_{t,\lambda}  > 1 \mid \mathcal{F}^N_t} &= o(\Delta), \\
\PP{N_{t+\Delta,\lambda} - N_{t,\lambda}  = 0 \mid \mathcal{F}^N_t} &= 1 - \lambda_t \Delta + o(\Delta),
\end{align*}
where $\mathcal{F}^N_t$ is the filtration of $N_{t,\lambda}$ up to time $t$ and $\lambda_t$ is given by
$$
\lambda_t = \lambda^* + \int_{-\infty}^t g(t-u) \mathrm{d}N_{u,\lambda},
$$
where $\lambda^* > 0$ and $g : \mathbb{R}^+ \to \mathbb{R}^+$ is such that $\int_0^\infty g(x) \mathrm{d}x < 1$. Through this definition, the intensity $\lambda_t$ captures the history of the arrival process up to time $t$. Thus,  $\lambda_t$ encapsulates the sequence of past events and uses it to determine the rate of future occurrences.  We refer to $\lambda^*$ as the baseline intensity and $g(\cdot)$ as the excitation kernel. The baseline intensity represents an underlying stationary arrival rate and the excitation kernel governs the effect that the history of the process has on the current intensity. A common modeling choice is to set $g(x) = \alpha e^{-\beta x}$, where $\beta > \alpha > 0$. This is often referred to as the ``exponential'' kernel and it is perhaps the most widely used form of the Hawkes process. In this case, $(\lambda_t, N_{t,\lambda})$ is a Markov process obeying the stochastic differential equation
$$
\mathrm{d}\lambda_t = \beta(\lambda^* - \lambda_t)\mathrm{d}t + \alpha \mathrm{d}N_{t,\lambda}
.
$$
That is, at arrival epochs $\lambda_t$ jumps upward by amount $\alpha$ and the $N_{t,\lambda}$ increases by 1; between arrivals $\lambda_t$ decays exponentially at rate $\beta$ towards the baseline intensity $\lambda^*$. Thus, each arrival increases the likelihood of additional arrivals occurring soon afterwards -- hence, it self-excites. This form of the Hawkes process is also often alternatively stated with an initial value for $\lambda_t$, say $\lambda_0 \geq \lambda^*$. In this case, the intensity can be expressed
$$
\lambda_t
=
\lambda^*
+
(\lambda_0 - \lambda^*)e^{-\beta t}
+
\alpha\int_{0}^t e^{-\beta (t-u)} \mathrm{d}N_{u,\lambda}
.
$$
Additional overview of the Hawkes process with the exponential kernel can be found in Section 2 of \cite{daw2017queues}. Another common choice for excitation kernel is the ``power-law'' kernel $g(x) = \frac{k}{(c+x)^p}$, where $k > 0$, $c > 0$, and $p > 0$. This kernel was originally popularized in seismology \cite{ogata1988statistical}.

\subsection{Defining the Ephemerally Self-Exciting Process}

As we have discussed in the introduction, a plethora of natural phenomena exhibit self-exciting features but only for a finite amount of time. This prompts the notion of ephemeral self-excitement. By comparison to the traditional Hawkes process we have reviewed in Subsection~\ref{subsecHawkes}, we seek a model in which a new occurrence increases the arrival rate only so long as the newly entered entity remains active in the system.  Thus, we now define the \textit{ephemerally self-exciting process} (ESEP), which trades the Hawkes process's eternal decay for randomly drawn expiration times. Moreover, in the following Markovian model, exponential decay is replaced with exponentially distributed durations. In Section~\ref{secGESEP}, we extend these concepts to generally distributed service. As another generalization, in Appendix~\ref{secHESEP} we consider a Markovian model with both decay and down-jumps. For now, we explore the effects of ephemerality through the ESEP model in Definition~\ref{esepDef}.

\begin{definition}[Ephemerally self-exciting process]\label{esepDef}
For times $t \geq 0$, a baseline intensity $\eta^* > 0$, intensity jump size $\alpha > 0$, and expiration rate $\beta > 0$, let $N_t$ be a  counting process with stochastic intensity $\eta_t$ such that
\begin{align}
\eta_t = \eta^* + \alpha Q_t
,
\end{align}
where $Q_t$ is incremented with $N_t$ and then is depleted at unit down-jumps according to the rate $\beta Q_t$. Then, we say that $(\eta_t, N_t)$ is an \textit{ephemerally self-exciting process} (ESEP).
\end{definition}

We will assume that $\eta_0$ and $Q_0$ are known initial values such that $\eta_0 = \eta^* + \alpha Q_0$. In addition to this definition, one could also describe the ESEP through its dynamics. In particular, the behavior of this process can be summarily cast through the life cycle of its arrivals:
\begin{enumerate}[i)]
\item At each arrival, the arrival rate $\eta_t$ increases by $\alpha$.
\item Each arrival remains active for an activity duration drawn from an i.i.d.~sequence of exponential random variables with rate $\beta$.
\item At the expiration of a given activity duration, $\eta_t$ decreases by $\alpha$.
\end{enumerate}
The ephemerality of the ESEP is embodied by this cycle. Because arrivals only contribute to the intensity for the length of their activity duration, their effect on the process's excitation vanishes when this clock expires. Furthermore, there is an affine relationship between the number of active ``exciters'' -- meaning unexpired arrivals still causing excitation -- and the intensity, i.e.~$\eta_t = \eta^* + \alpha Q_t$. Thus, we could also track the arrival rate through $Q_t$ in place of $\eta_t$ and still have full understanding of this process. This also means that results are readily transferrable between these two processes; we will often make use of this fact.

Because the ESEP is quite parsimonious, there are many alternative perspectives we could take to gain additional understanding of it. For example, one could consider $Q_t$ a Markovian queueing system with infinitely many servers and a state dependent arrival rate. Equivalently, one could also describe the ESEP as a Markov chain on the non-negative integers where transitions at state $i$ are to $i+1$ at rate $\eta^* + \alpha i$ and to $i-1$ at rate $\mu i$, with the counting process then defined as the epochs of the upward jumps in this chain. This Markov chain perspective certifies the existence and uniqueness of Definition~\ref{esepDef}. A visualization of this linear  birth-death-immigration process is given in Figure~\ref{ESEPchain}. Stability for this chain occurs when $\beta > \alpha$; we will assume this hereforward although it of course is not necessary for transient results. One could also view the ESEP as a generalization of Hawkes' original definition where the excitation kernel function $g(\cdot)$ is replaced with a randomly drawn indicator function that is different for each arrival. Each indicator function compares time to an independently drawn exponential random variable, and this perspective is closely aligned with our analysis in Section~\ref{secRelate}. 
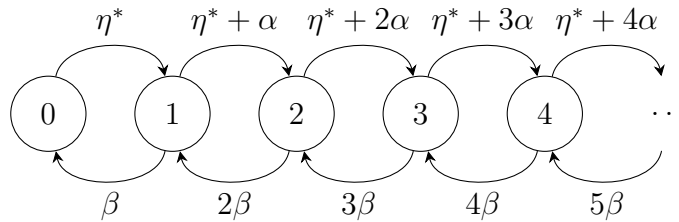
\begin{figure}[h]
\begin{center}
\begin{tikzpicture}[scale=.825,line cap=round,line join=round,>={Stealth[width=1.5mm,length=1.5mm]},x=1.0cm,y=1.0cm]

 \node (zero) [draw, circle, minimum size=1cm] at (-2,1cm) {$0$};
 \node (one) [draw, circle, minimum size=1cm] at (0,1cm) {$1$};
 \node (two) [draw, circle, minimum size=1cm] at (2,1cm) {$2$};
 \node (three) [draw, circle, minimum size=1cm] at (4,1cm) {$3$};
  \node (four) [draw, circle, minimum size=1cm] at (6,1cm) {$4$};
 \node (dots)  [draw, circle, minimum size=1cm,color=white] at (8,1cm) {$\mathbf{\textcolor{black}{\dots}}$};

\draw [->] (zero) .. controls +(.25,1.25) and +(-.25,1.25).. node [midway, above] {$\eta^*$} (one);
 \draw [->] (one) .. controls +(.25,1.25) and +(-.25,1.25).. node [midway, above] {$\eta^* + \alpha$} (two);
 \draw [->] (two) .. controls +(.25,1.25) and +(-.25,1.25) ..  node [midway, above] {$\eta^* + 2\alpha$}  (three);
 \draw [->] (three) .. controls +(.25,1.25) and +(-.25,1.25) ..  node [midway, above] {$\eta^* + 3 \alpha$}  (four);
  \draw [->] (four) .. controls +(.25,1.25) and +(-.25,1.25) ..  node [midway, above] {$\eta^* + 4 \alpha$}  (dots);

 \draw [->] (dots) .. controls +(-.25,-1.25) and +(.25,-1.25) .. node [midway, below] {$5 \beta$} (four);
  \draw [->] (four) .. controls +(-.25,-1.25) and +(.25,-1.25) .. node [midway, below] {$4 \beta$} (three);
 \draw [->] (three) .. controls +(-.25,-1.25) and +(.25,-1.25) .. node [midway, below] {$3 \beta$} (two);
 \draw [->] (two) .. controls +(-.25,-1.25) and +(.25,-1.25) .. node [midway, below] {$2 \beta$} (one);
 \draw [->] (one) .. controls +(-.25,-1.25) and +(.25,-1.25) .. node [midway, below] {$\beta$} (zero);
\end{tikzpicture}
\end{center}
\vspace{-.2in}
\caption{The transition diagram of the Markov chain for $Q_t$.}\label{ESEPchain}
\end{figure}

In the remainder of this subsection, let us now develop a few fundamental quantities for this stochastic process, particularly its intensity and active number in system as these capture the self-exciting behavior of the process. First, in Proposition~\ref{AQHtransmgf} we compute the transient moment generating function for the intensity $\eta_t$. As we have noted, this can also be used to immediately derive the same transform for $Q_t$, and the proof makes use of this fact.

\begin{proposition}\label{AQHtransmgf}
Let $\eta_t = \eta^* + \alpha Q_t$ be the intensity of an ESEP with baseline intensity $\eta^* > 0$, intensity jump $\alpha > 0$, and expiration rate $\beta > \alpha$. Then, the moment generating function  for $\eta_t$ is given by
\begin{align*}
\E{e^{\theta \eta_t}}
&=
\left(
\frac{\beta - \alpha e^{\alpha\theta} - \beta (1-e^{\alpha\theta})e^{-(\beta - \alpha)t}}
{\beta - \alpha e^{\alpha\theta} - \alpha (1-e^{\alpha\theta})e^{-(\beta - \alpha)t}}
\right)^{\frac{\eta_0 - \eta^*}{\alpha}}
\\
&\qquad \cdot
\left(
\frac{\beta e^{\alpha \theta}}{\beta - \alpha e^{\alpha\theta}}
-
\frac{\alpha e^{\alpha \theta}}{\beta - \alpha e^{\alpha\theta}}
\left(
\frac{\beta - \alpha e^{\alpha\theta} - \beta (1-e^{\alpha\theta})e^{-(\beta - \alpha)t}}
{\beta - \alpha e^{\alpha\theta} - \alpha (1-e^{\alpha\theta})e^{-(\beta - \alpha)t}}
\right)
\right)^{\frac{\eta^*}{\alpha}}
,
\end{align*}
for all $t \geq 0$ and $\theta < \frac{1}{\alpha}\log\left(\frac{\beta}{\alpha}\right)$.
\begin{proof}
We will approach this through the perspective of the active number in system, $Q_t$. Using Lemma~\ref{fubinifundQ}, we have that the probability generating function for $Q_t$, say $\mathcal{P}(z,t) = \E{z^{Q_t}}$ for $z \in [0,1]$, is given by the solution to the following partial differential equation:
\begin{align*}
\frac{\partial}{\partial t}\E{z^{Q_t}}
&
=
\E{\left(\eta^* + \alpha Q_t\right)\left(z^2 - z\right)z^{Q_t - 1} + \beta Q_t \left(1 - z\right)z^{Q_t - 1}}
,
\intertext{which is equivalently expressed}
\frac{\partial}{\partial t}\mathcal{P}(z,t)
&
=
\eta^* \left(z - 1\right)\mathcal{P}(z,t)
+
\left(
\alpha \left(z^2 - z\right)
+
\beta \left(1 - z\right)
\right)
\frac{\partial}{\partial z}\mathcal{P}(z,t)
,
\end{align*}
with initial condition $\mathcal{P}(z,0) = z^{Q_0}$. The solution to this initial value problem is given by
\begin{align*}
\mathcal{P}(z,t)
&
=
\left(
\frac{\beta - \alpha z - \beta (1-z)e^{-(\beta - \alpha)t}}
{\beta - \alpha z - \alpha (1-z)e^{-(\beta - \alpha)t}}
\right)^{Q_0}
\left(
\frac{\beta}{\beta - \alpha z}
-
\frac{\alpha}{\beta - \alpha z}
\left(
\frac{\beta - \alpha z - \beta (1-z)e^{-(\beta - \alpha)t}}
{\beta - \alpha z - \alpha (1-z)e^{-(\beta - \alpha)t}}
\right)
\right)^{\frac{\eta^*}{\alpha}}
,
\end{align*}
yielding the probability generating function for $Q_t$. By setting $z = e^{\theta}$ we receive the moment generating function. Finally, using the affine relationship $\eta_t = \eta^* + \alpha Q_t$, we have that
$$
\E{e^{\theta \eta_t}}
=
\E{e^{\theta \left(\eta^* + \alpha Q_t\right)}}
=
e^{\theta \eta^*}
\E{e^{\alpha \theta  Q_t}}
,
$$
with $\eta_0 = \eta^* + \alpha Q_0$.
\end{proof}
\end{proposition}



As we have mentioned, this Markov chain can be shown to be stable for $\beta > \alpha$ through standard techniques. Thus, using the moment generating function from Proposition~\ref{AQHtransmgf}, we can find the steady-state distributions of the intensity and the active number in system by taking the limit of $t$. We can quickly observe that this leads to a negative binomial distribution, as we state in Theorem~\ref{negbinthm}. Because of the varying definitions of the negative binomial distribution, we state the probability mass function explicitly.

\begin{theorem}\label{negbinthm}
Let $\eta_t = \eta^* + \alpha Q_t$ be an ESEP with baseline intensity $\eta^* > 0$, intensity jump $\alpha > 0$, and expiration rate $\beta > \alpha$. Then, the active number in system in steady-state follows a negative binomial distribution with parameters $\frac{\alpha}{\beta}$ and $\frac{\eta^*}{\alpha}$, which is to say that the steady-state probability mass function is
\begin{align}
\PP{Q_\infty = k}
&=
\frac{\Gamma\left(k + \frac{\eta^*}{\alpha}\right)}{\Gamma\left(\frac{\eta^*}{\alpha}\right) k! }
\left(\frac{\beta - \alpha}{\beta}\right)^{\frac{\eta^*}{\alpha}}
\left(\frac{\alpha}{\beta}\right)^k
.
\end{align}
Consequently, the steady-state distribution of the intensity is given by a shifted and scaled negative binomial with parameters $\frac{\alpha}{\beta}$ and $\frac{\eta^*}{\alpha}$, shifted by $\eta^*$ and scaled by $\alpha$.
\begin{proof}
Using Proposition~\ref{AQHtransmgf}, we can see that the steady-state moment generating function of $Q_t$ is given by
$$
\lim_{t\to\infty} \E{e^{\theta Q_t}}
=
\left(\frac{\beta-\alpha}{\beta - \alpha e^{\theta}}\right)^{\frac{\eta^*}{\alpha}}
.
$$
We can observe that this steady-state moment generating function is equivalent to that of a negative binomial. By the affine transformation $\eta_t = \eta^* + \alpha Q_t$, we find the steady-state distribution for the intensity.
\end{proof}
\end{theorem}

The negative binomial distribution presented in Theorem~\ref{negbinthm} allows for $\frac{\eta^*}{\alpha}$ to take on any positive real value; integrality is not required. If $\frac{\eta^*}{\alpha}$ is in fact an integer, then the gamma functions will match the corresponding factorial functions and this ratio of factorials will simplify to a binomial coefficient, reproducing the most familiar form of the negative binomial distribution. In this case, $\frac{\eta^*}{\alpha}$ can be interpreted as the number of failures upon which an experiment is stopped, where the success probability is $\frac{\alpha}{\beta}$ and $k + \frac{\eta^*}{\alpha}$ is the total number of trials.

Let us pause to note that this explicit characterization of the steady-state intensity is already an advantage of the ESEP over the traditional Markovian Hawkes process, for which there is not a closed form intensity stationary distribution available. As a consequence of Theorem~\ref{negbinthm}, we can observe that the steady-state mean of the intensity is $\eta_\infty := \frac{\beta \eta^*}{\beta - \alpha}$. Interestingly, this would also be the steady-state mean of the Hawkes process when given the same baseline intensity, the same intensity jump size, and an exponential decay rate equal to the rate of expiration. This leads us to ponder how the processes would otherwise compare when given equivalent parameters. In Proposition~\ref{momentOrder} we find that although this equivalence of means extends to transient settings, for all higher moments the ESEP dominates the Hawkes process in terms of both the intensity and the counting process.

\begin{proposition}\label{momentOrder}
Let $(\eta_t, N_{t,\eta})$ be an ESEP intensity and counting process pair with jump size $\alpha > 0$, expiration rate $\beta > \alpha$, and baseline intensity $\eta^* > 0$. Similarly, let $(\lambda_t, N_{t,\lambda})$ be a Hawkes process intensity and counting process pair with jump size $\alpha > 0$, decay rate $\beta > 0$, and baseline intensity $\eta^* > 0$. Then, if the two processes have equal initial values, their means will satisfy
\begin{align}
&\E{\lambda_t} = \E{\eta_t},
&&\E{N_{t,\lambda}} = \E{N_{t,\eta}}
,
\end{align}
and for $m \geq 2$ their $m^\text{th}$ moments are ordered such that
\begin{align}
&\E{\lambda_t^m} \leq \E{\eta_t^m},
&&\E{N_{t,\lambda}^m} \leq \E{N_{t,\eta}^m}
,
\end{align}
for all time $t \geq 0$.
\begin{proof}
Let us start with the means. For the intensities, we can note that these are given by the solutions to
\begin{align*}
\frac{\mathrm{d}}{\mathrm{d}t}\E{\lambda_t}
=
\alpha \E{\lambda_t} - \beta(\E{\lambda_t} - \eta^*)
\qquad
\text{and}
\qquad
\frac{\mathrm{d}}{\mathrm{d}t}\E{\eta_t}
=
\alpha\E{\eta_t} - \beta \alpha\left( \frac{\E{\eta_t} - \eta^*}{\alpha} \right)
,
\end{align*}
and through simplification one can quickly observe that these two ODE's are equivalent. Thus, because we have assumed that the processes have the same initial values, we find that $\E{\lambda_t} = \E{\eta_t}$. Since
$$
\frac{\mathrm{d}}{\mathrm{d}t}\E{N_{t,\lambda}} = \E{\lambda_t}
\qquad
\text{and}
\qquad
\frac{\mathrm{d}}{\mathrm{d}t}\E{N_{t,\eta}} = \E{\eta_t}
,
$$
this equality immediately extends to the means of the counting processes as well. We will now use these equations for the means as the base cases for inductive arguments, beginning again with the intensities. For the inductive step, we will assume that the intensity moment ordering holds for moments 1 to $m-1$. The $m^\text{th}$ moment of the Hawkes process intensity is thus given by the solution to
\begin{align*}
\frac{\mathrm{d}}{\mathrm{d}t}\E{\lambda_t^m}
&=
\sum_{k=0}^{m-1} {m \choose k} \E{\lambda_t^{k+1}} \alpha^{m-k}
-
m \beta \E{\lambda_t^m}
+
m \beta \eta^* \E{\lambda_t^{m-1}}
:=
f_\lambda(t, \E{\lambda_t^{m}})
,
\end{align*}
where $f_\lambda(t, \E{\lambda_t^{m}})$ is meant to capture that this ODE depends on the value of the $m^\text{th}$ moment and of the lower moments, which by the inductive hypothesis we take as known functions of the time $t$. Then, the $m^\text{th}$ moment of the ESEP intensity will satisfy
\begin{align*}
\frac{\mathrm{d}}{\mathrm{d}t}\E{\eta_t^m}
&=
\sum_{k=0}^{m-1} {m \choose k} \E{\eta_t^{k+1}} \alpha^{m-k}
+
\frac{\beta}{\alpha}
\sum_{k=0}^{m-1} {m \choose k} \E{\eta_t^{k+1}} (-\alpha)^{m-k}
-
\frac{\beta \eta^*}{\alpha}
\sum_{k=0}^{m-1} {m \choose k} \E{\eta_t^{k}} (-\alpha)^{m-k}
.
\end{align*}
By pulling the $k = m-1$ terms off the top of each summation, we can re-express this ODE as
\begin{align*}
\frac{\mathrm{d}}{\mathrm{d}t}\E{\eta_t^m}
&=
\sum_{k=0}^{m-1} {m \choose k} \E{\eta_t^{k+1}} \alpha^{m-k}
-
m\beta \E{\eta_t^{m}}
+
m\beta \eta^* \E{\eta_t^{m-1}}
\\
&
\quad
+
\frac{\beta}{\alpha}
\sum_{k=0}^{m-2} {m \choose k}  (-\alpha)^{m-k}
\left(
\E{\eta_t^{k+1}}
-
\eta^* \E{\eta_t^{k}}
\right)
,
\end{align*}
and through the definition of $f_\lambda(\cdot)$ and the inductive hypothesis, we can find the following lower bound:
\begin{align*}
\frac{\mathrm{d}}{\mathrm{d}t}\E{\eta_t^m}
&\geq
f_\lambda(t, \E{\eta_t^m})
+
\frac{\beta}{\alpha}
\sum_{k=0}^{m-2} {m \choose k}  (-\alpha)^{m-k}
\left(
\E{\eta_t^{k+1}}
-
\eta^* \E{\eta_t^{k}}
\right)
.
\end{align*}
This right-most term can then be expressed
\begin{align*}
\frac{\beta}{\alpha}
\sum_{k=0}^{m-2} {m \choose k}  (-\alpha)^{m-k}
\left(
\E{\eta_t^{k+1}}
-
\eta^* \E{\eta_t^{k}}
\right)
=
\frac{\beta}{\alpha}
\E{
(\eta_t - \eta^*)
\left(
(\eta_t - \alpha)^m
-
\eta_t^m
+
m \alpha \eta_t^{m-1}
\right)
}
,
\end{align*}
and we can now reason about the quantity inside the expectation. By definition, we have that $\eta_t \geq \eta^*$ surely, and furthermore we can observe that if $\eta_t - \eta^* > 0$, then $\eta_t \geq \eta^* + \alpha > \alpha$. Thus, let us consider $(\eta_t - \alpha)^m
-
\eta_t^m
+
m \alpha \eta_t^{m-1}$ assuming $\eta_t > \alpha$. Dividing through by $\eta_t^m$, we have the expression
\begin{align}
\left(1 - \frac{\alpha}{\eta_t}\right)^m
-
1
+
\frac{m \alpha}{ \eta_t}
.
\label{innerEq}
\end{align}
Since $(1-x)^m - 1 + mx$ is equal to 0 at $x = 0$ and is non-decreasing on $x \in [0,1)$ via a first-derivative check, we can note that \eqref{innerEq} is non-negative for all $\eta_t > \eta^*$. Thus, we have that
$$
\E{
(\eta_t - \eta^*)
\left(
(\eta_t - \alpha)^m
-
\eta_t^m
+
m \alpha \eta_t^{m-1}
\right)
}
\geq
0
,
$$
and by consequence,
$$
\frac{\mathrm{d}}{\mathrm{d}t}\E{\eta_t^m}
\geq
f_\lambda(t, \E{\eta_t^m})
,
$$
completing the proof of the intensity moment ordering via Lemma~\ref{complemma}. For the counting processes, let us again assume as an inductive hypothesis that the moment ordering holds for moments $1$ through $m-1$, with the mean equality serving as the base case. Then, the $m^\text{th}$ moment of the ESEP counting process will satisfy
\begin{align*}
\frac{\mathrm{d}}{\mathrm{d}t}\E{N_{t,\eta}^m}
&=
\sum_{k=0}^{m-1} {m \choose k} \E{\eta_t N_{t,\eta}^k}
,
\end{align*}
and the ODE for the $m^\text{th}$ Hawkes counting process moment is analogous. By the FKG inequality, we can observe that
$$
\sum_{k=0}^{m-1} {m \choose k} \E{\eta_t N_{t,\eta}^k}
\geq
\sum_{k=0}^{m-1} {m \choose k} \E{\eta_t} \E{N_{t,\eta}^k}
.
$$
By the inductive hypothesis, we have that $\E{N_{t,\eta}^k} \geq \E{N_{t,\lambda}^k}$ for each $k \leq m - 1$, and thus we can observe that
$$
\sum_{k=0}^{m-1} {m \choose k} \E{\eta_t} \E{N_{t,\eta}^k}
\geq
\sum_{k=0}^{m-1} {m \choose k} \E{\lambda_t} \E{N_{t,\lambda}^k}
.
$$
Finally, by another application of Lemma~\ref{complemma}, we have $\E{N_{t,\lambda}^m} \leq \E{N_{t,\eta}^m}$.
\end{proof}
\end{proposition}

The fact that the ESEP variance dominates the Hawkes variance should not be surprising, since the presence of both up and down jumps means that the ESEP sample paths should be subject to more abrupt changes. Nevertheless, this also shows that the ESEP is more  over-dispersed than the Hawkes process is. This may be an attractive feature for data modeling. It is worth noting that matrix computations are available for all moments of these intensities via \cite{daw2019matrix}, through which one could use the method of moments to fit the processes to data.

\subsection{The Ephemerally Self-Exciting Counting Process}\label{subsecAffCount}

Thus far we have studied the intensity of the ESEP, as this process is by definition tracking the self-excitement. However, this excitation is manifested in the actual arrivals from the process, which are counted in $N_t$. We now turn our attention to developing fundamental quantities for this counting process. To begin, we give the probability generating function of the counting process in closed form below in Proposition~\ref{Npgf}. One can note that by comparison, the generating functions of the Hawkes process are instead only expressible as functions of ordinary differential equations with no known closed form solutions, see for example Subsection 3.5 of \cite{daw2017queues}.

\begin{proposition}\label{Npgf}
Let $N_t$ be the number of arrivals by time $t \geq 0$ in an ESEP with baseline intensity $\eta^* > 0$, intensity jump $\alpha > 0$, and expiration rate $\beta > \alpha$. Then,  for $z \in [0,1]$ the probability generating function of $N_t$ is given by
\begin{align}
&
\E{z^{N_t}}
=
e^{\frac{\eta^*(\beta-\alpha)}{2\alpha}t}
\left(
\frac{
2
e^{
\frac{t}{2} \sqrt{(\beta+\alpha)^2 - 4\alpha\beta z}}
}
{
1-\frac{ \beta+\alpha-2\alpha z}{\sqrt{(\beta+\alpha)^2 - 4\alpha\beta z}}
+
\left(
1 + \frac{ \beta+\alpha-2\alpha z}{\sqrt{(\beta+\alpha)^2 - 4\alpha\beta z}}
\right)
e^{
t \sqrt{(\beta+\alpha)^2 - 4\alpha\beta z}}
}
\right)^{\frac{\eta^*}{\alpha}}
\nonumber
\\
&
\quad
\cdot
\left(
\frac{\beta + \alpha}{2\alpha}
+
\frac{\sqrt{(\beta+\alpha)^2 - 4\alpha\beta z}}{2\alpha}
\left(
\frac{
1-\frac{ \beta+\alpha-2\alpha z}{\sqrt{(\beta+\alpha)^2 - 4\alpha\beta z}}
-
\left(
1 + \frac{ \beta+\alpha-2\alpha z}{\sqrt{(\beta+\alpha)^2 - 4\alpha\beta z}}
\right)
e^{
t \sqrt{(\beta+\alpha)^2 - 4\alpha\beta z}}
}
{
1-\frac{ \beta+\alpha-2\alpha z}{\sqrt{(\beta+\alpha)^2 - 4\alpha\beta z}}
+
\left(
1 + \frac{ \beta+\alpha-2\alpha z}{\sqrt{(\beta+\alpha)^2 - 4\alpha\beta z}}
\right)
e^{
t \sqrt{(\beta+\alpha)^2 - 4\alpha\beta z}}
}
\right)
\right)^{Q_0}
\end{align}
where $Q_0$ is the active number in system at time 0.
\begin{proof}
Due to the cumbersome length of some equations, please see Appendix~\ref{pgfProof} for the proof.
\end{proof}
\end{proposition}

In addition to calculating the probability generating function, we can also find a matrix calculation for the transient probability mass function of the counting process. To do so, we recognize that the time until the next arrival occurs can be treated as the time to absorption in a continuous time Markov chain. By building from this idea to construct a transition matrix for several successive arrivals, we find the form for the distribution given in Proposition~\ref{countPMF}.

\begin{proposition}\label{countPMF}
Let $N_t$ be the number of arrivals by time $t$ in an ESEP with baseline intensity $\eta^* > 0$, intensity jump $\alpha > 0$, and expiration rate $\beta > \alpha$. Further, let $Q_0 = k$ be the initial active number in system. Then for $i \in \mathbb{N}$, define the matrices $\mathbf{D}_i \in \mathbb{R}^{k + i + 1 \times k + i + 1}$ and $\mathbf{S}_i \in \mathbb{R}^{k+i+1 \times k + i+2}$ as
\begin{align*}
\mathbf{D}_i
&
=
\begin{bmatrix}
-(\eta^* + (k+i)(\alpha + \beta)) & (k+i)\beta & & &\\
 & -(\eta^*+ (k+i-1)(\alpha + \beta)) & & &\\
 & & \ddots &&\\
 &&&-(\eta^*+ \alpha+\beta)&\beta\\
 &&&&-\eta^*
\end{bmatrix}
,
\end{align*}
and
\begin{align*}
\mathbf{S}_i
&
=
\begin{bmatrix}
\eta^* + \alpha(k+i) & & & & & 0\\
 & \eta^* + \alpha(k+i-1) & & && 0\\
 & & \ddots &&& \vdots\\
 &&&\eta^* + \alpha&& 0\\
 &&&&\eta^* & 0
\end{bmatrix}
.
\end{align*}
Further, let $\mathbf{Z}_n \in \mathbb{R}^{\hat d_n \times \hat d_n}$ for $\hat d_n = \frac{n(n+1)}{2} + (n+1)(k+1)$ be a matrix such that
$$
\mathbf{Z}_n
=
\begin{bmatrix}
\mathbf{D}_0 & \mathbf{S}_0 &&&&\\
& \mathbf{D}_1 & \mathbf{S}_1 &&&\\
&& \ddots & \ddots &&\\
&&& \ddots & \mathbf{S}_{n-2} &\\
&&&& \mathbf{D}_{n-1} & \mathbf{S}_{n-1}\\
&&&&&\mathbf{D}_n
\end{bmatrix}
.
$$
Then, the probability that $N_t = n$ is given by
\begin{align}
\PP{N_t = n}
=
\onevec{1}^\T e^{\mathbf{Z}_n t} \onevec{:}
\end{align}
where $\onevec{j} \in \mathbb{R}^{\hat d_n}$ is the unit column vector for the $j^\text{th}$ coordinate and  $\onevec{:} = \sum_{j = 0}^{k+n} \onevec{\hat d_n - j}$.
\begin{proof}
This follows directly from viewing $\mathbf{Z}_n$ as a sub-matrix of the generator matrix of a continuous time Markov chain (CTMC), much like one can do to calculate probabilities of phase-type distributions. Specifically, the sub-generator matrix is defined on the state space $\mathcal{S} = \bigcup_{i=0}^n \{(0,i), (1,i), \dots, (k+i-1,i),(k+i,i)\}$. In this scenario, the state $(s_1, s_2)$ represents having $s_1$ entities in system and having seen $s_2$ arrivals since time 0. Then, $\mathbf{D}_i$ is the sub-generator matrix for transitions among the sub-state space $\{(k+i,i), (k+i-1,i), \dots, (1,i),(0,i)\}$ to itself (where the states are ordered in that fashion). Similarly, $\mathbf{S}_i$ is for transitions from states in $\{(k+i,i), (k+i-1,i), \dots, (1,i),(0,i)\}$ to states in $\{(k+i+1,i+1), (k+i,i+1), \dots, (1,i+1),(0,i+1)\}$. Then, one can consider this from an absorbing CTMC perspective since if $n+1$ arrivals occur it is not possible to transition back to any state in which $n$ arrivals had occurred. Hence, we only need to use the matrix $\mathbf{Z}_n$ to consider up to $n$ arrivals. Then, $e^{\mathbf{Z}_n t}$ is the sub-matrix for probabilities of transitions among states in $\mathcal{S}$, where the rows will sum to less than 1 as it is possible that the chain has experienced more than $n$ arrivals by time $t$. Finally, because $Q_0 = k$ we know that the chain states in state $(k,0)$; further, because we are seeking the probability that there have been exactly $n$ arrivals by time $t$ we want the probability of transitions from $(k,0)$ to any of the states in $\{(k+n,n), (k+n-1,n), \dots, (1,n),(0,n)\}$.
\end{proof}
\end{proposition}

With these fundamental quantities in hand, let us now turn to explore more nuanced connections between the ESEP and other stochastic processes in the following section. Doing so will provide further comparison between the Hawkes process and the ESEP, and moreover will formally connect the notion of self-excitement to similar concepts such as contagion, virality, and rich-get-richer effects.

\section{Relating Ephemeral Self-Excitement to Branching Processes, Random Walks, and Epidemics}\label{secRelate}

Aside from the original definition, the most frequently utilized result for Hawkes processes is perhaps the immigration-birth representation first shown in \citet{hawkes_oakes_1974}. By viewing a portion of arrivals as immigrants -- externally driven and stemming from a homogenous Poisson process -- and then viewing the remaining portion as offspring -- excitation-driven descendants of the immigrants and the prior offspring -- one can take new perspectives on self-exciting processes. From this position, if an arrival is a descendant then it has a unique parent, the excitement of which spurred this arrival into existence. Every entity has the potential to generate offspring. This viewpoint takes on added meaning in the context of ephemeral self-excitement, as an entity only has the opportunity to generate descendants so long as it remains in the system.  In this section, we will use this idea to connect self-exciting processes to well-known stochastic models that have applications ranging from public health to Bayesian statistics. Furthermore, these connections will also help us form comparisons between the Hawkes process and the model we have introduced, the ESEP. The different dynamics are at the forefront of this process comparison, as the branching structure is dictated by the self-excitation caused by a single arrival. For the Hawkes process, this increase in the arrival rate is eternal but ever-diminishing, whereas in the ESEP the jump is ephemeral but constant when it does exist.

\subsection{Discrete Time Perspectives through Branching Processes}\label{subsecBranch}

Let us first view these processes through a discrete time lens as branching processes. In this subsection we will interpret classical branching processes results in application to these self-exciting processes. Taking the immigration-birth representation as inspiration, we start by considering the distribution of the total number of offspring of a single arrival. That is, we want to calculate the probability mass function for the number of arrivals that are generated directly from the excitement caused by the initial arrival. To constitute the \textit{total} number of offspring, we will consider all the children of this initial entity across all time. For the ESEP, this equates to the number of arrivals generated by the entity throughout its duration in the system; in the Hawkes process this counts the number of arrivals spurred by the entity as time goes to infinity. Given that the stability conditions are satisfied throughout, in Proposition~\ref{offProp} we calculate these distributions by way of inhomogeneous Poisson processes, yielding a Poisson mixture form for each.

\begin{proposition}\label{offProp}
Let $\beta > \alpha > 0$. Let $X^\eta$ be the number of new arrivals generated by the excitement caused by an arbitrary initial arrival throughout its duration in the system in an ESEP with jump size $\alpha$ and expiration rate $\beta$. Then, this offspring distribution is geometrically distributed with probability mass function
\begin{align}
\PP{X^\eta = k} &= \left(\frac{\beta}{\alpha+\beta}\right)\left(\frac{\alpha}{\alpha+\beta}\right)^k.
\end{align}
Similarly, let $X^\lambda$ be the number of new arrivals generated by the excitement caused by an arbitrary initial arrival in a Hawkes process with jump size $\alpha$ and decay rate $\beta$. This offspring distribution is then Poisson distributed with probability mass function
\begin{align}
\PP{X^\lambda = k} &= \frac{e^{-\frac{\alpha}\beta}}{k!}\left(\frac{\alpha}\beta\right)^k.
\end{align}
where all $k \in \mathbb{N}$.
\begin{proof}
Without loss of generality, we assume that the initial arrival in each process occurred at time 0. Then, at time $t \geq 0$ the excitement generated by these initial arrivals has intensities given by $\alpha e^{-\beta t}$ and $\alpha \mathbf{1}\{t < S\}$ for the Hawkes and ESEP processes, respectively, where $S \sim \mathrm{Exp}(\beta)$. Using~\citet{daley2007introduction}, one can note that the offspring distributions across all time can then be expressed as
$$
X^\lambda \sim \mathrm{Pois}\left(\alpha \int_0^\infty e^{-\beta t} \mathrm{d}t\right)
\qquad
\text{and}
\qquad
X^\eta \sim \mathrm{Pois}\left(\alpha \int_0^\infty \mathbf{1}\{t < S\} \mathrm{d}t\right),
$$
which are equivalently stated $X^\lambda \sim \mathrm{Pois}\left(\alpha\slash\beta\right)$ and $X^\eta \sim \mathrm{Pois}\left(\alpha S\right)$. This now immediately yields the stated distributions for $X^\lambda$ and $X^\eta$, as the Poisson-Exponential mixture is known to yield a geometric distribution, see for example the overview of Poisson mixtures in \citet{karlis2005mixed}.
\end{proof}
\end{proposition}

We now move towards considering the total progeny of an initial arrival, meaning the total number of arrivals generated by the excitement of an initial arrival \textit{and} the excitement of its offspring, and of their offspring, and so on across all time. It is important to note that by comparison to the  number of offspring, the  progeny includes the initial arrival itself. As we will see, the stability of the self-exciting processes implies that this total number of descendants is almost surely finite. This demonstrates the necessity of immigration for these processes to survive. 
From the offspring distributions in Proposition~\ref{offProp},  the Hawkes descendant process is a Poisson branching process and, similarly, the ESEP is a geometric branching process. These are well-studied models in branching processes, so we have many results available to us. In fact, we now use a result for random walks with potentially multiple simultaneous steps forward to derive the progeny distributions for these two processes. This is through the well-known hitting time theorem, stated below in Lemma~\ref{hitThm}.

\begin{lemma}[Hitting Time Theorem]\label{hitThm}
The total progeny $Z$ of a branching process with descendant distribution equivalent to $X_1$ is
$$
\PP{Z = k} = \frac{1}{k}\PP{X_1 + X_2 + \dots + X_k = k -1},
$$
where $X_1, \dots, X_k$ are i.i.d. for all $k \in \mathbb{Z}^+$.
\begin{proof}
See \citet{otter1949multiplicative} for the original statement and proof in terms of random walks; a review and elementary proof are given in the brief note \citet{van2008elementary}.
\end{proof}
\end{lemma}

We  now use the hitting time theorem to give the total descendants distributions for the Hawkes process and the ESEP in Proposition~\ref{clusterSize}. This is a common technique for branching processes, and it now yields valuable insight into these two self-exciting models.


\begin{proposition}\label{clusterSize}
Let $\beta > \alpha > 0$. Let $Z^\eta$ be a random variable for the total progeny of an arbitrary arrival in an ESEP with intensity jump $\alpha$ and expiration rate $\beta$. Likewise, let $Z^\lambda$ be a random variable for the total progeny of an arbitrary arrival in a Hawkes process with intensity jump $\alpha$ and decay rate $\beta$.   Then, the probability mass functions for $Z^\eta$ and $Z^\lambda$ are given by
\begin{align}
&
\PP{Z^\eta = k}
=
\frac{1}{k}
{2k-2 \choose k-1}
\left(
\frac{\beta}{\beta + \alpha}
\right)^k
\left(
\frac{\alpha}{\beta + \alpha}
\right)^{k-1}
\text{and}
&\PP{Z^\lambda = k}
=
\frac{e^{-\frac{\alpha}{\beta} k}}{k!}\left(\frac{\alpha k}{\beta}\right)^{k-1}
,
\end{align}
where $k \in \mathbb{Z}^+$.
\begin{proof}
This follows by applying Lemma~\ref{hitThm} to Proposition~\ref{offProp}. Because the sum of independent Poisson random variables is Poisson distributed with the sum of the rates, we have that
$$
\frac{1}{k}\PP{X^\lambda_{1} + X^\lambda_{2} + \dots X^\lambda_{k} = k -1}
=
\frac{1}k \PP{K_1 = k -1}
,
$$
where $K_1 \sim \mathrm{Pois}\left(\frac{\alpha k}{\beta}\right)$. This now yields the expression for the probability mass function for $Z^\lambda$. Similarly for $Z_\eta$ we note that the sum of independent geometric random variables has a negative binomial distribution, which implies that
$$
\frac{1}{k}\PP{X^\eta_{1} + X^\eta_{2} + \dots X^\eta_{k} = k -1}
=
\frac{1}k \PP{K_2 = k -1}
,
$$
where $K_2 \sim \mathrm{NegBin}\left(k,\frac{\alpha}{\beta + \alpha}\right)$, and this completes the proof.
\end{proof}
\end{proposition}

For a visual comparison of the descendants in the ESEP and the Hawkes process, we plot these two progeny distributions for equivalent parameters in Figure~\ref{descend}. As suggested by the variance ordering in Proposition~\ref{momentOrder}, the tail of the ESEP progeny distribution is heavier than that of the Hawkes process.

 \begin{figure}[h]
\begin{center}	
\includegraphics[width=.75\textwidth]{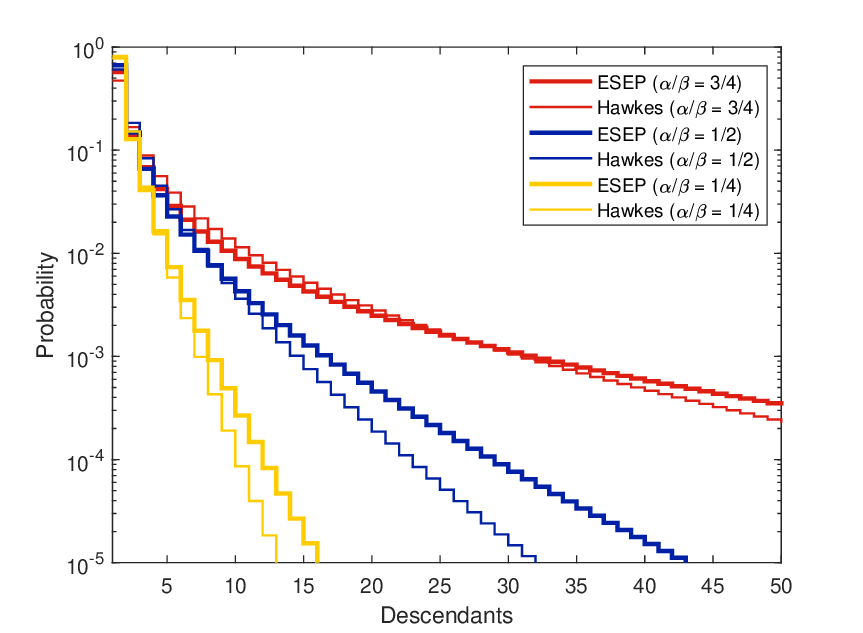}
\caption{Progeny distributions for the ESEP and the Hawkes process with matching parameters.} \label{descend}
\end{center}
\end{figure}

We can note that while one can calculate the mean of each progeny via the probability mass functions in Proposition~\ref{clusterSize}, they can also easily be found using Wald's identity. Through standard infinitesimal generator approaches, one can calculate that the expected number of arrivals (including by immigration) in the ESEP is
\begin{align*}
\E{N_t} = \frac{\beta \eta^* t}{\beta - \alpha} + \frac{\eta_0 - \eta_\infty}{\beta - \alpha}(1-e^{-(\beta - \alpha)t}).
\end{align*}
However, using these branching process representations, we can also express this as
\begin{align*}
\E{N_t} = \E{\sum_{i=1}^{M_t} Z_{i}(t)},
\end{align*}
where $M_t$ is a Poisson process with rate $\eta^*$ and $Z_i(t)$ are the total progeny up to time $t \geq 0$ that descend from the $i^\text{th}$ immigrant arrival. Now, by applying Wald's identity to the limit of $\frac{1}{t}\E{N_t}$ as $t \to \infty$, we see that
$$
\frac{\beta \eta^* }{\beta - \alpha} = \lim_{t \to \infty} \frac{\E{N_t}}{t} = \lim_{t\to\infty}\frac{1}{t}\E{\sum_{i=1}^{M_t} Z_{i}(t)} = \eta^* \E{Z^\eta} ,
$$
and so $\E{Z^\eta} = \frac{\beta}{\beta - \alpha}$. By analogous arguments for the Hawkes process, we see that $\E{Z^\lambda} = \frac{\beta}{\beta-\alpha}$.

As a final branching process comparison between these two processes, we calculate the distribution of the total number of generations of descendants of an initial arrival in the ESEP and the Hawkes process. That is, let the first entity be the first generation, its offspring be the second generation, their offspring the third, and so on. In Proposition~\ref{genProp} we find the probability mass function for the ESEP in closed form and a recurrence relation for the cumulative distribution function for the Hawkes process.

\begin{proposition}\label{genProp}
Let $\beta > \alpha > 0$. Let $\mathcal{G}^\eta$ be the number of distinct arrival generations across in full the progeny of an initial arrival in an ESEP with intensity jump $\alpha$ and service rate $\beta$. Then, the probability mass function for $\mathcal{G}^\eta$ is given by
\begin{align}
\PP{\mathcal{G}^\eta = k}
=
\frac{\alpha^{k-1}(\beta-\alpha)}{\beta^k-\alpha^k}
-
\frac{\alpha^k(\beta-\alpha)}{\beta^{k+1} - \alpha^{k+1}}
.
\end{align} Likewise, let $\mathcal{G}^\lambda$ be the number of distinct arrival generations in the full progeny of an initial arrival for a Hawkes process with intensity jump $\alpha  > 0$ and decay rate $\beta$. Then, $\mathcal{G}^\lambda$ has cumulative distribution function $F_{\mathcal{G^\lambda}}(k) = \PP{\mathcal{G}^\lambda \leq k}$ satisfying the recursion
\begin{align}
F_{\mathcal{G^\lambda}}(k) = e^{-\frac{\alpha}{\beta}\left(1 - F_{\mathcal{G^\lambda}}(k-1)\right)} ,
\end{align}
 where $F_{\mathcal{G^\lambda}}(0) = 0$ and all $k \in \mathbb{Z}^+$.
\begin{proof}
Let $Y_k^\lambda$ and $Y_k^\eta$ be Galton-Watson branching processes  defined as
\begin{align}
&Y_k^\lambda = \sum_{i=1}^{Y_{k-1}^\lambda} X_{\lambda, i}^{(k)}
,
&Y_k^\eta = \sum_{i=1}^{Y_{k-1}^\eta} X_{\eta, i}^{(k)}
,
\end{align}
with $X_{\lambda, i}^{(k)} \stackrel{i.i.d.}{\sim} \mathrm{Pois}\left(\frac{\alpha}{\beta}\right)$, $X_{\eta, i}^{(k)} \stackrel{i.i.d.}{\sim} \mathrm{Geo}\left(\frac{\alpha}{\alpha+\beta}\right)$, and $Y_0^\lambda = Y_0^\eta = 1$. These processes then have probability generating functions
$$
\mathcal{P}_k^\lambda(z)
=
\sum_{j=0}^\infty z^j \PP{Y_k^\lambda = j}
\quad
\text{ and }
\quad
\mathcal{P}_k^\eta(z)
=
\sum_{j=0}^\infty z^j \PP{Y_k^\eta = j}
,
$$
that are given by the recursions $\mathcal{P}_{k+1}^\lambda(z) = \mathcal{P}_{X^\lambda}\left(\mathcal{P}_{k}^\lambda(z)\right)$ and $\mathcal{P}_{k+1}^\lambda(z) = \mathcal{P}_{X^\eta}\left(\mathcal{P}_{k}^\eta(z)\right)$ with $\mathcal{P}_{1}^\lambda(z) = \mathcal{P}_{X^\lambda}(z)$ and $\mathcal{P}_{1}^\eta(z) = \mathcal{P}_{X^\eta}(z)$, where $\mathcal{P}_{X^\lambda}(z)$ and $\mathcal{P}_{X^\eta}(z)$ are the probability generating functions of $X_{\lambda,1}^{(1)}$ and $X_{\eta,1}^{(1)}$, respectively; see e.g. Section XII.5 of \citet{feller1957introduction}. One can then use induction to observe that
$$
\mathcal{P}_k^\eta(z)
=
1 - \frac{\alpha^k(1-z)}{\beta^k + \sum_{j=1}^k\alpha^j \beta^{k-j}(1-z)}
,
$$
whereas $\mathcal{P}_k^\lambda(z) = e^{-\frac{\alpha}{\beta}\left(1 - \mathcal{P}^\lambda_{k-1}(z)\right)}$, with $\mathcal{P}^\lambda_1(z) = e^{-\frac{\alpha}{\beta}(1-z)}$. Because of their shared offspring distribution constructions, the number of the progeny in the $k^\text{th}$ arrival generations of the Hawkes process and the ESEP are equivalent in distribution to $Y_k^\lambda$ and $Y_k^\eta$, respectively. In this way, we can express $\mathcal{G}^\lambda$ and $\mathcal{G}^\eta$ as
$$
\mathcal{G}^\lambda
=
\inf\{ k \in \mathbb{Z}^+ \mid Y_k^\lambda = 0 \}
\quad
\text{ and }
\quad
\mathcal{G}^\eta
=
\inf\{ k \in \mathbb{Z}^+ \mid Y_k^\eta = 0 \}
.
$$
This leads us to observe that the events $\{\mathcal{G}^\lambda = j\}$ and $\{Y_j^\lambda = 0, Y_{j-1}^\lambda > 0\}$ are equivalent, as are $\{\mathcal{G}^\eta = j\}$ and $\{Y_j^\eta = 0, Y_{j-1}^\eta > 0\}$. Focusing for now on $\mathcal{G}^\lambda$, we have that
\begin{align*}
\PP{Y_j^\lambda = 0, Y_{j-1}^\lambda > 0}
=
\sum_{i=1}^\infty \PP{X_{\lambda,1}^{(1)} = 0}^i\PP{Y_{j-1}^\lambda = i}
=
\mathcal{P}^\lambda_{j-1}\left(\PP{X_{\lambda,1}^{(1)} = 0}\right)
-
\PP{Y_{j-1}^\lambda = 0}
,
\end{align*}
and since $\PP{K = 0} = \mathcal{P}(0)$ for any non-negative discrete random variable $K$ with probability generating function $\mathcal{P}(z)$, this yields
$$
\PP{\mathcal{G}^\lambda = j} = \mathcal{P}^\lambda_{j}(0) - \mathcal{P}^\lambda_{j-1}(0) .
$$
Using $\mathcal{P}^\lambda_{0}(0) = 0$, this telescoping sum now produces the stated form of the cumulative distribution function for $\mathcal{G}^\lambda$. By analogous arguments for $\mathcal{G}^\eta$, we complete the proof.
\end{proof}
\end{proposition}

In the following subsection we  focus on the ESEP, using the insight we have now gained from branching processes to connect this process to stochastic models for preferential attachment that are popular in the Bayesian nonparametric and machine learning literatures.

\subsection{Similarities with Preferential Attachment and Bayesian Statistics Models}

In the branching process perspective of the ESEP, consider the total number of active families at one point in time. That is, across all the entities present in the system at a given time, we are interested in the number of distinct progeny to which these entities belong. As each  arrival occurs, the new entity either belongs to one of the existing families, meaning that the entity is a descendant, or it forms a new family, which is to say that it is an immigrant. If the entity is joining an existing family, it is more likely to join families that have more presently active family members.

We can note that these dynamics are quite similar to the definition of the Chinese Restaurant Process (CRP), see Chapter 11 in \citet{aldous1985exchangeability}. The CRP models the successive arrival of customers to the restaurant that has infinitely many tables that each have infinitely many seats. Each arriving customer chooses which table to join based on the decisions of those before. Specifically, the $n^\text{th}$ customer to arrive joins table $i$ with probability $\frac{c_i}{n-1+\lambda}$ or otherwise starts a new table with probability $\frac{\lambda}{n-1+\lambda}$, where $c_i$ is the number at table $i$ and $\lambda > 0$. As the number seated at table $i$ grows larger, it is increasingly likely that the next customer will choose to sit at table $i$. In the ESEP, a new arrival at time $t \geq 0$ was generated as part of  active excitement family $i$ with probability $\frac{\alpha Q_{t,i}}{\alpha Q_t + \eta^*}$ and otherwise was an externally generated arrival with probability $\frac{\eta^*}{\alpha Q_t + \eta^*}$, where $Q_{t,i}$ is the number of active exciters in the system at time $t$ in the $i^\text{th}$ excitement family with $Q_t = \sum_{i} Q_{t,i}$. By normalizing the numerator and denominator of these probabilities by $\frac{1}\alpha$, we see that these dynamics match the CRP almost exactly. The difference is hardly a novel idea for restaurants -- in the ESEP diners eventually leave. This departure then decreases the number of customers at the table, making it less attractive to the next person to arrive.

In addition to being an intriguing stochastic model, the CRP is also of interest for Bayesian statistics and machine learning through its connection to Bayesian nonparametric mixture models, specifically Dirichlet process mixtures. By consequence, the CRP then also has commonality with urn models and models for preferential attachment, see e.g. \citet{blackwell1973ferguson}. The CRP is also established enough to have its own generalizations, such as the distance dependent CRP in \citet{blei2011distance}, in which the probability a customer joins a table is dependent on a distance metric, and the recurrent CRP in \citet{ahmed2008dynamic}, in which the restaurant closes at the end of each day forcing all of that day's customers to simultaneously depart. Drawing inspiration from the CRP and from the branching process perspectives of the ESEP, we investigate the distribution of the number of active families in the ESEP. Equivalently stated, this is the number of active tables in a continuous time CRP in which customers leave after their exponentially distributed meal durations. To begin, we first find the expected amount of time until a newly formed table becomes empty.



\begin{proposition}\label{familyDuration}
Suppose that an ESEP receives an initial arrival at time 0. Let $X_t$ be the number of entities in the system at time $t \geq 0$ that are progeny of the initial arrival and let $\tau$ be a stopping time such that $\tau = \inf \{t \geq 0 \mid X_t = 0\}$. Then, the expected value of $\tau$ is
\begin{align}
\E{\tau}
=
\frac{1}{\alpha}\log\left(\frac{\beta}{\beta - \alpha}\right)
,
\end{align}
where $\alpha > 0$ is the intensity jump size and $\beta > \alpha$ is the expiration rate.
\end{proposition}
\begin{proof}
To observe this, we note that $X_t$ can be viewed as the state of an absorbing continuous time Markov chain on the non-negative integers. State 0 is the single absorbing state and in any other state $j$ the two possible transitions are to $j+1$ at rate $\alpha j$ and to $j-1$ at rate $\mu j$, as visualized below.
\begin{center}
\begin{tikzpicture}[scale=.825,line cap=round,line join=round,>={Stealth[width=1.5mm,length=1.5mm]},x=1.0cm,y=1.0cm]

 \node (zero) [draw, circle, minimum size=1cm] at (-2,1cm) {$0$};
 \node (one) [draw, circle, minimum size=1cm] at (0,1cm) {$1$};
 \node (two) [draw, circle, minimum size=1cm] at (2,1cm) {$2$};
 \node (three) [draw, circle, minimum size=1cm] at (4,1cm) {$3$};
  \node (four) [draw, circle, minimum size=1cm] at (6,1cm) {$4$};
 \node (dots)  [draw, circle, minimum size=1cm,color=white] at (8,1cm) {$\mathbf{\textcolor{black}{\dots}}$};

 \draw [->] (one) .. controls +(.25,1.25) and +(-.25,1.25).. node [midway, above] {$\alpha$} (two);
 \draw [->] (two) .. controls +(.25,1.25) and +(-.25,1.25) ..  node [midway, above] {$2\alpha$}  (three);
 \draw [->] (three) .. controls +(.25,1.25) and +(-.25,1.25) ..  node [midway, above] {$3 \alpha$}  (four);
  \draw [->] (four) .. controls +(.25,1.25) and +(-.25,1.25) ..  node [midway, above] {$4 \alpha$}  (dots);

 \draw [->] (dots) .. controls +(-.25,-1.25) and +(.25,-1.25) .. node [midway, below] {$5 \beta$} (four);
  \draw [->] (four) .. controls +(-.25,-1.25) and +(.25,-1.25) .. node [midway, below] {$4 \beta$} (three);
 \draw [->] (three) .. controls +(-.25,-1.25) and +(.25,-1.25) .. node [midway, below] {$3 \beta$} (two);
 \draw [->] (two) .. controls +(-.25,-1.25) and +(.25,-1.25) .. node [midway, below] {$2 \beta$} (one);
 \draw [->] (one) .. controls +(-.25,-1.25) and +(.25,-1.25) .. node [midway, below] {$\beta$} (zero);

\end{tikzpicture}
\end{center}
Then, $\tau$ is the time of absorption into state 0 when starting in state 1 and so $\E{\tau}$ can be calculated by standard first step analysis approaches, yielding
$$
\E{\tau}
=
\sum_{i=1}^\infty \frac{1}{\alpha i} \prod_{j=1}^{i} \frac{\alpha j}{\beta j} = \frac{1}{\alpha}\sum_{i=1}^\infty \frac{1}{i}\left(\frac{\alpha}{\beta}\right)^i = \frac{1}{\alpha}\log\left(\frac{1}{1-\frac{\alpha}{\beta}}\right),
$$
and this simplifies to the stated result.
\end{proof}

Proposition~\ref{familyDuration} gives the expectation of the total time of an excitement family is active in the system. Using this, in Proposition~\ref{activeProp} we now employ a classical queueing theory result to find the exact distribution of the number of active families simultaneously in the system  in steady-state.

\begin{proposition}\label{activeProp}
Let $B$ be the number of distinct excitement families that have progeny active in the system in steady-state of an ESEP with baseline intensity $\eta^* > 0$, intensity jump $\alpha > 0$, and expiration rate $\beta > \alpha$. Then, $B \sim \mathrm{Pois}\left(\frac{\eta^*}{\alpha}\log \left(\frac{\beta}{\beta - \alpha}\right)\right)$.
\begin{proof}
We first note that new excitement families are started when a baseline-generated arrival occurs, which follows a Poisson process with rate $\nu^*$. The duration excitement family's time in system then has mean given by Proposition~\ref{familyDuration}. Because there is no limitation on the number of possible families in the system at once, this is equivalent to an infinite server queue with Poisson process arrivals and generally distributed service, an $M/G/\infty$ queue in Kendall notation. This process is known to have Poisson distributed steady-state distribution, see e.g. \citet{eick1993physics}, with mean given by the product of the arrival rate and the mean service duration, which yields the stated form for $B$.
\end{proof}
\end{proposition}

An interesting consequence of the number of active families being Poisson distributed and the total number in system being negative binomially distributed is that it suggests that the number of simultaneously active family members is logarithmically distributed. We observe this via the known compound Poisson representation of the negative binomial distribution \cite{willmot1986mixed}. For $B \sim \mathrm{Pois}\left(\frac{\eta^*}{\alpha}\log \left(\frac{\beta}{\beta - \alpha}\right)\right)$, $Q \sim \mathrm{NegBin}\left(\frac{\alpha}{\beta}, \frac{\eta^*}{\alpha}\right)$, and $L_i \stackrel{\mathrm{iid}}{\sim} \mathrm{Log}\left(\frac{\alpha}{\beta} \right)$, then one can observe that
$$
Q \stackrel{D}{=} \sum_{i=1}^B L_i,
$$
where $\PP{L_1 = k} = \left(\frac{\alpha}{\beta}\right)^k\left(k \log\left(\frac{\beta}{\beta-\alpha}\right)\right)^{-1}$ for all $k \in \mathbb{Z}^+$. Thus, the idea that the number of active members of each family is logarithmically distributed follows from the fact that this is a sum of positive integer valued random variables, of which there are as many as there are active families, and this sum is equal to the total number in system.

\subsection{Connections to Epidemic Models}\label{epidemics}

As a final observation regarding the ESEP and its connections to other stochastic models, consider disease spread. As we discussed in the introduction to this paper, when a person becomes sick with a contagious disease she increases the rate of new infection through her contact with others. Furthermore when a person recovers from a disease such as the flu, she is no longer contagious and thus she no longer contributes to the rate of disease spread. While we have discussed in the introduction that this scenario has the hallmarks of self-excitement qualitatively, a classic model for studying this phenomenon is the Susceptible-Infected-Susceptible (SIS) process.

In the SIS model there is a finite population of $N \in \mathbb{Z}^+$ individuals. Each individual takes on one of two states, either infected or susceptible. Let $I_t$ be the number infected at time $t \geq 0$ and $S_t$ be the number susceptible. In the continuous time stochastic SIS model, each infected individual recovers after an exponentially distributed duration of the illness. Once a person recovers from the disease, she becomes susceptible again. Because there is a finite population, the rate of new infection depends on both the number infected and the number susceptible; a new person falls ill at a rate proportional to $ I_t\cdot \frac{S_t}{N}$. Because this CTMC would be absorbed into state $I_t = 0$, it is common to include an exogenous infection rate proportional to just $\frac{S_t}{N}$. We will refer to this model as the stochastic SIS with exogenous infections, and Figure~\ref{SISfig} shows rate diagram for the transitions from infected to susceptible and from susceptible to infected. For the sake of comparison, we set the exogenous infection rate as $\eta^*$, the epidemic infection rate as $\alpha$, and the recovery rate as $\beta$.

\begin{figure}[h]
\centering
\begin{tikzpicture}[scale=1,line cap=round,line join=round,>={Stealth[width=2.5mm,length=2.5mm]},x=1.0cm,y=1.0cm]

\draw  (-3.5,1.5) rectangle (-2.5,0.5);
\node at (-3,1) {$S_t$};
\draw  (-1,1.5) rectangle (0,0.5);
\node at (-0.5,1) {$I_t$};
\node (topS) at (-3,1.35) {};
\node (topI) at (-0.5,1.35) {};
\node (bottomI) at (-0.5,0.65) {};
\node (bottomS) at (-3,0.65) {};
 \draw [->] (topS) .. controls +(.35,1.25) and +(-.35,1.25) .. node [midway, above] {$\left(\eta^* + \alpha I_t \right)\frac{S_t}{N}$} (topI);
 \draw [->] (bottomI) .. controls +(-.35,-1.25) and +(.35,-1.25) .. node [midway, below] {$\beta I_t$} (bottomS);
\end{tikzpicture}
\caption{Stochastic SIS model with exogenous infections}\label{SISfig}
\end{figure}
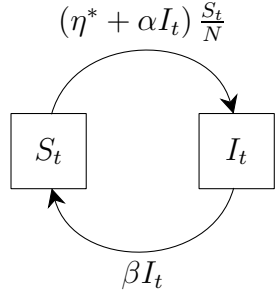

One can note that there are immediate similarities between this process and the ESEP. That is, new infections increase the infection rate while recoveries decrease it, and infections can be the result of either external or internal stimuli. However, the primary difference between these two models is that the SIS process has a finite population, whereas the ESEP does not. In Proposition~\ref{sisConvProp} we find that as this population size grows large the difference between these models fades, yielding that the distribution of the number infected in the exogenously driven SIS model converges to the distribution of the queue length in the ESEP.

\begin{proposition}\label{sisConvProp}
Let $I_t$ be the number of infected individuals at time $t \geq 0$ in an exogenously driven stochastic SIS model with population size $N \in \mathbb{Z}^+$, exogenous infection rate $\eta^* > 0$, epidemic infection rate $\alpha > 0$, and recovery rate $\beta > 0$. Then, as $N \to \infty$
$$
I_t
\stackrel{D}{\Longrightarrow}
Q_t
,
$$
where $Q_t$ is the active number in system at time $t$ for an ESEP with baseline intensity $\eta^*$, intensity jump $\alpha$, and expiration rate $\beta$.
\begin{proof}
Because the SIS model is a Markov process, one can use the infinitesimal generator approach to find a time derivative for the moment generating function of the number of infected individuals at time $t \geq 0$. Thus, by noting that $S_t = N - I_t$ we have that
\begin{align*}
\frac{\mathrm{d}}{\mathrm{d}t}
\E{e^{\theta I_t}}
&
=
\E{
\frac{\alpha I_t S_t}{N} \left( e^{\theta}-1\right) e^{\theta I_t}
+
\beta I \left(e^{-\theta} - 1 \right) e^{\theta I_t}
+
\frac{\eta^* S_t}{N} \left(e^{\theta} - 1\right) e^{\theta I_t}
}
\\
&
=
\E{
\frac{\alpha I_t (N - I_t)}{N} \left( e^{\theta}-1\right) e^{\theta I_t}
}
+
\E{
\beta I_t \left(e^{-\theta} - 1 \right) e^{\theta I_t}
}
+
\E{
\frac{\eta^* (N- I_t)}{N} \left(e^{\theta} - 1\right) e^{\theta I_t}
}
,
\end{align*}
which we can re-express in partial differential equation form as
\begin{align*}
\frac{\partial \E{e^{\theta I_t}}}{\partial t}
&
=
\left(
\alpha   \left( e^{\theta}-1\right)
+
\beta \left(e^{-\theta} - 1 \right)
-
\frac{\eta^* }{N} \left(e^{\theta} - 1\right)
\right)
\frac{\partial \E{e^{\theta I_t}}}{\partial \theta}
-
\frac{\alpha }{N} \left( e^{\theta}-1\right)
\frac{\partial^2 \E{e^{\theta I_t}}}{\partial \theta^2}
\\
&
\quad
+
\eta^*  \left(e^{\theta} - 1\right)
\E{
e^{\theta I_t}
}
.
\end{align*}
Now as the population size $N \to \infty$, this converges to
\begin{align*}
\frac{\partial \E{e^{\theta I_t}}}{\partial t}
&
=
\left(
\alpha   \left( e^{\theta}-1\right)
+
\beta \left(e^{-\theta} - 1 \right)
\right)
\frac{\partial \E{e^{\theta I_t}}}{\partial \theta}
+
\eta^*  \left(e^{\theta} - 1\right)
\E{
e^{\theta I_t}
}
,
\end{align*}
which we can recognize as the partial differential equation for the moment generating function of the ESEP through its own infinitesimal generator.
\end{proof}
\end{proposition}

 As a demonstration of this convergence, we plot the empirical steady-state distribution of the SIS process for increasing population size below in Figure~\ref{sisconv}. Note that in this example the distributions appear fairly close for populations of size $N = 1,000$. On the scale of the populations of cities or even some larger high schools, this is quite small. At a medium university size of $N = 10,000$, the distributions are essentially indistinguishable.

 \begin{figure}[h]
\begin{center}	
\includegraphics[width=.8\textwidth]{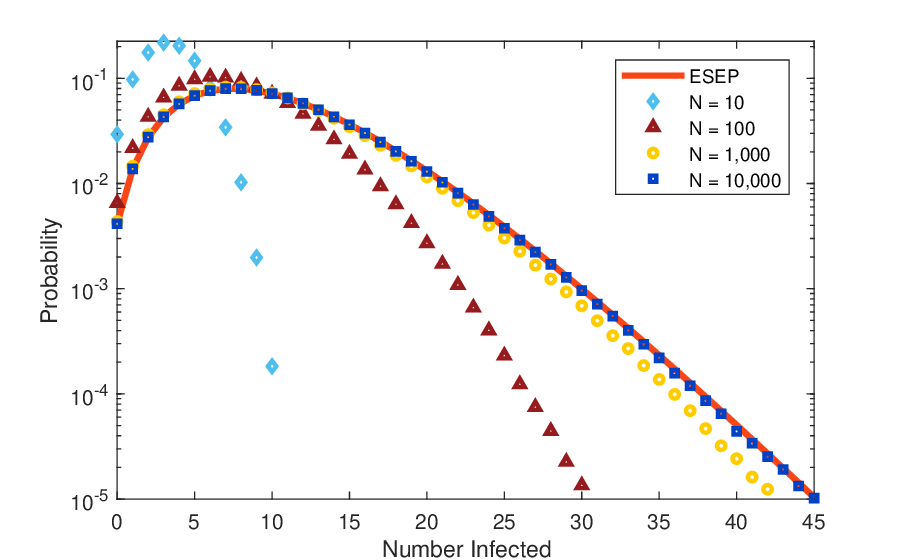}
\caption{Steady-state distribution of the number infected in the exogenously driven SIS model for increasing population size $N$, where $\eta^* = 10$, $\alpha = 2$, and $\beta = 3$.} \label{sisconv}
\end{center}
\end{figure}

 We would be remiss if we did not note that connections from epidemic models to birth-death processes are not new. For example, \citet{ball1983threshold} demonstrated that epidemic models converge to birth-death processes, and \citet{singh2014outbreak} even noted that the exogenously driven Susceptible-Infected-Recovered (SIR) model -- that is, people cannot become re-infected -- converges to a linear birth-death-immigration process; however, these works did not outright form connections to self-exciting processes. In \citet{rizoiu2018sir}, the similarities between the Hawkes process and the SIR process are shown and formal connections are made, although this is through a generalization of the Hawkes process defined on a finite population rather than through increasing the epidemic model population size. Regardless, the topics considered in these prior works serve to expand the practical relevance of the ESEP, as they note that these epidemic models are also of use outside of public health. For example, the contagious nature of these models has also been used to study topics like product adoption, idea spread, and social influence. These all also naturally relate to the concept of self-excitement, and in Proposition~\ref{sisConvProp} we observe that this connection can be formalized. 
 
 These connections also take on an added importance in the contemporaneous context of the COVID-19 public health crisis. It is worth noting that the convergence of distributions in Proposition~\ref{sisConvProp} does not require our commonly assumed stability condition $\beta > \alpha$. Thus, this convergence also covers potentially pandemic scenarios in which $\alpha \geq \beta$. In such settings, one can quickly observe that the mean number infected by time $t$ is such that $\E{N_t} \in O(e^{(\alpha - \beta)t})$, reproducing the exponential growth exhibited by this novel coronavirus in these early stages. The ESEP arrival process then captures the times of new infections, some portion of which may then be used as the arrival process to a queueing model for the cases that require hospitalization. 

\section{Constructing Eternal Self-Excitement from Ephemeral Self-Excitement}\label{secGESEP}

Thus far we have exclusively considered a Markovian model for ephemeral self-excitement. However, just as the Hawkes process need not be Markovian, we do not have to restrict ourselves to settings in which the Markov property exists. Recall from Subsection~\ref{subsecHawkes} that the general definition from \citet{hawkes1971spectra} described the intensity as
$$
\lambda_t
= 
\lambda^* + \int_{-\infty}^t g(t-u)\mathrm{d}N_{u,\lambda}
=
\lambda^* 
+
\sum_{i=1}^{N_t} g(t - A_i)
,
$$
where $g: \mathbb{R}^+ \to \mathbb{R}^+$ and $\int_0^\infty g(x) \mathrm{d}x < 1$ for stability. One could also consider a marked Hawkes process that draws jump sizes from a sequence of i.i.d.~positive random variables $\{M_i \mid i \in \mathbb{Z}^+\}$, in which case the summation form of intensity would instead be expressed
$$
\lambda_t
=
\lambda^* 
+
\sum_{i=1}^{N_t} M_i g(t - A_i)
.
$$
The ESEP model has provided us a natural comparison for the popular Markovian case where $g(x) = \alpha e^{-\beta x}$ (with no marks), but let us now consider more general excitation kernels and jump sizes. To do so, we will make two main generalizations while preserving the other key elements of the ESEP, such as the affine relationship between the intensity and the number of active exciters. First, we will change the activity duration to be a general distribution, mimicking the general excitation kernel. Secondly, we also change from solitary arrivals to batch arrivals, meaning groups of events that occur simultaneously at each arrival epoch. The size of these batches may be drawn from an i.i.d.~sequence of positive integer random variables, so we will use the parameter $n$ to represent the relative size of this batch through the mean of the batch size distribution. This leads us to the \textit{$n^\text{th}$ general ephemerally self-exciting process} ($n$-GESEP), defined now in Definition~\ref{gesepDef}.

\begin{definition}[$n^\text{th}$  general ephemerally self-exciting process]\label{gesepDef}
For times $t \geq 0$, a baseline intensity $\eta^* > 0$, cumulative distribution function $G:\mathbb{R}^+ \to [0,1]$, i.i.d.~sequence of positive random variables $\{B_i \mid i \in \mathbb{Z}^+\}$, and $n \in \mathbb{Z}^+$ such that $\E{B_i} \in O(n)$, let $N_t(n)$ be a  counting process for the arrival epochs occurring according to the stochastic intensity $\eta_t(n)$, defined such that
\begin{align}
\eta_t(n) = \eta^* + \frac{\alpha}{n} Q_t(n)
,
\end{align}
where $Q_t(n)$ is incremented by $B_i(n)$ at the $i^\text{th}$ arrival epoch in $N_t(n)$ and then is depleted at unit down-jumps at the expiration of each individual arrival's activity duration, which are i.i.d.~draws from the distribution $G(\cdot)$ across all batches and all epochs. Then, we say that $(\eta_t(n), N_t(n))$ is the \textit{$n^\text{th}$ general ephemerally self-exciting process} ($n$-GESEP).
\end{definition}

It is important to note that in this definition, a batch of size $n$ occurring at the current time would increase the present arrival rate by $\alpha$. However, each of these $n$ activity durations are mutually independent. Thus, despite the common arrival epoch, each expiration should cause an instantaneous decrease of just ${\alpha}\slash{n}$ if the activity duration distribution is continuous. It is also worth noting that this $n$-GESEP model encapsulates the simple generalization of the ESEP with general service durations, as this is given by $(\eta_t(1), N_t(1))$ with $\PP{B_1(1) = 1}$. Like how the ESEP could be viewed as a Markovian infinite server queue with a state dependent arrival rate, the $n$-GESEP can seen as an $M^B/G/\infty$ queue with a state dependent arrival rate. This perspective implies the existence and uniqueness of Definition~\ref{gesepDef}.

Just as it can often be beneficial to think of the length of an infinite server queue as a sum over all customers that remain in service, it will be quite useful for us to think about the number active within the $n$-GESEP as the sum of all arrivals that have not yet expired. For $A_i$ as the $i^\text{th}$ arrival epoch and $S_{i,j}$ as the $j^\text{th}$ activity duration within the $i^\text{th}$ batch, this can be expressed
$$
Q_t(n) = \sum_{i=1}^{N_t(n)} \sum_{j=1}^{B_i} \mathbf{1}\{t < A_i + S_{i,j}\}
,
$$
or equivalently for the intensity,
$$
\eta_t = \eta^* + \frac{\alpha}{n} \sum_{i=1}^{N_t(n)} \sum_{j=1}^{B_i} \mathbf\{t < A_i + S_{i,j}\}
.
$$
A nice consequence of this representation is that the ephemerality is at the forefront. Because the event within the indicator is that the current time $t$ is less than the sum of a given arrival time and activity duration, this indicator counts whether that particular excitement is currently active. Thus,  this indicator will switch from 1 to 0 when the present time passes a given expiration time, causing the intensity to drop by $\alpha \slash n$. From this point forward, the excitement brought by this particular arrival no longer has a direct effect on the system.

We will now use the $n$-GESEP to establish a main result of this work, in which we connect this general form of ephemeral self-excitement to general forms of the traditional, eternal notion of self-excitement. In Theorem~\ref{hawkconv}, we prove a scaling limit that incorporates random batch distributions and general service to construct marked, general decay Hawkes processes. We refer to this limit as a ``batch scaling,'' as we are letting the relative batch size $n$ grow large, which simultaneously shrinks the size of the excitement generated by each individual entity within a batch of arrivals.  Thus, while the effect of one individual exciter shrinks, the collective effect of a batch arrival remains fixed. In the limit, this provides an alternate construction of the Hawkes process.

\begin{theorem}\label{hawkconv}
For $t \geq 0$ and $n \in \mathbb{Z}^+$, let $(\eta_t(n), N_t(n))$ be a $n$-GESEP with baseline intensity $\eta^* > 0$, intensity jump size $\alpha > 0$, activity duration CDF $G(\cdot)$, and i.i.d. batch size sequence of non-negative, discrete random variables $\{B_i \mid i \in \mathbb{Z}^+\}$. Additionally, let $\eta_0(n) = \eta^*$, i.e.~$Q_0(n) = 0$. Then, for all $t \geq 0$, this $n$-GESEP process is such that
\begin{align}
 \eta_t(n)
\stackrel{D}{\Longrightarrow}
\lambda_t
\text{ and }
N_t(n)
\stackrel{D}{\Longrightarrow}
N_{t,\lambda}
\end{align}
as $n \to \infty$, where $(\lambda_t, N_{t,\lambda})$ is the general Hawkes process intensity and counting process pair such that
\begin{align}
\lambda_t
=
\eta^*
+
\sum_{i=1}^{N_{t,\lambda}}
M_i
\bar G(t-A_i)
,
\label{Hawkesdef}
\end{align}
where $\{A_i \mid i \in \mathbb{Z}^+\}$ are the Hawkes process arrival epochs, $\bar G(x) = 1 - G(x)$ for all $x \geq 0$, and $\{M_i \mid i \in \mathbb{Z}^+\}$ is an i.i.d.~sequence of positive real random variables such that ${\alpha B_1}\slash{n} \stackrel{D}{\Longrightarrow} M_1$ and ${B_1}\slash{n^2} \stackrel{p}{\longrightarrow} 0$.
\begin{proof}
We will organize the proof into two parts. Each part is oriented around the process arrival times, as these fully determine the sample path of the Hawkes process. We will first show through induction that the distributions of inter-arrival times converge. Then, we will demonstrate that given the same arrival times, the dynamics of the processes converge.

To begin, let $A_i^\eta$ for $i \in \mathbb{Z}^+$ be the time of the $i^\text{th}$ arrival in the $n$-GESEP and similarly let $A^\lambda_i$ be the $i^\text{th}$ arrival time for the Hawkes process. We start with the base case: for the time of the first arrival, we can note that for all $n$-GESEP models,
$$
\PP{A_1^\eta > x} = e^{-\eta^* x},
$$
as $Q_0(n) = 0$ and thus the first arrival is driven by the external baseline rate. Likewise for the Hawkes process, since Equation~\ref{Hawkesdef} implies that  $\lambda_t = \nu^*$ for $0 \leq t < A_1^\lambda$, we can see that
$$
\PP{A_1^\lambda > x} = e^{-\eta^* x},
$$
and thus $\PP{A_1^\lambda > x} = \PP{A_1^\eta > x}$. As an inductive hypothesis, we now assume that $\{A_1^\eta, \dots, A_k^\eta\}$ converge in joint and marginal distributions to $\{A_1^\lambda, \dots, A_k^\lambda\}$ where $k \in \mathbb{Z}^+$. Now, for the Hawkes process we can observe that
\begin{align*}
\mathrm{P}_k\left(A_{k+1}^\lambda - A_{k}^\lambda > x\right)
&:=
\PP{A_{k+1}^\lambda - A_{k}^\lambda > x \mid \{A_1^\lambda, \dots, A_k^\lambda\} }
=
\mathrm{E}_k\left[ e^{-\int_{0}^{x} \lambda_{A_k^\lambda + t} \mathrm{d}t} \right],
\end{align*}
because when conditioned on the arrival history, the Hawkes process behaves like an imhogoneous Poisson process until the next arrival occurs. Using Equation~\ref{Hawkesdef}, we can express this as
\begin{align*}
\mathrm{E}_k\left[ e^{-\int_{0}^{x} \lambda_{A_k^\lambda + t} \mathrm{d}t} \right]
&=
e^{-\eta^* x} \mathrm{E}_k\left[ e^{-\int_{0}^{x} \sum_{i=1}^k M_i \bar G \left(A_k^\lambda - A_i^\lambda + t\right) \mathrm{d}t} \right]
=
e^{-\eta^* x} \prod_{i=1}^k \mathrm{E}_k\left[ e^{-  M_i \int_{0}^{x} \bar G \left(A_k^\lambda - A_i^\lambda + t\right) \mathrm{d}t} \right]
.
\end{align*}
Turning to the ESEP epochs, we define $N_{i,j}^\eta\left( (t, t+s]\right)$ as the number of arrivals on the time interval $(t, t+s]$ that are generated by the excitement caused by the $j^\text{th}$ entity within the $i^\text{th}$ batch. Furthermore, let $N_{*}^\eta\left( (t, t+s]\right)$ be the number of arrivals on $(t, t+s]$ that are generated by the external, baseline rate $\eta^*$. Then, using this notation we have that
\begin{align*}
\mathrm{P}_k\left(A_{k+1}^\eta - A_{k}^\eta > x\right)
&:=
\PP{A_{k+1}^\eta - A_{k}^\eta > x \mid \{A_1^\eta, \dots, A_k^\eta\} }
\\
&
=
\mathrm{P}_k\left(    \bigcap_{i=1}^k \bigcap_{j=1}^{B_i} \left\{N_{i,j}^\eta\left( (A_k^\eta, A_{k}^\eta+x]\right) = 0\right\} \cap  \left\{N_{*}^\eta\left( (A_k^\eta, A_{k}^\eta+x]\right) = 0\right\} \right)
\\
&
=
\mathrm{E}_k\left[    \prod_{i=1}^k \prod_{j=1}^{B_i} \mathbf{1}\left\{N_{i,j}^\eta\left( (A_k^\eta, A_{k}^\eta+x]\right) = 0\right\}   \mathbf{1}\left\{N_{*}^\eta\left( (A_k^\eta, A_{k}^\eta+x]\right) = 0\right\} \right]
.
\end{align*}
From the independence of each of these arrival processes, we can move the probability for no arrivals in the external arrival process and the product over $i$ outside of the expectation to receive
\begin{align*}
&\mathrm{E}_k\left[    \prod_{i=1}^k \prod_{j=1}^{B_i} \mathbf{1}\left\{N_{i,j}^\eta\left( (A_k^\eta, A_{k}^\eta+x]\right) = 0\right\}   \mathbf{1}\left\{N_{*}^\eta\left( (A_k^\eta, A_{k}^\eta+x]\right) = 0\right\} \right]
\\
&
\quad
=
e^{-\eta^* x}\prod_{i=1}^k \mathrm{E}_k\left[    \prod_{j=1}^{B_i} \mathbf{1}\left\{N_{i,j}^\eta\left( (A_k^\eta, A_{k}^\eta+x]\right) = 0\right\}   \right]
.
\end{align*}
Consider an arbitrary entity, say the $j^\text{th}$ entity in the $i^\text{th}$ batch. Let $S_{i,j}$ be its service duration. If this entity has departed from the queue before $A_k^\eta$, then it cannot generate further arrivals and thus
$$
\mathrm{P}_k\left( N_{i,j}^\eta\left( (A_k^\eta, A_{k}^\eta+x]\right) = 0 \mid S_{i,j} \leq A_k^\eta - A_i^\eta\right) = 1
.
$$
Likewise, if it does not depart until after $A_k^\eta + x$, then the probability that it generates an arrival on $(A_k^\eta, A_{k}^\eta+x]$ is
$$
\mathrm{P}_k\left( N_{i,j}^\eta\left( (A_k^\eta, A_{k}^\eta+x]\right) = 0 \mid S_{i,j} \geq A_k^\eta - A_i^\eta + x\right) = e^{-\frac{\alpha}{n}x}
.
$$
Finally, if the entity departs in the interval $(A_k^\eta, A_{k}^\eta+x]$, the probability it generates an arrival before departing is
$$
\mathrm{P}_k\left( N_{i,j}^\eta\left( (A_k^\eta, A_{k}^\eta+x]\right) = 0 \mid S_{i,j} = A_k^\eta - A_i^\eta + z\right) = e^{-\frac{\alpha}{n}z}
,
$$
where $0 < z < x$. Therefore through conditioning on each entity's service duration, we have that
\begin{align*}
&e^{-\eta^* x}\prod_{i=1}^k \mathrm{E}_k\left[    \prod_{j=1}^{B_i} \mathbf{1}\left\{N_{i,j}^\eta\left( (A_k^\eta, A_{k}^\eta+x]\right) = 0\right\}   \right]
\\
&
\quad
=
e^{-\eta^* x}\prod_{i=1}^k  \mathrm{E}_k\left[   \prod_{j=1}^{B_i} \left( G(A^\eta_k - A^\eta_i) + e^{-\frac{\alpha}{n}x}\bar{G}(A^\eta_k - A^\eta_i + x) + \int_0^x e^{-\frac{\alpha}{n}z}g(A_k^\eta - A_i^\eta +z)\mathrm{d}z\right)   \right],
\end{align*}
where $g(\cdot)$ is the density corresponding to $G(\cdot)$. Since the term inside the inner product does not depend on the specific entity within a batch but rather just the batch itself, we can evaluate this inside the expectation as
\begin{align*}
&e^{-\eta^* x}\prod_{i=1}^k  \mathrm{E}_k\left[   \prod_{j=1}^{B_i} \left( G(A^\eta_k - A^\eta_i) + e^{-\frac{\alpha}{n}x}\bar{G}(A^\eta_k - A^\eta_i + x) + \int_0^x e^{-\frac{\alpha}{n}z}g(A_k^\eta - A_i^\eta +z)\mathrm{d}z\right)   \right]
\\
&
\quad
=
e^{-\eta^* x}\prod_{i=1}^k  \mathrm{E}_k\left[    \left( G(A^\eta_k - A^\eta_i) + e^{-\frac{\alpha}{n}x}\bar{G}(A^\eta_k - A^\eta_i + x) + \int_0^x e^{-\frac{\alpha}{n}z}g(A_k^\eta - A_i^\eta +z)\mathrm{d}z\right)^{B_i}   \right].
\end{align*}
Since the base term of this exponent is deterministic, we will simplify it as follows. Using integration by parts on $\int_0^x e^{-\frac{\alpha}{n}z}g(A_k^\eta - A_i^\eta +z)\mathrm{d}z$ and expanding $\bar{G}(x) = 1 - G(x)$, this simplifies to
\begin{align*}
G(A^\eta_k - A^\eta_i) + e^{-\frac{\alpha}{n}x}\bar{G}(A^\eta_k - A^\eta_i + x) + \int_0^x e^{-\frac{\alpha}{n}z}g(A_k^\eta - A_i^\eta +z)\mathrm{d}z
&
=
e^{-\frac{\alpha}{n}x} + \frac{\alpha}{n}\int_0^x e^{-\frac{\alpha}{n}z}G(A_k^\eta - A_i^\eta +z)\mathrm{d}z
.
\end{align*}
If we express $e^{-\frac{\alpha}{n}x}$ in integral form via $e^{-\frac{\alpha}{n}x} = 1 - \frac{\alpha}{n}\int_0^x e^{-\frac{\alpha}{n}z}\mathrm{d}z$, we can further simplify this expression of the base to
\begin{align*}
e^{-\frac{\alpha}{n}x} + \frac{\alpha}{n}\int_0^x e^{-\frac{\alpha}{n}z}G(A_k^\eta - A_i^\eta +z)\mathrm{d}z
&
=
1 - \frac{\alpha}{n}\int_0^x e^{-\frac{\alpha}{n}z} \bar{G}(A_k^\eta - A_i^\eta + z) \mathrm{d}z
.
\end{align*}
This form makes it quick to observe that this base term is at most 1. Thus we are justified in taking the expectation of this term raised to $B_i$, since that is equivalent to the probability generating function of the batch size and this exists for all discrete random variables when evaluated on values less than or equal to 1 in absolute value. Returning to this expectation, we first note that for all $x$, rearranging the Taylor expansion of $e^x$ produces
\begin{align*}
1 + x
&
=
e^{x} - \sum_{j=2}^\infty \frac{x^j}{j!}
=
e^x\left(1 - e^{-x} \sum_{j=2}^\infty \frac{x^j}{j!}\right)
=
e^{x + \log\left(1 - e^{-x} \sum_{j=2}^\infty \frac{x^j}{j!}\right)}
.
\end{align*}
Thus we re-express the expectation in exponential function form as
\begin{align*}
&e^{-\eta^* x}\prod_{i=1}^k  \mathrm{E}_k\left[    \left( G(A^\eta_k - A^\eta_i) + e^{-\frac{\alpha}{n}x}\bar{G}(A^\eta_k - A^\eta_i) + \int_0^x e^{-\frac{\alpha}{n}z}g(A_k^\eta - A_i^\eta +z)\mathrm{d}z\right)^{B_i}   \right]
\\
&
\quad
=
e^{-\eta^* x}\prod_{i=1}^k  \mathrm{E}_k\left[    e^{- \frac{\alpha}{n}B_i\int_0^x e^{-\frac{\alpha}{n}z} \bar{G}(A_k^\eta - A_i^\eta + z) \mathrm{d}z + O\left(\frac{B_i}{n^2} \right)}  \right]
.
\end{align*}
Through use of a Taylor expansion on $e^{-\frac{\alpha}{n}z}$ and absorbing higher terms into the $O\left(\frac{B_i}{n^2} \right)$ notation, we can further simplify to
\begin{align*}
e^{-\eta^* x}\prod_{i=1}^k  \mathrm{E}_k\left[    e^{- \frac{\alpha}{n}B_i\int_0^x e^{-\frac{\alpha}{n}z} \bar{G}(A_k^\eta - A_i^\eta + z) \mathrm{d}z + O\left(\frac{B_i}{n^2} \right)}  \right]
&
=
e^{-\eta^* x}\prod_{i=1}^k  \mathrm{E}_k\left[    e^{- \frac{\alpha}{n}B_i\int_0^x  \bar{G}(A_k^\eta - A_i^\eta + z) \mathrm{d}z + O\left(\frac{B_i}{n^2} \right)}  \right]
.
\end{align*}
We can now take the limit as $n \to \infty$ and observe that
\begin{align*}
e^{-\eta^* x}\prod_{i=1}^k  \mathrm{E}_k\left[    e^{- \frac{\alpha}{n}B_i\int_0^x  \bar{G}(A_k^\eta - A_i^\eta + z) \mathrm{d}z + O\left(\frac{B_i}{n^2} \right)}  \right]
&
\longrightarrow
e^{-\eta^* x}\prod_{i=1}^k  \mathrm{E}_k\left[    e^{- M_i \int_0^x  \bar{G}(A_k^\eta - A_i^\eta + z) \mathrm{d}z }  \right]
,
\end{align*}
as we have that $\frac{\alpha}{n}B_1 \stackrel{D}{\Longrightarrow} M_1$ and $\frac{B_i}{n^2} \stackrel{D}{\Longrightarrow} 0$. This is now equal to the Hawkes process inter-arrival probability $\mathrm{P}_k\left(A_{k+1}^\lambda - A_{k}^\lambda > x\right)$. Hence by induction and total probability the arrival times converge, completing the first part of the proof.

For the second part of the proof, we now show that the dynamics of the processes converge when we condition on having the same fixed arrival times, which we now denote $\{A_i \mid i \in \mathbb{Z}^+\}$ for both processes. Since $N_t(n)$ is defined as the counting process of arrival epcochs rather than total number of arrivals, $N_t(n) = N_{t,\lambda}$ for all $n$ and all $t$. We now treat the intensity in two cases, the jump at arrivals and the dynamics between these times. For the first case, we take $k \in \mathbb{Z}^+$ and let $\lambda_{A_{k^-}} = \inf_{A_{k-1} \leq t < A_k} \lambda_t$ and $\eta_{A_{k^-}}(n) = \inf_{A_{k-1} \leq t < A_k} \eta_t(n)$ for all $n$, where $A_0 = 0$. Then, the jump in the $n$-GESEP intensity at the $k^\text{th}$ jump is such that
\begin{align*}
\eta_{A_k}(n) - \eta_{A_{k^-}}(n)
&
=
\frac{\alpha}{n}B_k
\stackrel{D}{\Longrightarrow}
M_k
=
\lambda_{A_k} - \lambda_{A_{k^-}}
,
\end{align*}
as $n \to \infty$. For the behavior between arrival times we first note that for $S_j$ independent and distributed with CDF $G(\cdot)$ for all $j \in \mathbb{Z}^+$,  the probability generating function of $\frac{1}{n}\sum_{j=1}^{B_1}\mathbf{1}\{y < S_j\}$ is
\begin{align*}
\E{z^{\frac{1}{n}\sum_{j=1}^{B_1}\mathbf{1}\{y < S_j\}}}
&
=
\E{\left(G(y) + \bar{G}(y) z^{\frac{1}{n}} \right)^{B_1}}
=
\E{\left(1 - \bar{G}(y) \left(1 - e^{\frac{1}{n}\log{z}}\right) \right)^{B_1}}
,
\end{align*}
and by a  Taylor expansion approach similar to what we used in the proof's first part, we can see that
\begin{align*}
\E{\left(1 - \bar{G}(y) \left(1 - e^{\frac{1}{n}\log{z}}\right) \right)^{B_1}}
&
=
\E{e^{- B_1\bar{G}(y) \left(1 - e^{\frac{1}{n}\log{z}}\right) + O\left(\frac{B_1}{n^2}\right)}}
=
\E{e^{\frac{1}{n}B_1\bar{G}(y) \log(z) + O\left(\frac{B_1}{n^2}\right)}}
.
\end{align*}
Taking the limit as $n \to \infty$, this yields
$
\E{z^{\frac{1}{n}\sum_{j=1}^{B_1}\mathbf{1}\{y < S_j\}}}
\longrightarrow
\E{z^{\bar{G}(y)M_1}}
$,
which is to say that
$$
\frac{1}{n}\sum_{j=1}^{B_1}\mathbf{1}\{y < S_j\} \stackrel{D}{\Longrightarrow} \bar{G}(y)M_1
.
$$
Using this, we can now see that for $k \in \mathbb{Z}^+$ and $0 \leq x < A_{k+1} - A_{k}$, the intensity of the $n$-GESEP satisfies
\begin{align*}
\eta_{A_k + x}(n)
&
=
\eta^* + \frac{\alpha}{n} \sum_{i=1}^k \sum_{j=1}^{B_i} \mathbf{1}\{A_k + x < A_i + S_{i,j} \}
\stackrel{D}{\Longrightarrow}
\eta^*
+
\sum_{i=1}^k M_i \bar{G}(A_k - A_i + x)
=
\lambda_{A_k + x}
,
\end{align*}
as $n \to \infty$. Thus, both the jump sizes of $\eta_t(n)$ and the behavior of $\eta_t(n)$ between jumps converge to that of $\lambda_t$, completing the proof.
\end{proof}
\end{theorem}

For an empirical demonstrate of this convergence, in Figure~\ref{probComp} we plot cumulative distribution functions for the intensities and counts of the $n$-GESEP across increasing batch sizes and multiple different expiration distributions and compare them to the empirical distributions of the corresponding general Hawkes processes. Specifically, we consider three different choices of $\bar{G}(\cdot)$, all with a unit mean expiration time: a standard exponential $\bar G(x) = e^{-x}$, a two-phase Erlang distribution $\bar G(x) = (1+2x)e^{-2x}$, and a two-phase hyper-exponential $\bar G(x) = (e^{-2x} + e^{-2/3 x})/2$. Each experiment is conducted for deterministic batch sizes in the $n$-GESEP models and deterministic marks in the Hawkes processes. Similar performance was observed for other distributions such as the geometric (or exponential in limiting form) and are thus omitted for length. In each of the three distribution settings with batch size $n = 8$ in the $n$-GESEP model, the distribution of the intensity is quite close to that of the Hawkes intensity, as can be seen in the left-hand column. On the right, one can note that the count distributions are quite close in all the various settings, and again the performance is particularly strong by $n = 8$. The near-indistinguishable visible closeness of these count distributions lends added importance to the calculations in Proposition~\ref{countPMF}. Given these observations, it is of particular interest to study the convergence rates of these models in future work.

 \begin{figure}
\begin{center}	
\text{Exponential}
\includegraphics[width=.5\textwidth]{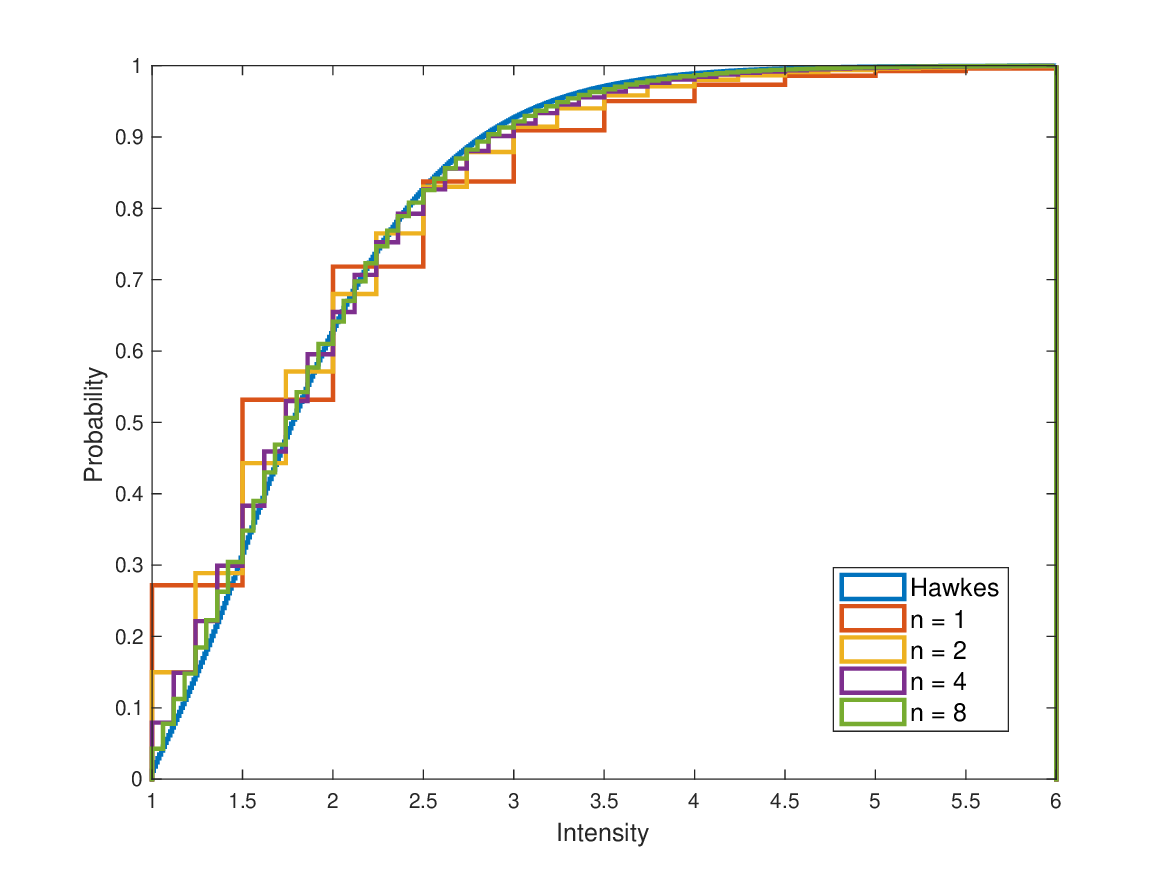}~\hspace{-.1in}~\includegraphics[width=.5\textwidth]{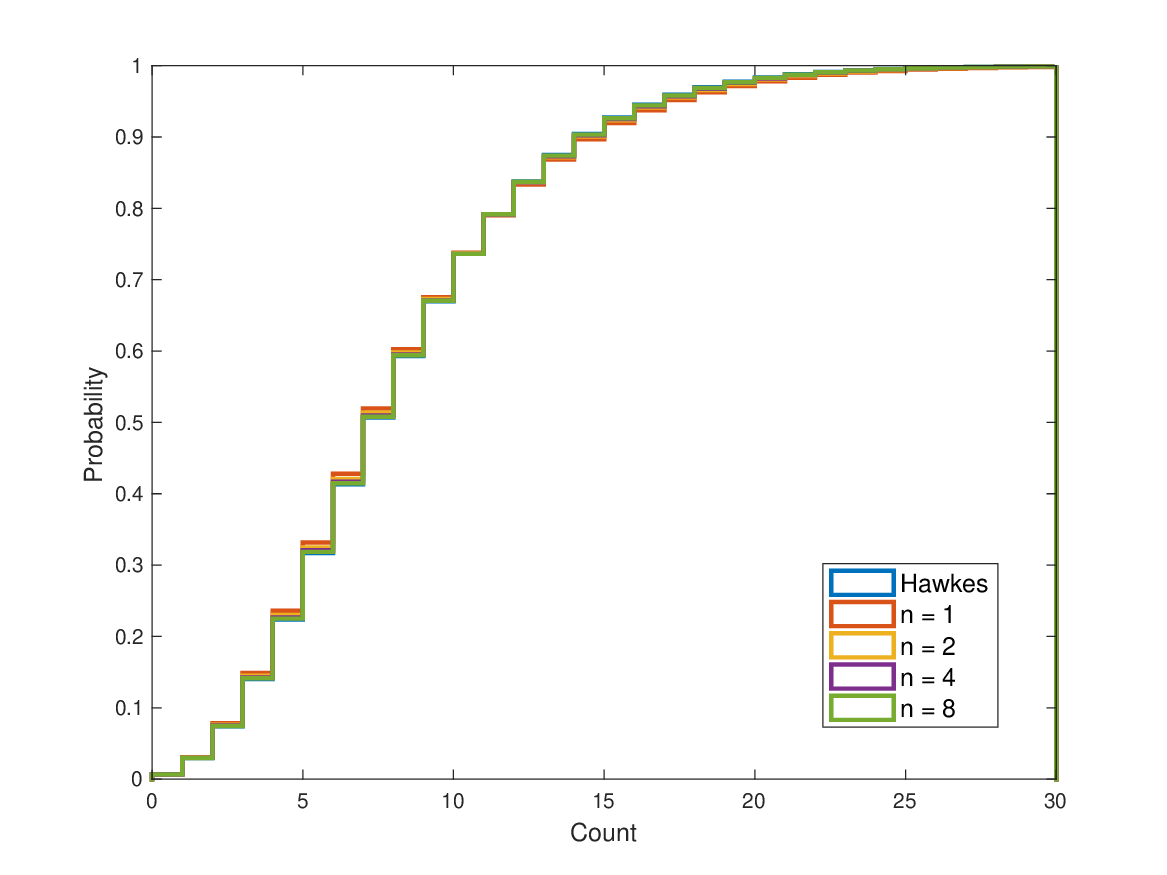}
\\
\text{Erlang}
\includegraphics[width=.5\textwidth]{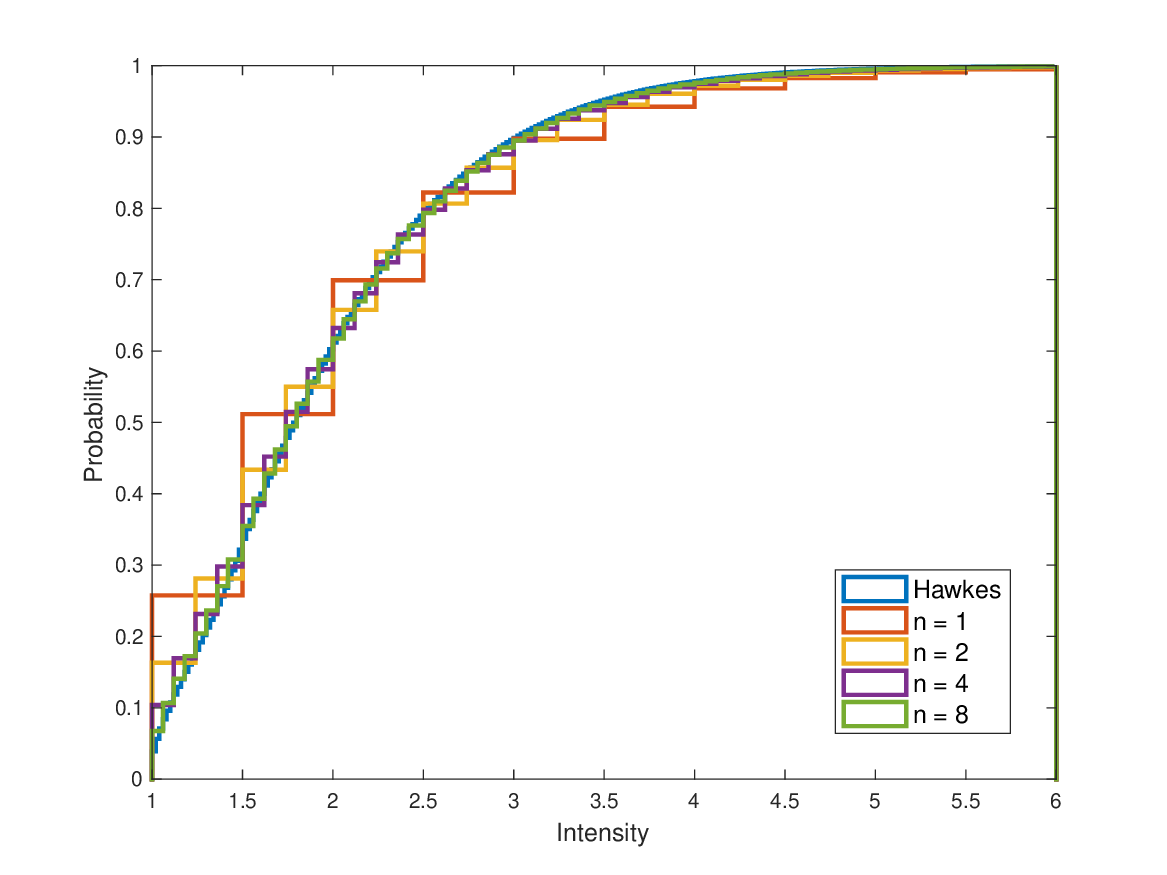}~\hspace{-.1in}~\includegraphics[width=.5\textwidth]{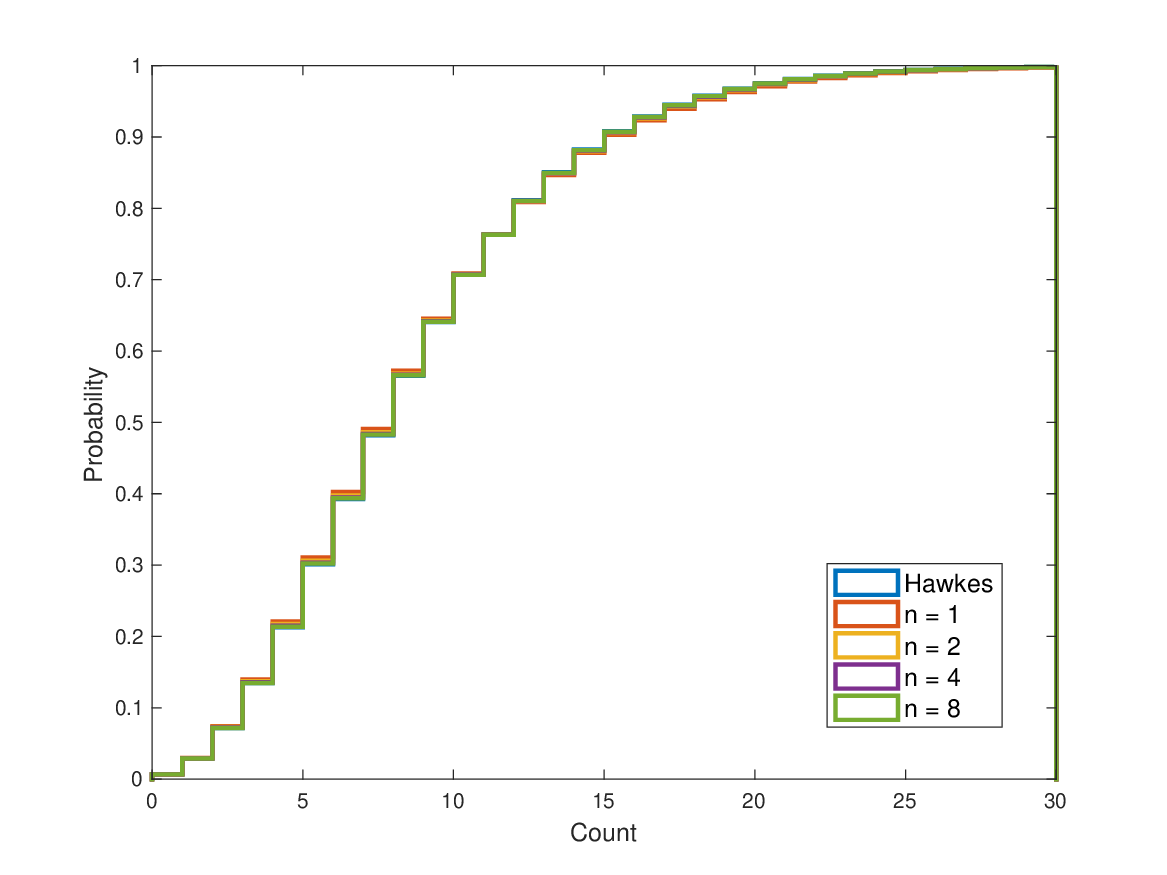}
\\
\text{Hyper-Exponential}
\includegraphics[width=.5\textwidth]{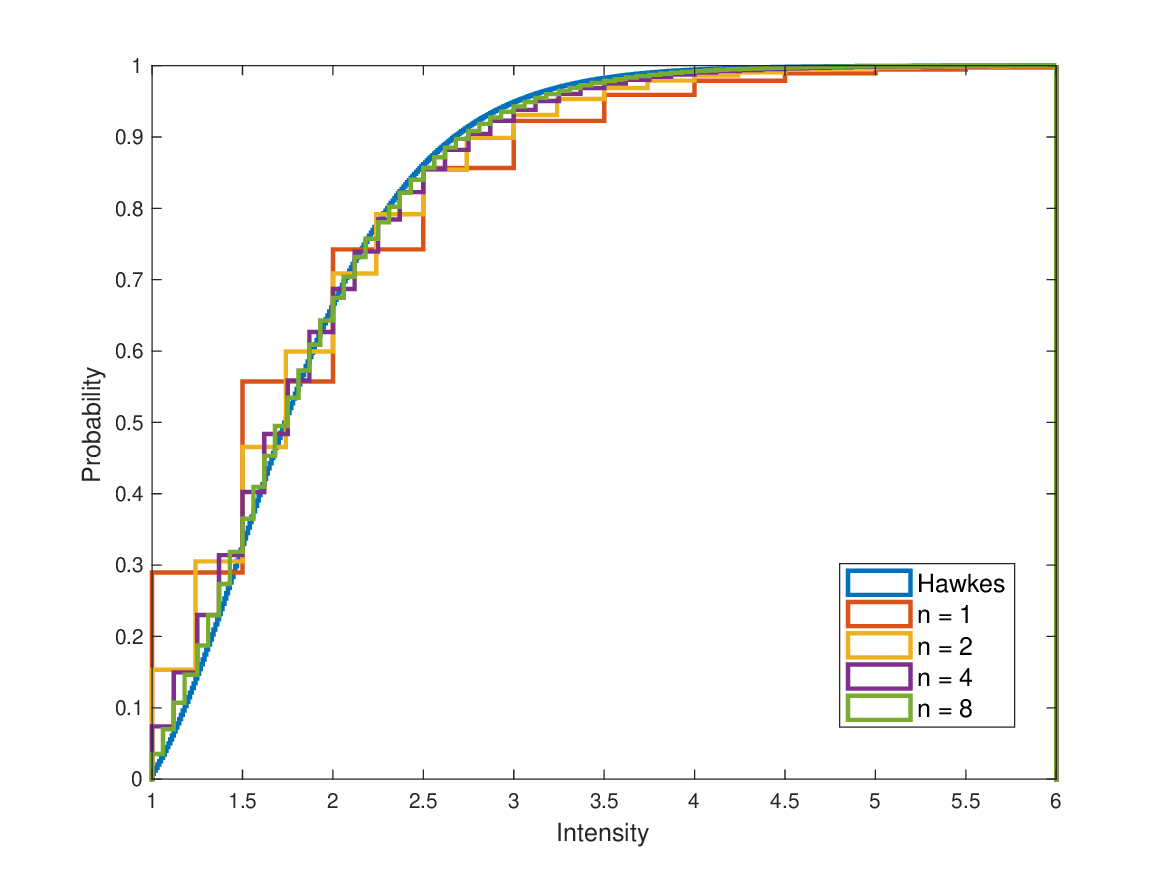}~\hspace{-.1in}~\includegraphics[width=.5\textwidth]{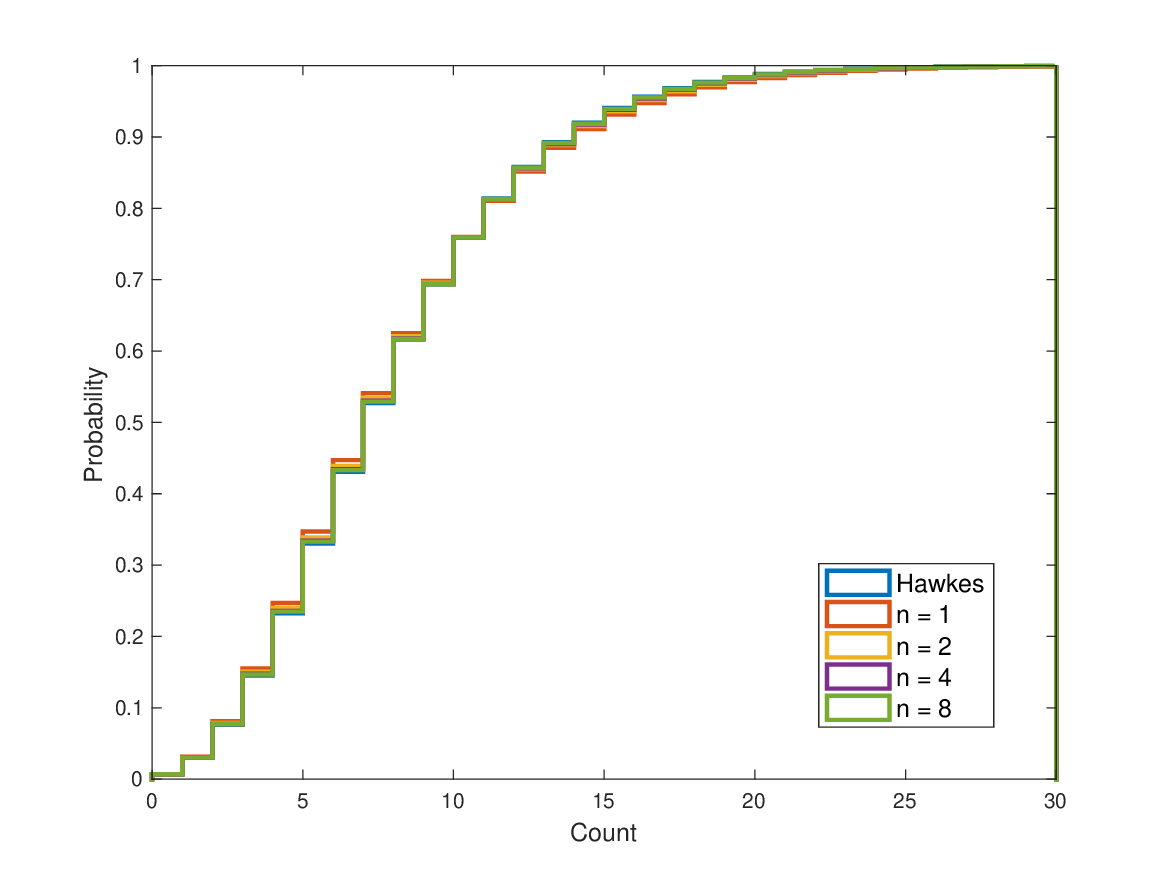}
\caption{Empirical transient distributions of the intensities (left) and counts (right) of $n$-GESEP with increasing batch size and for general Hawkes process simulations with $t = 5$, $\eta^* = 1$, $\alpha = 0.5$, and varying expiration distributions. These distributions are calculated from $2^{22}$ replications.} \label{probComp}
\end{center}
\end{figure}

As a reference, we list the components of the ephemerally self-exciting models and their corresponding limiting quantities in the general Hawkes process below in Table~\ref{convTable}. We can note that because the limiting excitation kernel given in  Theorem~\ref{hawkconv} is a complementary cumulative distribution function it is exclusively non-increasing, meaning that the excitement after each arrival immediately decays. It can be observed that this includes the two most popular excitation kernels, the exponential and power-law forms that we detailed in Subsection~\ref{subsecHawkes}. However, it does not include kernels that have a ``hump remote from the origin'' that Hawkes mentions briefly in the original paper \cite{hawkes1971spectra}. If desired, this can be remedied through extension to multi-phase service in the $n$-GESEP, with the intensity defined as an affine relationship with one of the later phases. 

\begin{table}[h]
$$
\begin{matrix}
 & n \longrightarrow \infty & \\
\hline
\text{Batch} & \Longrightarrow & \text{Mark}\\
\text{Expire} & \Longrightarrow & \text{Decay}\\
\hline
\text{ESEP} & \Longrightarrow & \text{Hawkes}
\end{matrix}
$$
\caption{Overview of convergence details in the batch-scaling of the ESEP.}\label{convTable}
\end{table}

Before concluding, let us remark that in addition to providing conceptual understanding into the Hawkes process itself, the alternate construction through the batch scaling in Theorem~\ref{hawkconv} is also of practical relevance in explaining the use of the Hawkes process in many application settings. For example, in biological applications such as the environmental management problem considered in \cite{gupta2018discrete}, one of the invasive species studied may produce multiple offspring simultaneously but only for the duration of its life cycle. That is, many species give birth in litters, creating batch arrivals, but of course only reproduce during their lifetime, yielding ephemerality. Furthermore, the numerical experiments in Figure~\ref{probComp} suggest that $n$ need not be overly large for the $n$-GESEP and the Hawkes process to be comparable in distribution.  

As another example, consider the spread of information on communication and social media platforms.  This setting has recently been a popular application of Hawkes processes, see e.g.~\cite{du2015dirichlet,farajtabar2017fake,halpin2013modelling,rizoiu2017expecting,rizoiu2018sir}.  When a user shares a post on these platforms, it is immediately and simultaneously dispatched to the real-time feeds of many other users, creating a batch increase of the response rate from the other users. The post then will typically only be seen on news feeds for a short period of time, as new content comes in to replace it. On top of this, social media administrators have been adopting a trend of intentionally introducing ephemerality into their platforms. For example, the expiration of posts and messages has been a defining feature of Snapchat since its inception.  Facebook and Instagram have recently adopted the same behavior with ``stories,'' and Twitter has responded in kind with the appropriately named ``fleets.'' Just as in the case of biological offspring processes, Theorem~\ref{hawkconv} offers an explanation of why the Hawkes process has become a popular and successful model in this space. Moreover, the insights from these connections can be deepened through the additional model relationships that we have discussed in Section~\ref{secRelate}.

\section{Conclusion}

Time is fleeting; excitement is ephemeral. In this paper we have introduced the \textit{ephemerally self-exciting process} (ESEP), a point process driven by the pieces of its history that remain presently active. That is, each arrival excites the arrival rate until the expiration of its randomly drawn activity duration, at which point its individual influence vanishes.  Throughout this work we have compared ephemeral self-excitement to eternal self-excitement through contrast with the well-known Hawkes process. These comparisons include an ordering of the moments of the two processes in Proposition~\ref{momentOrder} and through study of each process's branching structure in Subsection~\ref{subsecBranch}. We have also used the ESEP to relate ephemeral self-excitement to other well known stochastic models, including preferential attachment, random walks, and epidemics. Finally, we have also considered a generalized model with batch arrivals and general activity duration distributions, which we refer to as the \textit{$n^\text{th}$ general ephemerally self-exciting process} ($n$-GESEP). This $n$-GESEP model provides an alternate construction of general marked Hawkes processes through a batch scaling limit. ln Theorem~\ref{hawkconv} we prove that the $n$-GESEP model converges to a Hawkes process as its batch arrival size grows large, in which the limiting Hawkes process has an excitation kernel matching the tail CDF of the activity duration distribution and has marks given by the scaled limit of the batch sizes. As we have discussed, this limit both provides intuition for the occurrence of self-excitement in natural phenomena and relates the Hawkes process to the other stochastic models we connected to ephemeral self-excitement.

This presents many different directions for future research. First, we have frequently campaigned in this work that our results motivate the ESEP (and by extension, the $n$-GESEP) as a promising model for self-excitement in its own right. This follows from its tractability for analysis and its amenability to connection with other models. Because of this promise, we believe that further exploration and application of the ESEP holds great potential. Another natural and relevant avenue would be to continue to explore the connection between ephemeral self-excitement and the other stochastic models we have discussed. For example, one could study the connection between ESEP and epidemic models on more complex networks or with more complex dynamics. Doing so would give a point process representation for the times of infection in a more realistic epidemic model, which could be quite useful in practice for resource allocation and policy design. Similar deepened connections could also be pursued for other models such as preferential attachment.  In general, we believe the concept of ephemeral self-excitement merits further theoretical exploration and detailed empirical application, both of which we look forward to pursuing. Finally, to empower this model for full practical use, it is of great interest to study the estimation of the ESEP and $n$-GESEP processes. For example, likelihood-based estimation  could be pursued in a manner similar to the traditional Hawkes process, since the processes are also conditionally non-stationary Poisson when given the full history. In fact, similar estimation techniques may be achievable even if the data only contains arrival epochs (and not expiration times), perhaps through  missing data techniques such as Expectation-Maximization (EM) algorithms. These methods have applied successfully for both Hawkes processes and generally for branching processes (see, e.g.,~\citet{veen2008estimation,lewis2011nonparametric}), and thus also hold potential for the ephemerally self-exciting models we have studied here.

\section*{Acknowledgements}
We acknowledge the generous support of the National Science Foundation (NSF) for Andrew Daw's Graduate Research Fellowship under grant DGE-1650441, received during his doctoral studies at Cornell University when this research was initiated. Additionally, we are grateful for helpful discussions with Robert Hampshire, Emily Fischer, and Sam Gutekunst. Publicly available drafts of this work used the moniker ``Queue-Hawkes process,'' but we have re-titled  to avoid confusion and ambiguity. Alas, ``the process formerly known as the Queue-Hawkes'' is likely not an improvement in this regard.

\bibliographystyle{plainnat}
\bibliography{vanHawkes}

\appendix

\section{Lemmas and Auxiliaries}\label{secLemmas}

In this section of the appendix we give technical lemmas to support our analysis and brief auxiliary results that are of interest but not within the narrative of the body of this report. We begin by giving the infinitesimal generator form for time derivatives of the expectations of functions of our process. This is a valuable tool available to us because the ESEP is Markov, and it supports much of our analysis throughout this work.

\begin{lemma}\label{fubinifundQ}
For a sufficiently regular function $f:\left(\mathbb{R}^+ \times \mathbb{N}\right) \to \mathbb{R}$, the generator of the ESEP is given by
\begin{align}
\mathcal{L}f(\eta_t, N_t)
&
=
\underbrace{\sum_{i=1}^n \eta_t \left( f(\eta_t + \alpha, N_t + 1) - f(\eta_t, N_t) \right)}_{\mathrm{Arrivals}}
+
\underbrace{  \beta \left(\frac{\eta_t - \eta^*}{\alpha} \right)\left(f\left(\eta_t - \alpha,  N_t\right) - f(\eta_t,  N_t)\right)}_{\mathrm{Expirations}}
.
\end{align}
Then, the time derivative of the expectation of $f(\eta_t, N_t)$ is given by
\begin{equation}\label{generator}
\frac{\mathrm{d}}{\mathrm{d}t} \E{f(\eta_t, N_t)}
=
\E{\mathcal{L}f(\eta_t, N_t)}
\end{equation}
for all $t \geq 0$.
\begin{proof} This is a direct result of the ESEP belonging to the family of piece-wise deterministic Markov processes, as defined in \citet{davis1984piecewise}. Moreover, the specific regularity conditions are given in Theorem 5.5 of that work.
\end{proof}
\end{lemma}

Note that this is also immediately applicable to the active number in system process perspective of the ESEP, as $Q_t = (\eta_t - \eta^*)\slash \alpha$. Thus, we will leverage this infinitesimal generator for studying each of $\eta_t$, $Q_t$, and $N_t$ throughout both the main body of the text and these appendices.


As another supporting lemma, let us summarize a result that can be used with the infinitesimal generator to relate two different Markov processes. Throughout this work we make comparisons between different processes, in particular between the ESEP and the Hawkes process. One way that we do this is to investigate the differential equations found with use of Lemma~\ref{fubinifundQ}. In Lemma~\ref{complemma} we provide the method by which we make such comparisons.

\begin{lemma}[A Comparison Lemma]\label{complemma}
Let $f: \mathbb{R}^2 \to \mathbb{R}$ be a continuous function in both variables.  If we assume that initial value problem
\begin{equation}
\frac{\mathrm{d}x(t)}{\mathrm{d}t} = f(t, x(t)), \ x(0) = x_0
\end{equation}
has a unique solution for the time interval $[0,T]$ and
\begin{equation}
\frac{\mathrm{d}y(t)}{\mathrm{d}t} \leq f(t, y(t)) \quad \mathrm{for} \ t \in [0,T] \ \mathrm{and} \ y(0) \leq x_0
\end{equation}
then $ x(t) \geq y(t) $ for all $t \in [0,T]$.
\begin{proof}
The the proof of this result is given in \citet{hale2013introduction}.
\end{proof}
\end{lemma}

For a result that is both auxillary on the surface and beneficial in proofs, in Proposition~\ref{QDpgf} we give the probability generating function for the number in system and the number of departures, or expirations, in the ESEP. The departure process is largely outside of the scope of this work, but this result is instrumental in the proof of the probability generating function for the counting process in Proposition~\ref{Npgf}, which is given in Appendix~\ref{pgfProof}.

\begin{proposition}\label{QDpgf}
Let $Q_t$ be the active number in system at time $t \geq 0$ of an ESEP with baseline intensity $\eta^* > 0$, intensity jump size $\alpha > 0$, and expiration rate $\beta > \alpha$. Then, let $D_t$ be the number of arrivals by time $t$ that are no longer active. Then, the joint probability generating function of $Q_t$ and $D_t$, denoted $G(z_1, z_2, t) \equiv \E{z_1^{Q_t} z_2^{D_t}}$, is given by
\begin{align}
&
G(z_1, z_2, t)
=
z_2^{D_0} e^{\frac{\eta^*(\beta-\alpha)}{2\alpha}t}
\left(
1
-
\left(
\tanh\left(
\frac{t}2 \sqrt{(\beta+\alpha)^2 - 4\alpha\beta z_2}
+
\tanh^{-1}\left(\frac{\beta+\alpha-2\alpha z_1 }{\sqrt{(\beta+\alpha)^2 - 4\alpha\beta z_2}} \right)
\right)
\right)^2
\right)^{\frac{\eta^*}{2\alpha}}
\nonumber
\\
&
\quad
\cdot
\left(
\frac{\beta + \alpha}{2\alpha}
-
\frac{\sqrt{(\beta+\alpha)^2 - 4\alpha\beta z_2} }{2\alpha}
\tanh \left(\frac{t}2 \sqrt{(\beta+\alpha)^2 - 4\alpha\beta z_2} + \tanh^{-1}\left(\frac{ \beta+\alpha-2\alpha z_1}{\sqrt{(\beta+\alpha)^2 - 4\alpha\beta z_2}} \right) \right)
\right)^{Q_0}
\nonumber
\\
&
\quad
\cdot
\left(
\cosh \left(
\tanh^{-1}\left(\frac{2\alpha z_1 - \beta - \alpha}{\sqrt{(\beta+\alpha)^2 - 4\alpha\beta z_2}} \right)
\right)
\right)^{\frac{\eta^*}{\alpha}}
,
\end{align}
where $Q_0$ and $D_0$ are the active number in the system and the count of departures at time 0, respectively.
\begin{proof}
We will show this through the method of characteristics. We can first observe through Lemma~\ref{fubinifundQ} that
\begin{align*}
\frac{\mathrm{d}}{\mathrm{d}t}\E{z_1^{Q_t} z_2^{D_t}}
&
=
\E{(\eta^* + \alpha Q_t)(z_1-1)z_1^{Q_t}z_2^{D_t} + \beta Q_t \left(\frac{z_2}{z_1}-1\right)z_1^{Q_t}z_2^{D_t}},
\end{align*}
and so $G(z_1,z_2,t)$ is given by the following partial differential equation:
\begin{align*}
\frac{\partial}{\partial t} G(z_1,z_2,t)
+
\left(
\alpha (z_1-z_1^2) + \beta (z_1-z_2)
\right)
\frac{\partial}{\partial z_1} G(z_1,z_2,t)
&
=
\eta^* (z_1 - 1) G(z_1,z_2,t)
.
\end{align*}
To simplify our analysis, we will instead solve for $\log(G(z_1,z_2,t))$, which through the chain rule will by given by the solution to the partial differential equation expressed
\begin{align*}
\frac{\partial}{\partial t} \log(G(z_1,z_2,t))
+
\left(
\alpha (z_1-z_1^2) + \beta (z_1-z_2)
\right)
\frac{\partial}{\partial z_1} \log(G(z_1,z_2,t))
&
=
\eta^* (z_1 - 1)
,
\end{align*}
with initial condition $\log(G(z_1,z_2,0)) = \log(z_1^{Q_0}z_2^{D_0})$. This now gives us the characteristic equations as follows:
\begin{align*}
\frac{\mathrm{d}z_1}{\mathrm{d}s}(r,s)
&
=
\alpha (z_1-z_1^2) + \beta (z_1-z_2)
,
&&
z_1(r,0) = r
\\
\frac{\mathrm{d}t}{\mathrm{d}s}(r,s)
&
=
1
,
&&
t(r,0) = 0
\\
\frac{\mathrm{d}g}{\mathrm{d}s}(r,s)
&
=
\eta^* ( z_1 - 1)
,
&&
g(r,0) = \log(r^{Q_0}z_2^{D_0})
.
\end{align*}
Solving the first two equations we see that
\begin{align*}
z_1(r,s)
&
=
\frac{\beta + \alpha}{2\alpha}
+
\frac{\sqrt{(\beta+\alpha)^2-4\alpha\beta z_2}}{2\alpha}
\tanh\left(
\frac{s}{2} \sqrt{(\beta+\alpha)^2-4\alpha\beta z_2} - \tanh^{-1}\left(\frac{\beta+\alpha-2\alpha r}{\sqrt{(\beta+\alpha)^2-4\alpha\beta z_2}}\right)
\right)
\\
t(r,s)
&
=
s
,
\end{align*}
which allows us to now solve for $g(r,s)$. Using the solution to $z_1(r,s)$, the ordinary differential equation for $g(r,s)$ is given by
\begin{align*}
\frac{\mathrm{d}g}{\mathrm{d}s}(r,s)
&
=
\frac{\eta^*\sqrt{(\beta+\alpha)^2-4\alpha\beta z_2}}{2\alpha}
\tanh\Bigg(
\frac{s}{2} \sqrt{(\beta+\alpha)^2-4\alpha\beta z_2}
-
\tanh^{-1}\left(\frac{\beta+\alpha-2\alpha r}{\sqrt{(\beta+\alpha)^2-4\alpha\beta z_2}}\right)
\Bigg)
\\
&
\quad
+
\frac{\eta^*(\beta - \alpha)}{2\alpha}
,
\end{align*}
which yields a solution of
\begin{align*}
g(r,s)
&
=
\log(r^{Q_0} z_2^{D_0})
+
\frac{\eta^*(\beta-\alpha)}{2\alpha}s
+
\frac{\eta^*}{2\alpha}\log\left(
1
-
\frac{(\beta+\alpha - 2\alpha r)^2}{(\beta+\alpha)^2-4\alpha\beta z_2}
\right)
\\
&
\quad
+
\frac{\eta^*}{\alpha}
\log\left(
\cosh\left(
\frac{s}{2}\sqrt{(\beta+\alpha)^2 - 4\alpha\beta z_2}
-
\tanh^{-1}\left(\frac{\beta+\alpha-2\alpha r}{\sqrt{(\beta+\alpha)^2-4\alpha\beta z_2}}\right)
\right)
\right)
.
\end{align*}
Now, from these equations we can express the characteristics variables in terms of the original arguments as $s = t$ and
\begin{align*}
r
&
=
\frac{\beta+\alpha}{2\alpha}
-
\frac{\sqrt{(\beta+\alpha)^2 - 4\alpha\beta z_2} }{2\alpha}
\tanh \left(
\frac{t}{2}\sqrt{(\beta+\alpha)^2-4\alpha\beta z_2}
-
\tanh^{-1}\left(\frac{2\alpha z_1 -\beta-\alpha}{\sqrt{(\beta+\alpha)^2-4\alpha\beta z_2}}\right)
\right)
.
\end{align*}
Then, by performing the substitution
$
G(z_1,z_2,t)
=
e^{g(r(z_1,z_2,t),s(z_1,z_2,t))}
$
and simplifying, we achieve the stated result.
\end{proof}
\end{proposition}

As another auxiliary result, in Proposition~\ref{batchQ2} we give the steady-state moment generating function for the $2$-GESEP with exponentially distribution activity durations and deterministic batch sizes, meaning pairs of arrivals.

\begin{proposition} \label{batchQ2}
Consider the following 2-GESEP: arrivals occur at rate $\eta_t(2) = \eta^* + \frac{\alpha}{2}  Q_t(2)$, where $Q_t(2)$ receives arrivals batches of size $2$. Each activity duration is independent and exponentially distributed with rate $\beta > \alpha > 0$. Then, the steady-state moment generating function of $Q_t(2)$ is given by
\begin{align}\label{batchQ2eq}
&
\E{e^{\theta Q_\infty (2)}}
\equiv
\lim_{t\to\infty}
\E{e^{\theta Q_t (2)}}
\nonumber
\\
&
=
\exp \left(
\frac{
2\eta^*
}
{
 \sqrt{\alpha  (\alpha +8 \beta )}
}
  \left(
  \tanh^{-1}\left(\left(2 e^\theta +1\right) \sqrt{\frac{\alpha }{\alpha +8 \beta
   }}\right)
-
\tanh ^{-1}\left(
3 \sqrt{\frac{\alpha }{\alpha +8\beta }}
\right)
\right)
\right)
\left(
\frac{2\beta - 2\alpha}{2\beta - \alpha (e^\theta + e^{2\theta})}
\right)^{\frac{\eta^*}{\alpha}}
\end{align}
\begin{proof}
Using Lemma~\ref{fubinifundQ}, we see that the moment generating function will be given by the solution to
\begin{align*}
\frac{\mathrm{d}}{\mathrm{d}t}
\E{e^{\theta  Q_t (2)}}
&=
\E{
\left(
\eta^* + \frac{\alpha  Q_t (2)}{2}
\right)
\left(
e^{\theta ( Q_t (2)+2)}
-
e^{\theta  Q_t (2)}
\right)
+
\beta  Q_t(2)
\left(
e^{\theta ( Q_t (2)-1)}
-
e^{\theta  Q_t (2)}
\right)
}
,
\end{align*}
which can be equivalently expressed in PDE form as
\begin{align*}
\frac{\partial}{\partial t}
\mathcal{M}_2(\theta, t)
&=
\eta^*
\left(e^{2\theta} - 1\right)
\mathcal{M}_2(\theta, t)
+
\left(
\frac{\alpha}{2}
\left(e^{2\theta} - 1\right)
+
\beta
\left(e^{-\theta} - 1\right)
\right)
\frac{\partial}{\partial \theta}
\mathcal{M}_2(\theta, t)
,
\end{align*}
where $\mathcal{M}_2(\theta, t) = \E{e^{\theta  Q_t (2)}}$. To solve for the steady-state moment generating function we consider the ODE given by
\begin{align*}
\frac{\mathrm{d}}{\mathrm{d} \theta}
\mathcal{M}_2(\theta, \infty)
&=
\frac{
\eta^*
\left(1 - e^{2\theta} \right)
\mathcal{M}_2(\theta, \infty)
}
{
\frac{\alpha}{2}
\left(e^{2\theta} - 1\right)
+
\beta
\left(e^{-\theta} - 1\right)
}
,
\end{align*}
with the initial condition that $\mathcal{M}_2(0, \infty) = 1$. Through taking the derivative of the expression in Equation~\ref{batchQ2eq}, we verify the result.
\end{proof}
\end{proposition}

\section{Exploring a Hybrid Self-Exciting Model}\label{secHESEP}

In the main body of the text we have defined the ESEP, a model that features arrivals that self-excite but only for a finite period of time. By comparison to the traditional Hawkes process models for self-excitement, the effect from one arrival does not decay through time but is fixed at a constant value for as long as it remains active. In this way, the ESEP features ephemeral but piecewise constant self-excitement whereas the Hawkes process has eternal but ever decreasing self-excitement. One can note though that ephemeral self-excitement need not be piecewise constant. A model could feature both decay and down-jumps as manners of regulating its self-excitement. In this section of the appendix, we will consider such a model, specifically a Markovian one. To begin, let us now define the \textit{hybrid ephemerally self-exciting process} (HESEP) in Defintion~\ref{qhDef}.

\begin{definition}[Hybrid ephemerally self-exciting process]\label{qhDef}
Let $t \geq 0$ and suppose that $\nu^* > 0$, $\alpha > 0$, $\beta \geq 0$, and $\mu \geq 0$. Then, define $\nu_t$, $N_{t,\nu}$, and $Q_{t,\nu}$ such that:
\begin{enumerate}[i)]
\item $N_{t,\nu}$ is an arrival process driven by the intensity $\nu_t$,
\item $Q_{t,\nu}$ is the number of arrivals from $N_{t,\nu}$ that have not yet expired according to their i.i.d.~Exp($\mu$) activity durations,
\item $\nu_t$ is governed by
\begin{align*}
\mathrm{d}\nu_t
=
\beta(\nu^* - \nu_t)\mathrm{d}t
+
\alpha \mathrm{d}N_{t,\nu}
-
\frac{\nu_t - \nu^*}{Q_{t,\nu}}
\mathrm{d}D_{t,\nu}
\end{align*}
where $D_{t,\nu} = N_{t,\nu} - Q_{t,\nu}$.
\end{enumerate}
Then, we say that the intensity-queue-counting process triplet $(\nu_t, Q_{t,\nu}, N_{t,\nu})$  is a \textit{hybrid ephemerally self-exciting process} (HESEP) with baseline intensity $\nu^*$, intensity jump size $\alpha$, decay rate $\beta$, expiration rate $\mu$, and initial values $(\nu_0, Q_0^\nu, N_0^\nu)$.
\end{definition}

By definition, one can view the HESEP as a hybrid between the ESEP and Hawkes process models. If $\beta = 0$ then we recover the ESEP; if $\mu = 0$ then we recover the Hawkes process. In this way, much of the dynamics are quite familiar: up-jumps of size $\alpha$ at teach arrival, exponential decay between events at rate $\beta$, and down-jumps upon each activity duration expiration. Perhaps the least intuitive part of this definition is the size of the down-jump, as this depends on the current levels of the intensity and the number of active exciters in the system. This draws inspiration from Markovian infinite server queues. In an $M/M/\infty$, all jobs currently in the system are equally likely to be the next to leave, regardless of the order they entered the system. Similarly, in the HESEP, each exciter in the system is equally likely to be the next to leave. Moreover, the down-jump size is the same regardless of which exciter is next to leave. When an expiration of an activity durations means that there are no longer any presently active exciters, by definition the intensity will return to the baseline value. One can note that this down-jump size is actually always bounded below by 0 since the intensity decays down towards the baseline $\nu^*$ and bounded above by $\alpha$ since $\nu_t - \nu^* \leq \alpha Q_{t,\nu}$ due to the fact that each arrival increases $\nu_t$ by $\alpha$ before it decays. As a quick interesting fact regarding this process, in Proposition~\ref{nodowndownjumps} we show that the size of a downjump, $(\nu_t - \nu^*)\slash Q_{t,\nu}$, does not have down-jumps itself.

\begin{proposition}\label{nodowndownjumps}
Let $\phi_t = \frac{\nu_t - \nu^*}{Q_{t,\nu}}$ be the size of a down-jump occurring at time $t \geq 0$. Suppose that $b \geq a \geq 0$ is such that $Q_{t,\nu}$ is positive for all $t \in [a,b]$. Then, the $\phi_t$ has no downward jumps on $[a,b]$.
\begin{proof}
Suppose that $[a,b]$ is such as interval, and then for $t \in [a,b]$  we note that
\begin{align*}
\frac{\nu_t - \frac{\nu_t - \nu^*}{Q_{t,\nu}} - \nu^*}{Q_{t,\nu} - 1}
=
\frac{Q_t\nu_t - \nu_t + \nu^* - Q_{t,\nu}\nu^*}{Q_{t,\nu}(Q_{t,\nu} - 1)}
=
\frac{\nu_t - \nu^*}{Q_{t,\nu}},
\end{align*}
and this is equal to $\phi_t$.
\end{proof}
\end{proposition}

\subsection{Sandwiching the Hybrid Model}

As one might expect for a so-called ``hybrid'' model, the HESEP can be connected to the ESEP and the Hawkes process in many different ways. In Proposition~\ref{intenOrder}, we show that one can actual sandwich this model between its two extremes. That is, we show that the means of the processes are equal when given the same parameters, whereas the variances are ordered with Hawkes as the smallest and ESEP as the largest.

\begin{proposition}\label{intenOrder}
Let $\alpha > 0$, $\beta > 0$, and $\mu > 0$ be such that $\mu + \beta > \alpha > 0$. Additionally, let $\nu^* > 0$. Let $\nu_t$ be an HESEP with baseline intensity $\nu^*$, intensity jump size $\alpha$, decay rate $\beta$, and service rate $\mu$. Similarly, let $\lambda_t$ be the intensity of a Hawkes process with baseline intensity $\nu^*$, intensity jump $\alpha$, and decay rate $\mu + \beta$. Finally, let $\eta_t$ be the intensity of an ESEP with baseline intensity $\nu^*$, intensity jump $\alpha$, and service rate $\mu + \beta$, then the means of these process intensities are all equal:
\begin{align}
\E{\lambda_t} = \E{\nu_t} = \E{\eta_t}.
\end{align}
Furthermore, the process variances are ordered such that
\begin{align}
\Var{\lambda_t} \leq \Var{\nu_t} \leq \Var{\eta_t}.
\end{align}
Additionally, let $N_{t,\nu}$, $N_{t,\lambda}$, and $N_{t,\eta}$ be the counting processes of the HESEP, Hawkes process, and ESEP, respectively. Then, the means of these counting process are equal
\begin{align}
\E{N_{t,\lambda}} = \E{N_{t,\lambda}} = \E{N_{t,\eta}},
\end{align}
and the variances of these counting processes are again ordered such that
\begin{align}
\Var{N_{t,\lambda}} \leq \Var{N_{t,\nu}} \leq \Var{N_{t,\eta}}.
\end{align}
Finally, the covariances among each intensity and counting process pair are likewise ordered such that
\begin{align}
\Cov{\lambda_t, N_{t,\lambda}} \leq \Cov{\nu_t, N_{t,\nu}} \leq \Cov{\eta_t, N_{t,\eta}},
\end{align}
where $t \geq 0$ and where all intensities have the same initial value.
\begin{proof}
By a quick check of the differential equations for each mean, we can directly observe that $\E{\nu_t} = \E{\eta_t} = \E{\lambda_t}$. To show the variance ordering we begin by considering the ODE for the second moment of $\nu_t$:
\begin{align*}
\frac{\mathrm{d}}{\mathrm{d}t}\E{\nu_t^2}
&
=
2\beta
\left(
\nu^* \E{\nu_t}
-
\E{\nu_t^2}
\right)
+
\alpha^2 \E{\nu_t}
+
2\alpha \E{\nu_t^2}
+
\mu \E{\left(\left(\nu_t - \frac{\nu_t - \nu^*}{Q_{t,\nu}}\right)^2 - \nu_t^2\right)Q_{t,\nu}}
.
\end{align*}
Now, let's observe that
\begin{align*}
\E{\left(\left(\nu_t - \frac{\nu_t - \nu^*}{Q_{t,\nu}}\right)^2 - \nu_t^2\right)Q_{t,\nu}}
&
=
2\left(\nu^* \E{\nu_t} - \E{\nu_t^2}\right) + \E{\frac{(\nu_t - \nu^*)^2}{Q_{t,\nu}}}
,
\end{align*}
which follows by expanding the squared term. Because $\frac{(\nu_t - \nu^*)^2}{Q_{t,\nu}} \geq 0$, this gives us that
\begin{align}
\E{\left(\left(\nu_t - \frac{\nu_t - \nu^*}{Q_{t,\nu}}\right)^2 - \nu_t^2\right)Q_{t,\nu}}
&
\geq
2\left(\nu^* \E{\nu_t} - \E{\nu_t^2}\right)
.
\label{2ineq}
\end{align}
This inequality now allows us to directly compare $\frac{\mathrm{d}}{\mathrm{d}t}\E{\nu_t^2}$ to $\frac{\mathrm{d}}{\mathrm{d}t}\E{\lambda_t^2}$ and $\frac{\mathrm{d}}{\mathrm{d}t}\E{\eta_t^2}$. First, we can use Equation~\ref{2ineq} to see that
\begin{align*}
\frac{\mathrm{d}}{\mathrm{d}t}\E{\nu_t^2}
&
=
2\beta
\left(
\nu^* \E{\nu_t}
-
\E{\nu_t^2}
\right)
+
\alpha^2 \E{\nu_t}
+
2\alpha \E{\nu_t^2}
+
\mu \E{\left(\left(\nu_t - \frac{\nu_t - \nu^*}{Q_{t,\nu}}\right)^2 - \nu_t^2\right)Q_{t,\nu}}
\\
&
\geq
2(\mu+\beta)
\left(
\nu^* \E{\nu_t}
-
\E{\nu_t^2}
\right)
+
\alpha^2 \E{\nu_t}
+
2\alpha \E{\nu_t^2}
.
\end{align*}
Because we have already shown that $\E{\lambda_t} = \E{\nu_t}$, we see that $\frac{\mathrm{d}}{\mathrm{d}t}\E{\lambda_t^2} \leq \frac{\mathrm{d}}{\mathrm{d}t}\E{\nu_t^2}$ when evaluated at the same point and thus by Lemma~\ref{complemma}, $\Var{\lambda_t} \leq \Var{\nu_t}$. By analogous arguments for $\eta_t$, we achieve the stated result. For the counting process means, we can now observe that all the differential equations are such that
$$
\frac{\mathrm{d}}{\mathrm{d}t}\E{N_{t,\lambda}} = \E{\lambda_t}
= \frac{\mathrm{d}}{\mathrm{d}t}\E{N_{t,\nu}} = \E{\nu_t}
= \frac{\mathrm{d}}{\mathrm{d}t}\E{N_{t,\eta}} = \E{\eta_t}
.
$$
We assume that all counting processes start at 0 and thus we have that the counting process means are equal throughout time. This also implies that the products of means, $\E{\lambda_t}\E{N_{t,\lambda}}$, $\E{\nu_t}\E{N_{t,\nu}}$, and $\E{\eta_t}\E{N_{t,\eta}}$, are equal. Hence to show the ordering of the covariances we will focus solely on the expectations of the products. This differential equation is given by
\begin{align*}
\frac{\mathrm{d}}{\mathrm{d}t}\E{\nu_t N_{t,\nu}}
&
=
-(\mu + \beta - \alpha)\E{\nu_t N_{t,\nu}} + (\mu+\beta)\nu^*\E{N_{t,\nu}} + \alpha\E{\nu_t} + \E{\nu_t^2}
,
\end{align*}
and we can note that the coefficients are the same for each of the processes. Not including the function for which we want to solve, $\E{\nu_t N_{t,\nu}}$, we can also observe that every function is equivalent across the processes other than the second moments of the intensities. We have shown that these second moments are in fact ordered and therefore we receive the stated ordering of the covariances. Finally, we observe that the differential equation for the second moment of each counting process is of the form
$$
\frac{\mathrm{d}}{\mathrm{d}t}\E{(N_{t,\nu})^2} = \E{\nu_t} + 2\E{\nu_t N_{t,\nu}}
.
$$
From the ordering of the covariances and the equivalences of the means, we can conclude the proof.
\end{proof}
\end{proposition}

As a simple consequence of Propositon~\ref{intenOrder}, we can note that because the Hawkes process is over-dispersed, i.e. its variance is larger than its mean, so too are the ESEP and HESEP. One can note that these bounds on the variance of the HESEP are not only useful for comparison but also practical for the study of the HESEP itself, as the differential equations for the variance via the infinitesimal generator is not closed.


\subsection{Strong Convergence of the HESEP Counting Process}\label{subsecStrong}

In the final three subsections of analysis of the HESEP, we obtain a trio of limiting results. We begin with the almost sure convergence of the ratio of the HESEP counting process and time, which is an elementary renewal result in the style of \citet{blackwell1948renewal} or \citet{lindvall1977probabilistic}, for example. However, by comparison to the context of such works, we can bound the mean and variance of the HESEP via the ESEP and we are instead solely interested in establishing the convergence, as we will obtain additional results by consequence. Using these expressions for the first two moments, the proof of Theorem~\ref{bhthm} follows standard approaches using the Borel-Cantelli lemma. In Corollary~\ref{qhSLLN} we use this renewal result to find a strong law of large numbers for the dependent and non-identically distributed inter-arrival times of the HESEP by way of the continuous mapping theorem, which is another standard technique.

\begin{theorem}\label{bhthm}
Let $(\nu_t, Q_{t,\nu}, N_{t,\nu})$ be a HESEP with baseline intensity $\nu^*$, intensity jump $\alpha > 0$, intensity decay rate $\beta \geq 0$, and rate of exponentially distributed service $\mu \geq 0$, where $\mu + \beta > \alpha$. Then,
\begin{align}
 \frac{N_{t,\nu}}{t}
\stackrel{\mathrm{a.s.}}{\longrightarrow}
\nu_\infty
\end{align}
as $t \to \infty$, where $\nu_\infty = \frac{(\mu + \beta) \nu^*}{\mu + \beta - \alpha}$.
\begin{proof}
We will show this through use of the Borel-Cantelli Lemma. Let $\epsilon > 0$ be arbitrary and define the event $E_s$ for $s \in \mathbb{N}$ as
$$
E_s = \left\{\sup_{t \in (s^2, (s+1)^2]} \frac{|N_{t,\nu} - \E{N_{t,\nu}}|}{t} > \epsilon \right\}.
$$
We now note that $N_{t,\nu} - \E{N_{t,\nu}}$ is a martingale by definition, and so $|N_{t,\nu} - \E{N_{t,\nu}}|$ is a sub-martingale. Additionally, we can observe that
$$
\PP{E_s}
\leq
\PP{\sup_{t \in (s^2, (s+1)^2]} |N_{t,\nu} - \E{N_{t,\nu}}| > s^2 \epsilon}
$$
because $s^2 \leq t$ for any $t$. By Doob's martingale inequality, we have
$$
\PP{\sup_{t \in (s^2, (s+1)^2]} |N_{t,\nu} - \E{N_{t,\nu}}| > s^2 \epsilon}
\leq
\frac{\E{|N_{(s+1)^2,\nu} - \E{N_{(s+1)^2,\nu}}|^2}}{s^4 \epsilon^2}
=
\frac{\Var{N_{(s+1)^2,\nu}}}{s^4 \epsilon^2}
.
$$
From Proposition~\ref{intenOrder}, we note that the variance of an HESEP counting process with baseline intensity $\nu^*$, intensity jump size $\alpha$, decay rate $\beta$, and service rate $\mu$ is upper-bounded by the variance of an ESEP counting process with baseline $\nu^*$, jump size $\alpha$, and service rate $\mu + \beta$. Using the explicit form of the ESEP counting process variance as computed through Lemma~\ref{fubinifundQ}, we have the bound
\begin{align*}
&
\Var{N_{(s+1)^2,\nu}}
\leq
\Var{N_{(s+1)^2,\eta}}
=
\frac{((\mu+\beta)^2 + \alpha^2)\nu_\infty}{(\mu+\beta-\alpha)^2}(s+1)^2
-
\frac{2\alpha\mu(\nu_0 - \nu_\infty)}{(\mu+\beta-\alpha)^3}
\bigg(
e^{-(\mu+\beta - \alpha)(s+1)^2}
\\
&
\quad
+ (\mu+\beta-\alpha)(s+1)^2 e^{-(\mu+\beta -\alpha)(s+1)^2}
\bigg)
+
\left(
\frac{\nu_0 - \nu_\infty}{\mu+\beta-\alpha}
-
\frac{\alpha\mu \nu_\infty}{(\mu+\beta-\alpha)^3}
-
\frac{(\alpha^2 + \alpha(\mu+\beta))\nu_0}{(\mu+\beta-\alpha)^3}
\right)
\\
&
\quad
\cdot
\left(
1 - e^{-(\mu+\beta-\alpha)(s+1)^2}
\right)
\nonumber
+
\left(
\frac{(\alpha^2 + \alpha(\mu+\beta))\nu_0}{2(\mu+\beta-\alpha)^3} - \frac{\alpha (\mu+\beta)\nu_\infty}{2(\mu+\beta-\alpha)^3}
\right)
\left(
1 - e^{-2(\mu+\beta-\alpha)(s+1)^2}
\right)
.
\end{align*}
Together, this implies that $\PP{E_s} \in O\left(\frac{1}{s^2}\right)$. Therefore
$
\sum_{s = 0}^\infty \PP{E_s} < \infty
$,
and so by the Borel-Cantelli Lemma, $\frac{|N_{t,\nu} - \E{N_{t,\nu}}|}{t} \stackrel{\mathrm{a.s.}}{\longrightarrow} 0$. Since $\lim_{t \to \infty} \frac{\E{N_{t,\nu}}}{t} = \nu_\infty$, we complete the proof.
\end{proof}
\end{theorem}

As an immediate consequence of this, we achieve a law of large numbers for the dependent inter-arrival times.

\begin{corollary}\label{qhSLLN}
Let $(\nu_t, Q_{t,\nu}, N_{t,\nu})$ be an HESEP counting process with baseline intensity $\nu^* > 0$, intensity jump $\alpha > 0$, intensity decay rate $\beta \geq 0$, and rate of exponentially distributed service $\mu \geq 0$, where $\mu + \beta > \alpha$. Further, let $S_k^\nu$ denote the $k^\text{th}$ inter-arrival time for $k \in \mathbb{Z}^+$. Then,
\begin{align}
\frac{1}{n}\sum_{k=1}^n S_k^\nu \stackrel{\mathrm{a.s.}}{\longrightarrow} \frac{1}{\nu_\infty}
\end{align}
as $n \to \infty$, where $\nu_\infty = \frac{(\mu+\beta) \nu^*}{\mu+\beta - \alpha}$.
\begin{proof}
Let $A_n^\nu$ denote the time of the $n^\text{th}$ arrival for each $n \in \mathbb{Z}^+$, which is to say that $A_n^\nu = \sum_{k=1}^n S_k^\nu$. Now, observe that the time of the most recent arrival up to time $t$, $A_{N_{t,\nu}}^\nu$, can be bounded as
$$
t - S_{N_{t,\nu} + 1}^\nu \leq A_{N_{t,\nu}}^\nu \leq t,
$$
since if $t - S_{N_{t,\nu} + 1}^\nu > A_{N_{t,\nu}}^\nu$, then arrival $N_{t,\nu} + 1$ would have occurred before time $t$. Now, we also note that because $\nu^* > 0$ then $N_{t,\nu} \to \infty$ as $t \to \infty$ and this implies that
$$
\frac{S_{N_{t,\nu} + 1}}{N_{t,\nu}} \stackrel{\mathrm{a.s.}}{\longrightarrow} 0
$$
as $t \to \infty$. From Proposition~\ref{bhthm} and the continuous mapping theorem, we know that $\frac{t}{N_{t,\nu}} \to \frac{1}{\nu_\infty}$ and $\frac{t - S_{N_{t,\nu} + 1}}{N_{t,\nu}} \to \frac{1}{\nu_\infty}$ almost surely. By the sandwiching $A_{N_{t,\nu}}$, this yields the stated result.
\end{proof}
\end{corollary}

Because the Hawkes process and the ESEP are special cases of this hybrid model, we can note that both the renewal result and the law of large numbers apply directly to each.

\begin{corollary}
Let $(\lambda_t, N_{t,\lambda})$ be the intensity and count of a Hawkes process with baseline intensity $\lambda^* > 0$, intensity jump $\alpha > 0$, and decay rate $\beta > \alpha$. Similarly, let $(\eta_t, N_{t,\eta})$ be the intensity and counting process pair for an ESEP with baseline intensity $\nu^* > 0$, intensity jump $\alpha > 0$, and rate of exponentially distributed service $\mu > \alpha$. Then, for $S^\lambda_k$ and $S^\eta_k$ as the $k^\text{th}$ inter-arrival times for the  Hawkes process and the ESEP process, respectively, we have that
\begin{align}
&
\quad
 \frac{N_{t,\lambda}}{t}
\stackrel{\mathrm{a.s.}}{\longrightarrow}
\lambda_\infty
,
&
&
\quad
\frac{N_{t,\eta}}{t}
\stackrel{\mathrm{a.s.}}{\longrightarrow}
\eta_\infty
,
\intertext{and}
&
\frac{1}{n}\sum_{k=1}^n S_k^\lambda \stackrel{\mathrm{a.s.}}{\longrightarrow} \frac{1}{\lambda_\infty}
,&
&
\frac{1}{n}\sum_{k=1}^n S_k^\eta \stackrel{\mathrm{a.s.}}{\longrightarrow} \frac{1}{\eta_\infty}
,
\end{align}
where $\lambda_\infty = \frac{\beta \lambda^*}{\beta - \alpha}$ and $\eta_\infty = \frac{\mu \nu^*}{\mu - \alpha}$.
\end{corollary}

\subsection{Baseline Fluid Limit of the HESEP} \label{fluidlim}

In this subsection and in the sequel, we consider a baseline scaling of the HESEP. That is, we  investigate limiting properties of the process as the baseline intensity grows large and the intensity and queue length are normalized in some fashion. To begin, we take the normalization as directly proportional to the baseline scaling, which is the fluid limit. The derivation of this is empowered by the following lemma, which allows us to make use of Taylor expansions.

\begin{lemma} \label{explemma}
Suppose that for some $b > 0$, $ -b \leq z_n(t) \leq 0$ for all values of $n$. Then there exists constants $C_1$ and $C_2$ where $C_1 \leq C_2$, which imply the following bounds for sufficiently large values of $n$
\begin{eqnarray}
z_{n}(t) + \frac{C_1}{n} \leq n \cdot \left( e^{ \frac{ z_n(t) }{n} } - 1 \right) \leq z_{n}(t) + \frac{C_2}{n}
.
\end{eqnarray}
\begin{proof}
The proof follows by performing a second order Taylor expansion for the exponential function and observing that since $z_n(t)$ lies in a compact interval, we can construct uniform lower and upper bounds for the exponential function.
\end{proof}
\end{lemma}

With this lemma in hand, we now proceed to finding the fluid limit in Theorem~\ref{basefluid}. In this case, we scale the baseline intensity by $n$, whereas we scale the intensity and the queue length by $\frac{1}{n}$. As one would expect to see, we find that the fluid limit converges to the means of the intensity and queue.

\begin{theorem} \label{basefluid}
For $n \in \mathbb{Z}$, let the $n^\text{th}$ fluid-scaled HESEP $(\nu_t(n), Q_{t,\nu}(n))$ be defined such that the baseline intensity is $n\nu^*$, the intensity jump size is $\alpha > 0$, the intensity decay rate is $\beta \geq 0$, and  the rate of exponentially distributed service is $\mu > 0$, where $\mu + \beta > \alpha$. Then, for the scaled quantities $(\frac{\nu_{t,\nu}(n)}{n}, \frac{Q_{t,\nu}(n)}{n} )$, the limit of the moment generating function
\begin{align}
\tilde {\mathcal{M}}^{\infty}(t,\theta_{\nu},\theta_Q)
\equiv
 \lim_{n \to \infty} \E{e^{\frac{\theta_{\nu}}{n} \nu_t(n) + \frac{\theta_Q}{n} Q_{t,\nu}(n)} }
 ,
\end{align}
is given by
\begin{align}
\tilde {\mathcal{M}}^{\infty}(t,\theta_{\nu},\theta_Q)
&
=
e^{\theta_\nu \E{\nu_t} + \theta_Q \E{Q_{t,\nu}}}
,
\end{align}
for all $t \geq 0$.
\begin{proof}
The proof will follow in two steps.  The first step is to show that the limiting moment generating function converges to a PDE given by $\tilde{\mathcal{M}}^{\infty}$ using properties of the exponential function and Lemma \ref{explemma}.  The second step is to solve this PDE using the method of characteristics.  Finally, by the uniqueness of moment generating functions, we can assert that the random variables to which our limit converges are deterministic functions of time, which are also known as the fluid limit. We begin with the infinitesimal generator form which simplifies through the linearity of expectation as
\begin{align*}
&
\frac{\partial}{\partial t} \tilde{\mathcal{M}}^n(t, \theta_\nu, \theta_Q)
\equiv
\frac{\partial }{\partial t} \E{e^{\frac{\theta_{\nu}} {n} \nu_t(n)  + \frac{\theta_Q }{n} Q_{t,\nu}(n)  } }
\\
&
\quad
=
\E{\beta (\nu^* n - \nu_t(n)) \frac{\theta_\nu}{n}  e^{\frac{\theta_{\nu}} {n} \nu_t(n)  + \frac{\theta_Q }{n} Q_{t,\nu}(n)  } }
+
\E{\nu_t(n) \left( e^{\frac{\alpha\theta_\nu}{n} + \frac{\theta_Q}{n}} - 1\right) e^{\frac{\theta_{\nu}} {n} \nu_t(n)  + \frac{\theta_Q }{n} Q_{t,\nu}(n)  } }
\\
&
\quad
+
\E{\mu Q_{t,\nu}(n) \left( e^{-\frac{\theta_\nu(\nu_t(n) - \nu^*n)}{n Q_{t,\nu}(n)} - \frac{\theta_Q }{n}} - 1 \right) e^{\frac{\theta_{\nu}} {n} \nu_t(n)  + \frac{\theta_Q }{n} Q_{t,\nu}(n)  } }
\\
&
=
\beta \nu^* \theta_\nu \E{  e^{\frac{\theta_{\nu}} {n} \nu_t(n)  + \frac{\theta_Q }{n} Q_t(n)  } }
-
\beta \theta_\nu
\E{ \frac{\nu_t(n)}{n}  e^{\frac{\theta_{\nu}} {n} \nu_t(n)  + \frac{\theta_Q }{n} Q_{t,\nu}(n)  } }
\\
&
\quad
+
n\left( e^{\frac{\alpha\theta_\nu}{n} + \frac{\theta_Q}{n}} - 1\right)
\E{\frac{\nu_t(n)}{n}  e^{\frac{\theta_{\nu}} {n} \nu_t(n)  + \frac{\theta_Q }{n} Q_t(n)  } }
\\
&
\quad
+
\frac{\mu}{n}
\E{ Q_{t,\nu}(n) n\left( e^{-\frac{\theta_\nu(\nu_t(n) - \nu^*n)}{n Q_{t,\nu}(n)} - \frac{\theta_Q }{n}} - 1 \right) e^{\frac{\theta_{\nu}} {n} \nu_t(n)  + \frac{\theta_Q }{n} Q_{t,\nu}(n)  } }
\\
&
=
\beta \nu^* \theta_\nu \tilde{\mathcal{M}}(t, \theta_\nu, \theta_Q)
+
\left(
n\left( e^{\frac{\alpha\theta_\nu}{n} + \frac{\theta_Q}{n}} - 1\right)
-
\beta \theta_\nu
\right)
\frac{\partial}{\partial \theta_\nu} \tilde{\mathcal{M}}(t, \theta_\nu, \theta_Q)
\\
&
\quad
+
\frac{\mu}{n}
\E{ Q_{t,\nu}(n) \left( -\frac{\theta_\nu(\nu_t(n) - \nu^*n)}{ Q_{t,\nu}(n)} - \theta_Q  + \frac{\epsilon_n}{n} \right) e^{\frac{\theta_{\nu}} {n} \nu_t(n)  + \frac{\theta_Q }{n} Q_{t,\nu}(n)  } }
,
\intertext{where the last equality holds for sufficiently large $n$, where $\epsilon_n$ is in some bounded interval as according to Lemma~\ref{explemma}. Then, by rearranging further we can see that in limit this becomes}
&
\frac{\partial}{\partial t} \tilde{\mathcal{M}}^n(t, \theta_\nu, \theta_Q)
\\
&
=
\beta \nu^* \theta_\nu M^n(t, \theta_\nu, \theta_Q)
+
\left(
n\left( e^{\frac{\alpha\theta_\nu}{n} + \frac{\theta_Q}{n}} - 1\right)
-
\beta \theta_\nu
\right)
\frac{\partial}{\partial \theta_\nu} \tilde{\mathcal{M}}^n(t, \theta_\nu, \theta_Q)
-
\mu \theta_\nu
\E{ \frac{\nu_t(n)}{n}  e^{\frac{\theta_{\nu}} {n} \nu_t(n)  + \frac{\theta_Q }{n} Q_{t,\nu}(n)  } }
\\
&
\quad
+
\mu \theta_\nu  \nu^*
\E{    e^{\frac{\theta_{\nu}} {n} \nu_t(n)  + \frac{\theta_Q }{n} Q_{t,\nu}(n)  } }
-
\mu \theta_Q
\E{ \frac{Q_{t,\nu}(n)}{n} e^{\frac{\theta_{\nu}} {n} \nu_t(n)  + \frac{\theta_Q }{n} Q_{t,\nu}(n)  } }
+
\frac{\mu \epsilon_n}{n}
\E{ Q_{t,\nu}(n)  e^{\frac{\theta_{\nu}} {n} \nu_t(n)  + \frac{\theta_Q }{n} Q_{t,\nu}(n)  } }
\\
&
=
(\mu + \beta) \nu^* \theta_\nu \tilde{\mathcal{M}}^n(t, \theta_\nu, \theta_Q)
+
\left(
n\left( e^{\frac{\alpha\theta_\nu}{n} + \frac{\theta_Q}{n}} - 1\right)
-
(\mu + \beta) \theta_\nu
\right)
\frac{\partial}{\partial \theta_\nu} \tilde{\mathcal{M}}^n(t, \theta_\nu, \theta_Q)
\\
&
\quad
-
\mu \theta_Q
\frac{\partial}{\partial \theta_Q} \tilde{\mathcal{M}}^n(t, \theta_\nu, \theta_Q)
+
\frac{\mu \epsilon_n}{n^2}
\E{ Q_{t,\nu}(n)  e^{\frac{\theta_{\nu}} {n} \nu_t(n)  + \frac{\theta_Q }{n} Q_{t,\nu}(n)  } }
\\
&
\stackrel{n \to \infty}{\longrightarrow}
(\mu + \beta) \nu^* \theta_\nu \tilde{\mathcal{M}}^\infty(t, \theta_\nu, \theta_Q)
+
\left(
\theta_Q
-
(\mu + \beta - \alpha) \theta_\nu
\right)
\frac{\partial}{\partial \theta_\nu} \tilde{\mathcal{M}}^\infty(t, \theta_\nu, \theta_Q)
-
\mu \theta_Q
\frac{\partial}{\partial \theta_Q} \tilde{\mathcal{M}}^\infty(t, \theta_\nu, \theta_Q)
.
\end{align*}
We now solve this partial differential equation for $\tilde{\mathcal{M}}^\infty(t, \theta_\nu, \theta_Q)$ through the method of characteristics. For simplicity's sake we will instead use this procedure to solve for $G(t,\theta_\nu,\theta_Q) = \log\left(\tilde{\mathcal{M}}^\infty(t, \theta_\nu, \theta_Q)\right)$. This PDE is given by
$$
(\mu + \beta) \nu^* \theta_\nu
=
\frac{\partial}{\partial t}G(t,\theta_\nu,\theta_Q)
+
\mu \theta_Q
\frac{\partial}{\partial \theta_Q }G(t,\theta_\nu,\theta_Q)
+
\left(
(\mu + \beta - \alpha) \theta_\nu - \theta_Q
\right)
\frac{\partial}{\partial \theta_\nu }G(t,\theta_\nu,\theta_Q)
,
$$
with boundary condition $G(0,\theta_\nu,\theta_Q) = \theta_Q Q_0 + \theta_\nu \nu_0$. This corresponds to the following system of characteristics equations:
\begin{align*}
\frac{\mathrm{d} \theta_Q}{\mathrm{d}z}(x,y,z)
&
=
\mu \theta_Q
,
&&
\theta_Q(x,y,0) = x
\\
\frac{\mathrm{d} \theta_\nu}{\mathrm{d}z}(x,y,z)
&
=
(\mu + \beta - \alpha) \theta_\nu - \theta_Q
,
&&
\theta_\nu(x,y,0) = y
\\
\frac{\mathrm{d} t}{\mathrm{d}z}(x,y,z)
&
=
1
,
&&
t(x,y,0) = 0
\\
\frac{\mathrm{d} g}{\mathrm{d}z}(x,y,z)
&
=
(\mu + \beta) \nu^* \theta_\nu
=
(\mu+\beta-\alpha)\nu_\infty \theta_\nu
,
&&
g(x,y,0) = x Q_0 + y \nu_0
.
\end{align*}
If $\beta \ne \alpha$, the solutions to these initial value problems are given by:
\begin{align*}
 \theta_Q(x,y,z)
&
=
x e^{\mu z}
,
\\
 \theta_\nu(x,y,z)
&
=
y e^{(\mu+\beta-\alpha)z}
+
\frac{x}{\beta-\alpha}\left(e^{\mu z} - e^{(\mu+\beta-\alpha)z}\right)
,
\\
&
=
\left( y - \frac{x}{\beta - \alpha}\right)e^{(\mu + \beta - \alpha)z}
+
\frac{x e^{\mu z}}{\beta-\alpha}
,
\\
 t(x,y,z)
&
=
z
,
\\
g(x,y,z)
&
=
x Q_0 + y \nu_0
+
\nu_\infty \left( y - \frac{x}{\beta - \alpha}\right)\left( e^{(\mu + \beta - \alpha)z} - 1 \right)
+
\frac{x \nu_\infty (\mu+\beta -\alpha) (e^{\mu z}-1)}{\mu(\beta-\alpha)}
.
\end{align*}
Now, we can solve for the characteristic variables in terms of the original variables and find
$
x
=
\theta_Q e^{-\mu t}
$,
$
y
=
\theta_\nu e^{-(\mu+\beta-\alpha)t}
+
\frac{\theta_Q}{\beta-\alpha}\left(e^{-\mu t} - e^{-(\mu+\beta-\alpha)t}\right)
$,
and
$
z
=
t
$,
so this gives a PDE solution of
\begin{align*}
G(t,\theta_Q, \theta_\nu)
&
=
g\left(
\theta_Q e^{-\mu t}
,
\theta_\nu e^{-(\mu+\beta-\alpha)t} + \frac{\theta_Q}{\beta-\alpha}\left(e^{-\mu t} - e^{-(\mu+\beta-\alpha)t}\right)
,
t
\right)
\\
&
=
Q_0
\theta_Q e^{-\mu t}
+
\nu_0
\left(
\theta_\nu e^{-(\mu+\beta-\alpha)t} + \frac{\theta_Q}{\beta-\alpha}\left(e^{-\mu t} - e^{-(\mu+\beta-\alpha)t}\right)
\right)
\\
&
\quad
+
\nu_\infty
\left(
\theta_\nu  - \frac{\theta_Q}{\beta-\alpha}
\right)
\left( 1 - e^{-(\mu+\beta-\alpha)t} \right)
+
\frac{\theta_Q  \nu_\infty (\mu+\beta -\alpha) (1-e^{-\mu t})}{\mu(\beta-\alpha)}
.
\end{align*}
If instead $\beta = \alpha$, the solutions to the characteristic ODE's are as follows:
\begin{align*}
 \theta_Q(x,y,z)
&
=
x e^{\mu z}
,
\\
 \theta_\nu(x,y,z)
&
=
e^{\mu z} \left( y  - x z \right)
,
\\
 t(x,y,z)
&
=
z
,
\\
g(x,y,z)
&
=
x Q_0 + y \nu_0
+
\nu_\infty y \left( e^{\mu z} - 1 \right)
-
\frac{x \nu_\infty}{\mu}  \left(e^{\mu z}(\mu z - 1) + 1\right)
.
\end{align*}
This makes our expressions for the characteristic variables
$
x
=
\theta_Q e^{-\mu t}
$,
$
y
=
\theta_\nu e^{-\mu t} + \theta_Q t e^{-\mu t}
$,
and
$
z
=
t
$
.
This now makes the PDE solution
\begin{align*}
G(t,\theta_Q, \theta_\nu)
&
=
g\left(
\theta_Q e^{-\mu t}
,
\theta_\nu e^{-\mu t} + \theta_Q t e^{-\mu t}
,
t
\right)
\\
&
=
Q_0 \theta_Q e^{-\mu t}
+
\nu_0 \theta_\nu e^{-\mu t}
+
\nu_0 \theta_Q t e^{-\mu t}
+
\nu_\infty \left( \theta_\nu  + \theta_Q t  \right) \left( 1 - e^{-\mu t} \right)
-
\frac{ \nu_\infty \theta_Q }{\mu}  \left(\mu t - 1 + e^{-\mu t}\right)
,
\end{align*}
and we can observe that each of these cases simplify to the corresponding means of the queue and the intensity, which yields the stated result.
\end{proof}
\end{theorem}



\subsection{Baseline Diffusion Limit of the HESEP} \label{difflim}

To now consider a diffusion limit we will still scale the baseline intensity by $n$, but we now instead scale the process intensity and the queue length by $\frac{1}{\sqrt{n}}$. More specifically, we scale the centered version of the processes by $\frac{1}{\sqrt{n}}$. While we can make use of some of the techniques used for the fluid limit in Theorem~\ref{basefluid}, the diffusion scaling also involves second order terms. It is challenging to calculate such quantities for the HESEP. Thus, we will use the same idea bound the quantities above and below via
\begin{align}
0 \leq \frac{(\nu_t - \nu^*)^2}{Q_{t,\nu}} \leq \alpha (\nu_t - \nu^*) \label{diffBound}
,
\end{align}
which follows from our previously discussed bounds on the down-jump size.
By doing so, we create upper and lower bounds for the true diffusion limit of the HESEP. To facilitate a variety of approximations that fit within these bounds, we introduce the parameter $\gamma \in [0,1]$,  with $\gamma = 0$ corresponding to the lower bound and $\gamma = 1$ as the upper.

\begin{theorem} \label{basediffusion}
For $n \in \mathbb{Z}$, let the $n^\text{th}$ diffusion-scaled HESEP $(\nu_t(n), Q_{t,\nu}(n))$ be defined such that the baseline intensity is $n\nu^*$, the intensity jump size is $\alpha > 0$, the intensity decay rate is $\beta \geq 0$, and  the rate of exponentially distributed service is $\mu > 0$, where $\mu + \beta > \alpha$. For the scaled quantities $(\frac{\nu_t(n)}{\sqrt{n}}, \frac{Q_{t,\nu}(n)}{\sqrt{n}} )$, let  $\hat{\mathcal{M}}^{\infty}(t,\theta_{\nu},\theta_Q)$ be defined
\begin{align}
\hat{\mathcal{M}}^{\infty}(t,\theta_{\nu},\theta_Q)
&
\equiv
\lim_{n \to \infty} \E{e^{\frac{\theta_{\nu}} {\sqrt{n}} \left(  \nu_t(n)  - n\nu_\infty \right)+ \frac{\theta_Q }{\sqrt{n}} \left( Q_{t,\nu}(n)  -  \frac{n\nu_\infty}{\mu } \right) } }
,
\end{align}
if the limit converges. Then for $\beta \neq \alpha$, this is bounded above and below by $\mathcal{B}_0(t,\theta_{\nu},\theta_Q) \leq \hat{\mathcal{M}}^{\infty}(t,\theta_{\nu},\theta_Q) \leq \mathcal{B}_1(t,\theta_{\nu},\theta_Q)$, where $\mathcal{B}_\gamma(t,\theta_{\nu},\theta_Q)$ is given by
\begin{align}
\mathcal{B}_\gamma(t,\theta_{\nu},\theta_Q)
&
=
\nonumber
e^{
\nu_0\theta_\nu e^{-(\mu+\beta-\alpha)t} + \frac{\nu_0\theta_Q}{\beta-\alpha}\left(e^{-\mu t} - e^{-(\mu+\beta-\alpha)t}\right)
+
Q_0 \theta_Q e^{-\mu t}
+
\left(\theta_\nu - \frac{\theta_Q}{\beta - \alpha}\right)^2
\left(
\frac{\gamma \alpha \mu (\nu_\infty - \nu^*) }{2}
+
\frac{\alpha^2 \nu_\infty }{2 }
\right)
\frac{1 - e^{-2(\mu + \beta - \alpha)t}}{2(\mu + \beta - \alpha)}
}
\\
&
\quad
\cdot
e^{
\left(\theta_\nu\theta_Q - \frac{\theta_Q^2}{\beta - \alpha}\right)
\left(
\left(
\frac{\gamma \alpha\mu}{\beta-\alpha}
+
\mu
\right)(\nu_\infty - \nu^*)
+
\frac{\alpha \beta \nu_\infty}{\beta-\alpha}
\right)
\frac{1-e^{-(2\mu + \beta-\alpha)t}}{2\mu+\beta-\alpha}
}
\nonumber
\\
&
\quad
\cdot
e^{
\theta_Q^2
\left(
\frac{\gamma \alpha \mu (\nu_\infty- \nu^*) }{2(\beta-\alpha)^2}
+
\frac{\mu(\nu_\infty - \nu^*) }{\beta-\alpha}
+
\frac{\nu_\infty }{2}
+
\frac{\nu_\infty \beta^2 }{2(\beta - \alpha)^2}
\right)\frac{1 - e^{-2\mu t}}{2\mu}
}
,
\end{align}
 whereas if $\beta = \alpha$, it is instead
\begin{align}
\mathcal{B}_\gamma(t,\theta_{\nu},\theta_Q)
&
=
\nonumber
e^{
\nu_0 \theta_\nu e^{-\mu t} + \nu_0 \theta_Q t e^{- \mu t}
+
Q_0 \theta_Q e^{-\mu t}
+
\bigg(
\left(\frac{\gamma \alpha  (\theta_\nu + \theta_Q t)^2}{2}  +  \theta_\nu \theta_Q + \theta_Q^2 t\right) \frac{1 - e^{-2\mu t}}{2}
-
\left(\gamma \alpha (\theta_\nu \theta_Q + \theta_Q^2 t)  +  \theta_Q^2 \right) \frac{2\mu t - 1 + e^{-2\mu t} }{4 \mu}
}
\\
&
\quad
\cdot
e^{
\frac{\gamma \alpha \theta_Q^2}{2}
\left(
\frac{2\mu t(\mu t-1)+1 - e^{-2\mu t}}{4 \mu^2}
\right)
\bigg) (\nu_\infty - \nu^*)
+
\frac{\nu_\infty }{2 }
\bigg(
\left(
\theta_Q^2 + (\theta_Q + \alpha \theta_\nu)^2
+
2\left(\alpha^2 \theta_\nu \theta_Q   + \alpha \theta_Q^2 \right) t
+
\alpha^2 \theta_Q^2 t^2
\right) \frac{1 - e^{-2\mu t}}{2\mu}
}
\nonumber
\\
&
\quad
\cdot
e^{
-
2 \left(
\alpha^2 \theta_\nu \theta_Q  + \alpha \theta_Q^2 + \alpha^2 \theta_Q^2 t
\right)
\left(\frac{2\mu t - 1 + e^{-2\mu t} }{4 \mu^2}\right)
+
\alpha^2 \theta_Q^2 \left(  \frac{ 2\mu t(\mu t-1) +1  - e^{-2\mu t}}{4 \mu^3} \right)
\bigg)
}
,
\end{align}
for $\gamma \in [0,1]$ with $t \geq 0$ and $\nu_\infty = \frac{(\mu + \beta)\nu^*}{\mu+\beta-\alpha}$.
\begin{proof}
We begin by bounding the quantity $Q_{t,\nu}(n) \left(\frac{\nu_t(n) - n\nu^*}{Q_{t,\nu}(n)}\right)^2$ above and below by observing
$$
0
\leq
Q_{t,\nu}(n) \left(\frac{\nu_t(n) - n\nu^*}{Q_{t,\nu}(n)}\right)^2
=
(\nu_t(n) - n\nu^*)\left(\frac{\nu_t(n) - n\nu^*}{Q_{t,\nu}(n)}\right)
\leq
\alpha (\nu_t(n) - n\nu^*)
.
$$
To consolidate the development of the two bounds into one approach, we introduce the extra parameter $\gamma \in \{0,1\}$ and replace $Q_{t,\nu}(n) \left(\frac{\nu_t(n) - n\nu^*}{Q_{t,\nu}(n)}\right)^2$ by $\gamma \alpha (\nu_t(n) - n\nu^*)$ in the following diffusion limit derivation. In this notation, $\gamma = 0$ yields the lower bound and $\gamma = 1$ the upper. These two cases share the same start -- identifying the moment generating function form of the pre-limit object. By Lemma~\ref{fubinifundQ}, this is
\begin{align*}
&
\frac{\partial}{\partial t}\hat{\mathcal{M}}^n(\theta_\nu, \theta_Q, t)
=
\frac{\partial}{\partial t}\E{e^{\theta_\nu \left(\frac{\nu_t(n) - n \nu_\infty}{\sqrt{n}}\right) + \theta_Q \left(\frac{Q_{t,\nu}(n) - \frac{n \nu_\infty}{\mu}}{\sqrt{n}}\right)}}
\\
&
=
\E{\nu_t(n)\left(e^{\frac{\alpha\theta_\nu + \theta_Q}{\sqrt{n}}} - 1 \right)e^{\theta_\nu \left(\frac{\nu_t(n) - n \nu_\infty}{\sqrt{n}}\right) + \theta_Q \left(\frac{Q_{t,\nu}(n) - \frac{n \nu_\infty}{\mu}}{\sqrt{n}}\right)}}
\\
&
\quad
+
\E{\mu Q_{t,\nu}(n)\left(e^{-\frac{\theta_\nu}{\sqrt{n}}\left(\frac{\nu_t(n)-n\nu^*}{Q_{t,\nu}(n)}\right) - \frac{\theta_Q}{\sqrt{n}}} - 1 \right)e^{\theta_\nu \left(\frac{\nu_t(n) - n \nu_\infty}{\sqrt{n}}\right) + \theta_Q \left(\frac{Q_{t,\nu}(n) - \frac{n \nu_\infty}{\mu}}{\sqrt{n}}\right)}}
\\
&
\quad
+
\E{\frac{\beta \theta_\nu}{\sqrt{n}}\left(n\nu^* - \nu_t(n)\right)e^{\theta_\nu \left(\frac{\nu_t(n) - n \nu_\infty}{\sqrt{n}}\right) + \theta_Q \left(\frac{Q_{t,\nu}(n) - \frac{n \nu_\infty}{\mu}}{\sqrt{n}}\right)}}
.
\end{align*}
As a first step, we simplify this expression through the linearity of expectation. Moving deterministic terms outside of the expectation and re-scaling, we have
\begin{align*}
&
\frac{\partial}{\partial t}\hat{\mathcal{M}}^n(\theta_\nu, \theta_Q, t)
=
\sqrt{n}\left(e^{\frac{\alpha\theta_\nu + \theta_Q}{\sqrt{n}}} - 1 \right)\E{\frac{\nu_t(n)}{\sqrt{n}}e^{\theta_\nu \left(\frac{\nu_t(n) - n \nu_\infty}{\sqrt{n}}\right) + \theta_Q \left(\frac{Q_{t,\nu}(n) - \frac{n \nu_\infty}{\mu}}{\sqrt{n}}\right)}}
\\
&
\quad
+
\E{\mu Q_{t,\nu}(n)\left(e^{-\frac{\theta_\nu}{\sqrt{n}}\left(\frac{\nu_t(n)-n\nu^*}{Q_{t,\nu}(n)}\right) - \frac{\theta_Q}{\sqrt{n}}} - 1 \right)e^{\theta_\nu \left(\frac{\nu_t(n) - n \nu_\infty}{\sqrt{n}}\right) + \theta_Q \left(\frac{Q_{t,\nu}(n) - \frac{n \nu_\infty}{\mu}}{\sqrt{n}}\right)}}
\\
&
\quad
+
\beta \theta_\nu \nu^* \sqrt{n} \E{ e^{\theta_\nu \left(\frac{\nu_t(n) - n \nu_\infty}{\sqrt{n}}\right) + \theta_Q \left(\frac{Q_{t,\nu}(n) - \frac{n \nu_\infty}{\mu}}{\sqrt{n}}\right)}}
-
\beta \theta_\nu \E{\frac{\nu_t(n)}{\sqrt{n}}e^{\theta_\nu \left(\frac{\nu_t(n) - n \nu_\infty}{\sqrt{n}}\right) + \theta_Q \left(\frac{Q_{t,\nu}(n) - \frac{n \nu_\infty}{\mu}}{\sqrt{n}}\right)}}
.
\end{align*}
For the terms on the first and third lines in the right-hand side of the above equation, we are able to re-express the expectation in terms of the moment generating function or its derivatives. For the first and second lines, we perform Taylor expansions and truncate terms from the third order and above. This now yields
\begin{align*}
&
\frac{\partial}{\partial t}\hat{\mathcal{M}}^n(\theta_\nu, \theta_Q, t)
=
\left(\alpha\theta_\nu + \theta_Q + \frac{(\alpha\theta_\nu + \theta_Q)^2}{2 \sqrt{n}} + O\left(\frac{1}{n}\right) \right)\left( \frac{\partial}{\partial \theta_\nu}\hat{\mathcal{M}}^n(\theta_\nu, \theta_Q, t) + \nu_\infty\sqrt{n} \hat{\mathcal{M}}^n(\theta_\nu, \theta_Q, t) \right)
\\
&
\quad
+
\mathrm{E}\Bigg[\mu Q_{t,\nu}(n)\left( -\frac{\theta_\nu}{\sqrt{n}}\left(\frac{\nu_t(n)-n\nu^*}{Q_{t,\nu}(n)}\right) - \frac{\theta_Q}{\sqrt{n}} + \frac{1}{2n}\left(\theta_\nu\left(\frac{\nu_t(n)-n\nu^*}{Q_{t,\nu}(n)}\right) + \theta_Q\right)^2 + O\left(n^{-\frac{3}{2}}\right) \right)
\\
&
\quad
\cdot
e^{\theta_\nu \left(\frac{\nu_t(n) - n \nu_\infty}{\sqrt{n}}\right) + \theta_Q \left(\frac{Q_{t,\nu}(n) - \frac{n \nu_\infty}{\mu}}{\sqrt{n}}\right)}\Bigg]
+
\beta \theta_\nu \nu^* \sqrt{n} \hat{\mathcal{M}}^n(\theta_\nu, \theta_Q, t)
-
\beta \theta_\nu \left( \frac{\partial}{\partial \theta_\nu}\hat{\mathcal{M}}^n(\theta_\nu, \theta_Q, t) + \nu_\infty\sqrt{n} \hat{\mathcal{M}}^n(\theta_\nu, \theta_Q, t) \right)
.
\end{align*}
We now begin distributing and combining like terms through linearity of expectation. Moreover, we distribute within the expectation on the second line and cancel $Q_{t,\nu}(n)$ across the numerator and denominator where possible.
\begin{align*}
&
\frac{\partial}{\partial t}\hat{\mathcal{M}}^n(\theta_\nu, \theta_Q, t)
=
\left(\alpha\theta_\nu + \theta_Q + \frac{(\alpha\theta_\nu + \theta_Q)^2}{2 \sqrt{n}} - \beta \theta_\nu + O\left(\frac{1}{n}\right) \right)\left( \frac{\partial}{\partial \theta_\nu}\hat{\mathcal{M}}^n(\theta_\nu, \theta_Q, t) + \nu_\infty\sqrt{n} \hat{\mathcal{M}}^n(\theta_\nu, \theta_Q, t) \right)
\\
&
\quad
-
\mu\theta_\nu \E{\frac{\nu_t(n)}{\sqrt{n}}e^{\theta_\nu \left(\frac{\nu_t(n) - n \nu_\infty}{\sqrt{n}}\right) + \theta_Q \left(\frac{Q_{t,\nu}(n) - \frac{n \nu_\infty}{\mu}}{\sqrt{n}}\right)}}
+
\mu \theta_\nu \nu^* \sqrt{n} \E{ e^{\theta_\nu \left(\frac{\nu_t(n) - n \nu_\infty}{\sqrt{n}}\right) + \theta_Q \left(\frac{Q_{t,\nu}(n) - \frac{n \nu_\infty}{\mu}}{\sqrt{n}}\right)}}
\\
&
\quad
-
\mu \theta_Q \E{\frac{ Q_{t,\nu}(n)}{\sqrt{n}}e^{\theta_\nu \left(\frac{\nu_t(n) - n \nu_\infty}{\sqrt{n}}\right) + \theta_Q \left(\frac{Q_{t,\nu}(n) - \frac{n \nu_\infty}{\mu}}{\sqrt{n}}\right)}}
\\
&
\quad
+
 \frac{\mu\theta_\nu^2}{2n}\E{ Q_{t,\nu}(n) \left(\frac{\nu_t(n)-n\nu^*}{Q_{t,\nu}(n)}\right)^2 e^{\theta_\nu \left(\frac{\nu_t(n) - n \nu_\infty}{\sqrt{n}}\right) + \theta_Q \left(\frac{Q_{t,\nu}(n) - \frac{n \nu_\infty}{\mu}}{\sqrt{n}}\right)}}
\\
&
\quad
+
 \frac{\mu\theta_\nu\theta_Q}{\sqrt{n}}\E{ \frac{\nu_t(n)}{\sqrt{n}}e^{\theta_\nu \left(\frac{\nu_t(n) - n \nu_\infty}{\sqrt{n}}\right) + \theta_Q \left(\frac{Q_{t,\nu}(n) - \frac{n \nu_\infty}{\mu}}{\sqrt{n}}\right)}}
-
 \mu\theta_\nu\theta_Q \nu^* \E{e^{\theta_\nu \left(\frac{\nu_t(n) - n \nu_\infty}{\sqrt{n}}\right) + \theta_Q \left(\frac{Q_{t,\nu}(n) - \frac{n \nu_\infty}{\mu}}{\sqrt{n}}\right)}}
\\
&
\quad
+
 \frac{\mu \theta_Q^2}{2\sqrt{n}}\E{ \frac{Q_{t,\nu}(n)}{\sqrt{n}}e^{\theta_\nu \left(\frac{\nu_t(n) - n \nu_\infty}{\sqrt{n}}\right) + \theta_Q \left(\frac{Q_{t,\nu}(n) - \frac{n \nu_\infty}{\mu}}{\sqrt{n}}\right)}}
+
 O\left(\frac{1}{n}\right)\E{\frac{Q_{t,\nu}(n)}{\sqrt{n}}  e^{\theta_\nu \left(\frac{\nu_t(n) - n \nu_\infty}{\sqrt{n}}\right) + \theta_Q \left(\frac{Q_{t,\nu}(n) - \frac{n \nu_\infty}{\mu}}{\sqrt{n}}\right)}}
\\
&
\quad
+
\beta \theta_\nu \nu^* \sqrt{n} \hat{\mathcal{M}}^n(\theta_\nu, \theta_Q, t)
.
\end{align*}
For all remaining components of this equation that are still expressed in terms of the expectation, we substitute equivalent forms in terms of the moment generating function or its partial derivatives. Furthermore, we will now replace   $Q_{t,\nu}(n) \left(\frac{\nu_t(n) - n\nu^*}{Q_{t,\nu}(n)}\right)^2$ by $\gamma \alpha (\nu_t(n) - n\nu^*)$ inside the expectation and  re-express the expectation in terms of the moment generating function accordingly. To denote that we have now made this replacement and changed the function, we add $\gamma$ as a subscript to the moment generating function, i.e. $\hat{\mathcal{M}}_\gamma^n(\theta_\nu, \theta_Q, t)$.
\begin{align*}
&
\frac{\partial}{\partial t}\hat{\mathcal{M}}_\gamma^n(\theta_\nu, \theta_Q, t)
=
\left(\alpha\theta_\nu + \theta_Q + \frac{(\alpha\theta_\nu + \theta_Q)^2}{2 \sqrt{n}} - \beta \theta_\nu + O\left(\frac{1}{n}\right) \right)\left( \frac{\partial}{\partial \theta_\nu}\hat{\mathcal{M}}_\gamma^n(\theta_\nu, \theta_Q, t) + \nu_\infty\sqrt{n} \hat{\mathcal{M}}_\gamma^n(\theta_\nu, \theta_Q, t) \right)
\\
&
\quad
-
\mu\theta_\nu \left( \frac{\partial}{\partial \theta_\nu}\hat{\mathcal{M}}_\gamma^n(\theta_\nu, \theta_Q, t) + \nu_\infty\sqrt{n} \hat{\mathcal{M}}_\gamma^n(\theta_\nu, \theta_Q, t) \right)
+
\mu \theta_\nu \nu^* \sqrt{n} \hat{\mathcal{M}}_\gamma^n(\theta_\nu, \theta_Q, t)
\\
&
\quad
-
\mu \theta_Q \left( \frac{\partial}{\partial \theta_Q}\hat{\mathcal{M}}_\gamma^n(\theta_\nu, \theta_Q, t)
+
\frac{\nu_\infty\sqrt{n}}{\mu} \hat{\mathcal{M}}_\gamma^n(\theta_\nu, \theta_Q, t) \right)
\\
&
\quad
+
 \frac{\gamma \alpha \mu\theta_\nu^2}{2\sqrt{n}}\left( \frac{\partial}{\partial \theta_\nu}\hat{\mathcal{M}}_\gamma^n(\theta_\nu, \theta_Q, t) + \nu_\infty\sqrt{n} \hat{\mathcal{M}}_\gamma^n(\theta_\nu, \theta_Q, t) \right)
-
 \frac{\gamma \alpha \mu \nu^* \theta_\nu^2}{2}\hat{\mathcal{M}}_\gamma^n(\theta_\nu, \theta_Q, t)
\\
&
\quad
+
 \frac{\mu\theta_\nu\theta_Q}{\sqrt{n}}\left( \frac{\partial}{\partial \theta_\nu}\hat{\mathcal{M}}_\gamma^n(\theta_\nu, \theta_Q, t) + \nu_\infty\sqrt{n} \hat{\mathcal{M}}_\gamma^n(\theta_\nu, \theta_Q, t) \right)
-
 \mu\theta_\nu\theta_Q \nu^* \hat{\mathcal{M}}_\gamma^n(\theta_\nu, \theta_Q, t)
\\
&
\quad
+
 \frac{\mu \theta_Q^2}{2\sqrt{n}}\left( \frac{\partial}{\partial \theta_Q}\hat{\mathcal{M}}_\gamma^n(\theta_\nu, \theta_Q, t) + \frac{\nu_\infty\sqrt{n}}{\mu} \hat{\mathcal{M}}_\gamma^n(\theta_\nu, \theta_Q, t) \right)
+
 O\left(\frac{1}{n}\right) \left( \frac{\partial}{\partial \theta_Q}\hat{\mathcal{M}}_\gamma^n(\theta_\nu, \theta_Q, t) + \frac{\nu_\infty\sqrt{n}}{\mu} \hat{\mathcal{M}}_\gamma^n(\theta_\nu, \theta_Q, t) \right)
\\
&
\quad
+
\beta \theta_\nu \nu^* \sqrt{n} \hat{\mathcal{M}}_\gamma^n(\theta_\nu, \theta_Q, t)
.
\end{align*}
Before we find the limiting object, we first combine like terms of the moment generating function, consolidating coefficients and absorbing into $O(\cdot)$ notation where possible.
\begin{align*}
&
\frac{\partial}{\partial t}\hat{\mathcal{M}}_\gamma^n(\theta_\nu, \theta_Q, t)
=
\left(  \theta_Q - (\mu + \beta - \alpha)\theta_\nu    + O\left(\frac{1}{\sqrt{n}}\right) \right) \frac{\partial}{\partial \theta_\nu}\hat{\mathcal{M}}_\gamma^n(\theta_\nu, \theta_Q, t)
-
\left( \mu \theta_Q - O\left(\frac{1}{\sqrt{n}}\right) \right) \frac{\partial}{\partial \theta_Q}\hat{\mathcal{M}}_\gamma^n(\theta_\nu, \theta_Q, t)
\\
&
\quad
\left(
\frac{\gamma \alpha \mu(\nu_\infty - \nu^*) \theta_\nu^2}{2}
+
\mu\theta_\nu\theta_Q (\nu_\infty - \nu^*)
+
\frac{\theta_Q^2 \nu_\infty}{2}
+
\frac{(\alpha\theta_\nu + \theta_Q)^2\nu_\infty }{2 }
+
O\left(\frac{1}{\sqrt{n}}\right) \right) \hat{\mathcal{M}}_\gamma^n(\theta_\nu, \theta_Q, t)
.
\end{align*}
Taking the limit as $n \to \infty$, we receive
\begin{align*}
&
\frac{\partial}{\partial t}\hat{\mathcal{M}}_\gamma^\infty(\theta_\nu, \theta_Q, t)
=
\left(  \theta_Q - (\mu + \beta - \alpha)\theta_\nu  \right) \frac{\partial}{\partial \theta_\nu}\hat{\mathcal{M}}_\gamma^\infty(\theta_\nu, \theta_Q, t)
-
 \mu \theta_Q  \frac{\partial}{\partial \theta_Q}\hat{\mathcal{M}}_\gamma^\infty(\theta_\nu, \theta_Q, t)
\\
&
\quad
+
\left(
\frac{\gamma \alpha \mu(\nu_\infty - \nu^*) \theta_\nu^2}{2}
+
\mu\theta_\nu\theta_Q (\nu_\infty - \nu^*)
+
\frac{\theta_Q^2 \nu_\infty}{2}
+
\frac{(\alpha\theta_\nu + \theta_Q)^2\nu_\infty }{2 }
 \right) \hat{\mathcal{M}}_\gamma^\infty(\theta_\nu, \theta_Q, t)
.
\end{align*}
We will now solve this limiting partial differential equation through the method of characteristics. To simplify this approach, we let $G_\gamma(\theta_\nu, \theta_Q, t) = \log\left(\hat{\mathcal{M}}_\gamma^\infty(\theta_\nu, \theta_Q, t)\right)$, which is the cumulant generating function. The resulting PDE for the cumulant generating function is then
\begin{align*}
&
\frac{\partial}{\partial t}G_\gamma(\theta_\nu, \theta_Q, t)
+
\left(  (\mu + \beta - \alpha)\theta_\nu - \theta_Q \right) \frac{\partial}{\partial \theta_\nu}G_\gamma(\theta_\nu, \theta_Q, t)
+
 \mu \theta_Q  \frac{\partial}{\partial \theta_Q}G_\gamma(\theta_\nu, \theta_Q, t)
\\
&
\quad
=
\frac{\gamma \alpha \mu(\nu_\infty - \nu^*) \theta_\nu^2}{2}
+
\mu\theta_\nu\theta_Q (\nu_\infty - \nu^*)
+
\frac{\theta_Q^2 \nu_\infty}{2}
+
\frac{(\alpha\theta_\nu + \theta_Q)^2\nu_\infty }{2 }
,
\end{align*}
with initial condition $G_\gamma(\theta_\nu, \theta_Q, 0) = \theta_\nu \nu_0 + \theta_Q Q_0$. Thus, the resulting system of characteristic equations is
\begin{align*}
\frac{\mathrm{d}t}{\mathrm{d}z}(x,y,z)
&
=
1
,
&t(x,y,0) = 0
,
\\
\frac{\mathrm{d}\theta_\nu}{\mathrm{d}z}(x,y,z)
&
=
(\mu + \beta - \alpha)\theta_\nu - \theta_Q
,
&\theta_\nu(x,y,0) = x
,
\\
\frac{\mathrm{d}\theta_Q}{\mathrm{d}z}(x,y,z)
&
=
\mu \theta_Q
,
&\theta_Q(x,y,0) = y
,
\\
\frac{\mathrm{d}g}{\mathrm{d}z}(x,y,z)
&
=
\left(\frac{\gamma \alpha \mu \theta_\nu^2}{2}
+
\mu\theta_\nu\theta_Q \right) (\nu_\infty - \nu^*)
+
\left(\theta_Q^2 + (\alpha\theta_\nu + \theta_Q)^2\right) \frac{\nu_\infty }{2 }
,
&g(x,y,0) = x \nu_0 + y Q_0
.
\end{align*}
Assuming $\beta \ne \alpha$, we can solve these first three ordinary differential equations to find that
$$
t = z,\qquad
\theta_Q = y e^{\mu z},\qquad
\theta_\nu = \left(x - \frac{y}{\beta - \alpha}\right)e^{(\mu+\beta-\alpha)z} + \frac{y}{\beta-\alpha}e^{\mu z},
$$
which we now use to solve the remaining equation. Re-writing the characteristic equation for $g$, we have
\begin{align*}
\frac{\mathrm{d}g}{\mathrm{d}z}(x,y,z)
&
=
\frac{\gamma \alpha \mu (\nu_\infty - \nu^*) }{2}
\left(
\left(x - \frac{y}{\beta - \alpha}\right)^2e^{2(\mu+\beta-\alpha)z} + \frac{2}{\beta-\alpha}\left( xy - \frac{y^2}{\beta - \alpha}\right)e^{(2\mu+\beta-\alpha)z} + \frac{y^2}{(\beta-\alpha)^2}e^{2\mu z}
\right)
\\
&
\quad
+
\mu (\nu_\infty - \nu^*) \left(\left(xy - \frac{y^2}{\beta - \alpha}\right)e^{(2\mu+\beta-\alpha)z} + \frac{y^2}{\beta-\alpha}e^{2\mu z} \right)
+
\frac{\nu_\infty }{2 }
\Bigg(
\alpha^2 \left(x - \frac{y}{\beta - \alpha}\right)^2 e^{2(\mu+\beta-\alpha)z}
\\
&
\quad
+
\frac{2\alpha\beta}{\beta-\alpha}\left( xy
-
\frac{ y^2}{\beta - \alpha}\right)e^{(2\mu+\beta-\alpha)z}
+
\left(1 + \frac{\beta^2 }{(\beta-\alpha)^2}  \right) y^2 e^{2\mu z}
\Bigg)
,
\end{align*}
and so by grouping coefficients of like exponential functions and then integrating with respect to $z$, this solves to
\begin{align*}
g(x,y,z)
&
=
x\nu_0 + y Q_0
+
\left(x - \frac{y}{\beta - \alpha}\right)^2
\left(
\frac{\gamma \alpha \mu (\nu_\infty - \nu^*) }{2}
+
\frac{\alpha^2 \nu_\infty }{2 }
\right)
\frac{e^{2(\mu + \beta - \alpha)z}-1}{2(\mu + \beta - \alpha)}
\\
&
\quad
+
\left(xy - \frac{y^2}{\beta - \alpha}\right)
\left(
\left(
\frac{\gamma \alpha\mu}{\beta-\alpha}
+
\mu
\right)(\nu_\infty - \nu^*)
+
\frac{\alpha \beta \nu_\infty}{\beta-\alpha}
\right)
\frac{e^{(2\mu + \beta-\alpha)z}-1}{2\mu+\beta-\alpha}
\\
&
\quad
+
y^2
\left(
\frac{\gamma \alpha \mu (\nu_\infty- \nu^*) }{2(\beta-\alpha)^2}
+
\frac{\mu(\nu_\infty - \nu^*) }{\beta-\alpha}
+
\frac{\nu_\infty }{2}
+
\frac{\nu_\infty \beta^2 }{2(\beta - \alpha)^2}
\right)\frac{e^{2\mu z}-1}{2\mu}
.
\end{align*}
From the solutions to the characteristic equations, we can express each of $x$, $y$, and $z$ in terms of the three cumulant generating function parameters:
$$
z = t,\qquad
y = \theta_Q e^{-\mu t},\qquad
x = \theta_\nu e^{-(\mu+\beta-\alpha)t} + \frac{\theta_Q}{\beta-\alpha}\left(e^{-\mu t} - e^{-(\mu+\beta-\alpha)t}\right)
.
$$
Thus, we can then solve for $G_\gamma(\theta_\nu, \theta_Q, t)$ via
\begin{align*}
G_\gamma(\theta_\nu, \theta_Q, t)
&
=
g\left(\theta_\nu e^{-(\mu+\beta-\alpha)t} + \frac{\theta_Q}{\beta-\alpha}\left(e^{-\mu t} - e^{-(\mu+\beta-\alpha)t}\right), \theta_Q e^{-\mu t}, t \right)
\\
&
=
\nu_0\theta_\nu e^{-(\mu+\beta-\alpha)t} + \frac{\nu_0\theta_Q}{\beta-\alpha}\left(e^{-\mu t} - e^{-(\mu+\beta-\alpha)t}\right)
+
Q_0 \theta_Q e^{-\mu t}
\\
&
\quad
+
\left(\theta_\nu - \frac{\theta_Q}{\beta - \alpha}\right)^2
\left(
\frac{\gamma \alpha \mu (\nu_\infty - \nu^*) }{2}
+
\frac{\alpha^2 \nu_\infty }{2 }
\right)
\frac{1 - e^{-2(\mu + \beta - \alpha)t}}{2(\mu + \beta - \alpha)}
\\
&
\quad
+
\left(\theta_\nu\theta_Q - \frac{\theta_Q^2}{\beta - \alpha}\right)
\left(
\left(
\frac{\gamma \alpha\mu}{\beta-\alpha}
+
\mu
\right)(\nu_\infty - \nu^*)
+
\frac{\alpha \beta \nu_\infty}{\beta-\alpha}
\right)
\frac{1-e^{-(2\mu + \beta-\alpha)t}}{2\mu+\beta-\alpha}
\\
&
\quad
+
\theta_Q^2
\left(
\frac{\gamma \alpha \mu (\nu_\infty- \nu^*) }{2(\beta-\alpha)^2}
+
\frac{\mu(\nu_\infty - \nu^*) }{\beta-\alpha}
+
\frac{\nu_\infty }{2}
+
\frac{\nu_\infty \beta^2 }{2(\beta - \alpha)^2}
\right)\frac{1 - e^{-2\mu t}}{2\mu}
.
\end{align*}
By Lemma~\ref{complemma}, we have that $\hat{\mathcal{M}}^\infty_0(\theta_\nu, \theta_Q, t) \leq \hat{\mathcal{M}}^\infty(\theta_\nu, \theta_Q, t) \leq \hat{\mathcal{M}}^\infty_1(\theta_\nu, \theta_Q, t)$ and since $\hat{\mathcal{M}}^\infty_\gamma(\theta_\nu, \theta_Q, t) = e^{G_\gamma(\theta_\nu, \theta_Q, t)}$, we have completed the proof of the joint moment generating function bounds when $\beta \ne \alpha$. We now apply this to the two marginal generating functions by setting the opposite space parameter to 0. That is, for the intensity we let $\theta_Q = 0$, yielding
\begin{align*}
\hat{\mathcal{M}}^\infty_\gamma(\theta_\nu, 0, t)
&
=
\mathrm{exp}\left(
\nu_0\theta_\nu e^{-(\mu+\beta-\alpha)t}
+
\frac{\theta_\nu^2}2
\left(
\gamma \alpha \mu (\nu_\infty - \nu^*) + \alpha^2 \nu_\infty
\right)
\frac{1 - e^{-2(\mu + \beta - \alpha)t}}{2(\mu + \beta - \alpha)}
\right)
,
\end{align*}
whereas for the queue we take $\theta_\nu = 0$ and receive
\begin{align*}
\hat{\mathcal{M}}^\infty_\gamma(0,\theta_Q, t)
&
=
\mathrm{exp}\Bigg(
\frac{\nu_0\theta_Q}{\beta-\alpha}\left(e^{-\mu t} - e^{-(\mu+\beta-\alpha)t}\right)
+
\frac{\theta_Q^2}{(\beta - \alpha)^2}
\left(
\frac{\gamma \alpha \mu (\nu_\infty - \nu^*) }{2}
+
\frac{\alpha^2 \nu_\infty }{2 }
\right)
\frac{1 - e^{-2(\mu + \beta - \alpha)t}}{2(\mu + \beta - \alpha)}
\\
&
\quad
+
Q_0 \theta_Q e^{-\mu t}
-
\frac{\theta_Q^2}{\beta - \alpha}
\left(
\left(
\frac{\gamma \alpha\mu}{\beta-\alpha}
+
\mu
\right)(\nu_\infty - \nu^*)
+
\frac{\alpha \beta \nu_\infty}{\beta-\alpha}
\right)
\frac{1-e^{-(2\mu + \beta-\alpha)t}}{2\mu+\beta-\alpha}
\\
&
\quad
+
\theta_Q^2
\left(
\frac{\gamma \alpha \mu (\nu_\infty- \nu^*) }{2(\beta-\alpha)^2}
+
\frac{\mu(\nu_\infty - \nu^*) }{\beta-\alpha}
+
\frac{\nu_\infty }{2}
+
\frac{\nu_\infty \beta^2 }{2(\beta - \alpha)^2}
\right)\frac{1 - e^{-2\mu t}}{2\mu}
\Bigg)
.
\end{align*}
Now if $\beta = \alpha$, the solution to the characteristic ODE for $\theta_\nu$ is instead
$$
\theta_\nu = x e^{\mu z} - y z e^{\mu z}
,
$$
whereas the solutions for $\theta_Q$ and $t$ are unchanged: $\theta_Q = y e^{\mu z}$ and $t = z$. This then implies that ODE for $g$ is given by
\begin{align*}
\frac{\mathrm{d}g}{\mathrm{d}z}(x,y,z)
&
=
\left(\frac{\gamma \alpha \mu }{2}
\left(  x^2 e^{2\mu z} - 2x y z e^{2\mu z} + y^2z^2 e^{2\mu z} \right)
+
\mu \left( x y e^{2\mu z} - y^2 z e^{2\mu z}\right)
\right) (\nu_\infty - \nu^*)
\\
&
\quad
+
\frac{\nu_\infty }{2 }
\left(
(2y^2 + \alpha^2 x^2 + 2\alpha x y) e^{2\mu z}
-
2(\alpha^2 x y + \alpha y^2) z e^{2\mu z}
+
\alpha^2 y^2 z^2 e^{2\mu z}
\right)
,
\end{align*}
which yields a solution of
\begin{align*}
g(x,y,z)
&
=
x\nu_0
+
y Q_0
+
\Bigg(
\left(\frac{\gamma \alpha  x^2}{2}  +  x y\right) \frac{e^{2\mu z} - 1}{2}
-
\left(\gamma \alpha x y  +  y^2 \right) \frac{e^{2 \mu z}(2\mu z - 1) + 1 }{4 \mu}
\\
&
\quad
+
\frac{\gamma \alpha y^2}{2}
\left(
\frac{e^{2\mu z}\left(2\mu z(\mu z-1)+1\right) - 1}{4 \mu^2}
\right)
\Bigg) (\nu_\infty - \nu^*)
+
\frac{\nu_\infty }{2 }
\Bigg(
\left(2y^2 + \alpha^2 x^2 + 2\alpha x y \right)  \frac{e^{2\mu z} - 1}{2\mu}
\\
&
\quad
-
2(\alpha^2 x y + \alpha y^2) \left(\frac{e^{2 \mu z}(2\mu z - 1) + 1 }{4 \mu^2}\right)
+
\alpha^2 y^2   \left(  \frac{e^{2\mu z}\left(2\mu z(\mu z-1)+1\right) - 1}{4 \mu^3} \right)
\Bigg)
.
\end{align*}
In this case the inverse solutions are
$$
z = t,\qquad
y = \theta_Q e^{-\mu t},\qquad
x = \theta_\nu e^{-\mu t} + \theta_Q t e^{- \mu t}
,
$$
and so $G_\gamma(\theta_\nu, \theta_Q, t)$ is given by
\begin{align*}
G_\gamma(\theta_\nu, \theta_Q, t)
&=
g(\theta_\nu e^{-\mu t} + \theta_Q t e^{- \mu t}, \theta_Q e^{-\mu z}, t)
\\
&
=
\nu_0 \theta_\nu e^{-\mu t} + \nu_0 \theta_Q t e^{- \mu t}
+
Q_0 \theta_Q e^{-\mu t}
+
\Bigg(
\left(\frac{\gamma \alpha  (\theta_\nu + \theta_Q t)^2}{2}  +  \theta_\nu \theta_Q + \theta_Q^2 t\right) \frac{1 - e^{-2\mu t}}{2}
\\
&
\quad
-
\left(\gamma \alpha (\theta_\nu \theta_Q + \theta_Q^2 t)  +  \theta_Q^2 \right) \frac{2\mu t - 1 + e^{-2\mu t} }{4 \mu}
+
\frac{\gamma \alpha \theta_Q^2}{2}
\left(
\frac{2\mu t(\mu t-1)+1 - e^{-2\mu t}}{4 \mu^2}
\right)
\Bigg) (\nu_\infty - \nu^*)
\\
&
\quad
+
\frac{\nu_\infty }{2 }
\Bigg(
\left(2\theta_Q^2 + \alpha^2 \theta_\nu^2 + 2\alpha^2 \theta_\nu \theta_Q t + \alpha^2 \theta_Q^2 t^2 + 2 \alpha \theta_\nu \theta_\nu + 2\alpha \theta_Q^2 t \right) \frac{1 - e^{-2\mu t}}{2\mu}
\\
&
\quad
-
2 \left(\alpha^2 \theta_\nu \theta_Q + \alpha^2 \theta_Q^2 t + \alpha \theta_Q^2 \right) \left(\frac{2\mu t - 1 + e^{-2\mu t} }{4 \mu^2}\right)
+
\alpha^2 \theta_Q^2 \left(  \frac{ 2\mu t(\mu t-1) +1  - e^{-2\mu t}}{4 \mu^3} \right)
\Bigg)
.
\end{align*}
By taking $\hat{\mathcal{M}}^\infty_\gamma(\theta_\nu, \theta_Q, t) = e^{G_\gamma(\theta_\nu, \theta_Q, t)}$, we complete the proof.
\end{proof}
\end{theorem}

As a consequence of these diffusion approximations, we can give normally distributed approximations for the steady-state distributions of the HESEP intensity and queue length. These are stated below in Corollary~\ref{diffss} again in terms of $\gamma$. One can note that the approximate intensity variance in Equation~\ref{varnugamma} can be used to provide upper and lower bounds on the HESEP variance that may be tighter than the bounds from the ESEP and the Hawkes process in Proposition~\ref{intenOrder}

\begin{corollary}\label{diffss}
Let $(\nu_t, Q_{t,\nu})$ be an HESEP with baseline intensity $\nu^* > 0$, intensity jump $\alpha > 0$, decay rate $\beta > 0$, and rate of exponential service $\mu > 0$, with $\mu + \beta > \alpha$. Then, the steady-state distributions of processes $\nu_t$ and $Q_{t,\nu}$ are approximated by the random variables $X_\nu(\gamma) \sim N(\nu_\infty, \sigma^2_\nu(\gamma))$ and $X_Q(\gamma) \sim N(\frac{\nu_\infty}{\mu}, \sigma^2_Q(\gamma))$,  respectively, where
\begin{align}
\sigma^2_\nu(\gamma)
=
\frac{\gamma \alpha \mu (\nu_\infty - \nu^*) + \alpha^2 \nu_\infty }{2(\mu+\beta-\alpha)}
\label{varnugamma}
,
\end{align}
and if $\beta \ne \alpha$ then
\begin{align}
\sigma^2_Q(\gamma)
&
=
\frac{\gamma \alpha \mu (\nu_\infty - \nu^*) + \alpha^2 \nu_\infty }{2(\beta-\alpha)^2(\mu+\beta-\alpha)}
-
\frac{\left(2\gamma \alpha\mu + 2\mu(\beta-\alpha)\right)(\nu_\infty - \nu^*) + 2\alpha \beta \nu_\infty}{(\beta-\alpha)^2(2\mu+\beta-\alpha)}
\nonumber
\\
&
\quad
+
\frac{\gamma \alpha \mu (\nu_\infty- \nu^*) + \nu_\infty \beta^2  }{2\mu(\beta-\alpha)^2}
+
\frac{\nu_\infty - \nu^* }{\beta-\alpha}
+
\frac{\nu_\infty }{2\mu}
,
\end{align}
whereas if $\beta = \alpha$ then
\begin{align}
\sigma^2_Q(\gamma)
=
\left(
\frac{1}{2\mu} + \frac{\gamma\alpha}{4\mu^2}
\right)(\nu_\infty - \nu^*)
+
\left(
\frac{1}{\mu} + \frac{\alpha }{2 \mu^2} + \frac{\alpha^2}{4\mu^3}
\right)\nu_\infty
,
\end{align}
with $\nu_\infty = \frac{(\mu + \beta)\nu^*}{\mu+\beta-\alpha}$ and $\gamma \in [0,1]$.
\end{corollary}

In Figures~\ref{diffhist100} and~\ref{diffhist} we plot the simulated steady-state distributions of an HESEP with large baseline intensities, as calculated from 100,000 replications. We then also plot the densities corresponding to the upper and lower approximate diffusion distributions as well as an additional candidate approximation with $\gamma = \frac{\mu}{\mu+\beta}$. We motivate this choice by a ratio of mean approximations of the terms in Equation~\ref{diffBound}:
$$
\frac{\frac{(\nu_\infty - \nu^*)^2}{\frac{\nu_\infty}\mu}}{\alpha(\nu_\infty - \nu^*)}
=
\frac{\mu (\nu_\infty - \nu^*)}{\alpha \nu_\infty}
=
\frac{\mu}{\mu+\beta}
.
$$
 In Figure~\ref{diffhist100} the baseline intensity is equal to 100, whereas in Figure~\ref{diffhist} it is 1,000. While there are known limitations of Gaussian approximations for queueing processes such as is discussed in \citet{massey2013gaussian}, we see that these approximations appear to be quite close, particularly so for the $\nu^* = 1,000$ case. The upper and lower bounds predictably over- and under-approximate the tails, while the case of $\gamma = \frac{\mu}{\mu+\beta}$ closely mimics the true distribution.



 \begin{figure}[h]
\centering
\includegraphics[width=.5\textwidth]{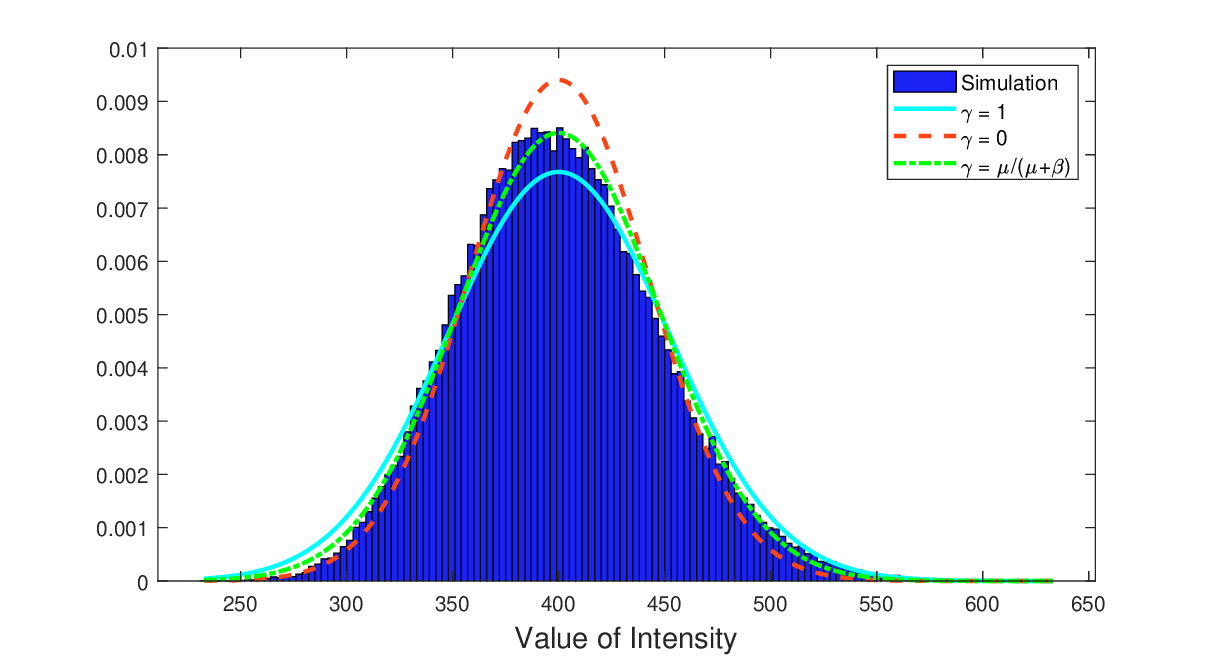}~\hspace{-.35in}~\includegraphics[width=.5\textwidth]{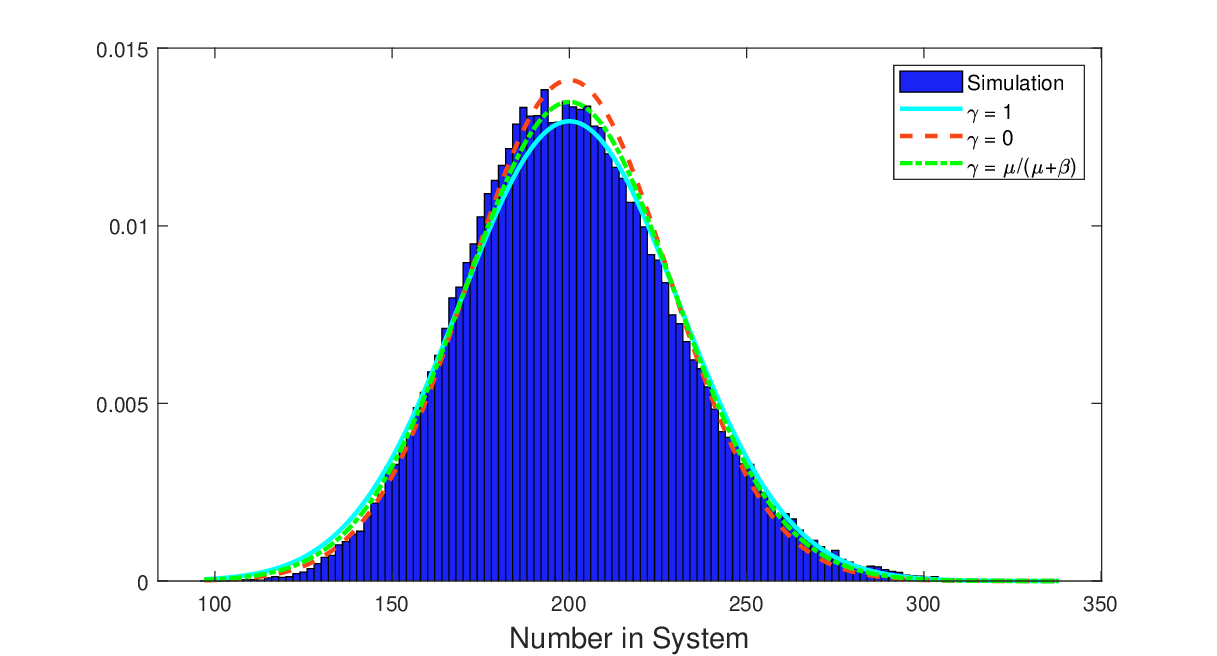}
\caption{Histogram comparing the simulated steady-state HESEP intensity (left) and queue (right) to their diffusion approximations evaluated at multiple values of $\gamma$, where $\nu^* = 100$, $\alpha = 3$, $\beta = 2$, and $\mu = 2$.} \label{diffhist100}
\end{figure}

 \begin{figure}[h]
\centering
\includegraphics[width=.5\textwidth]{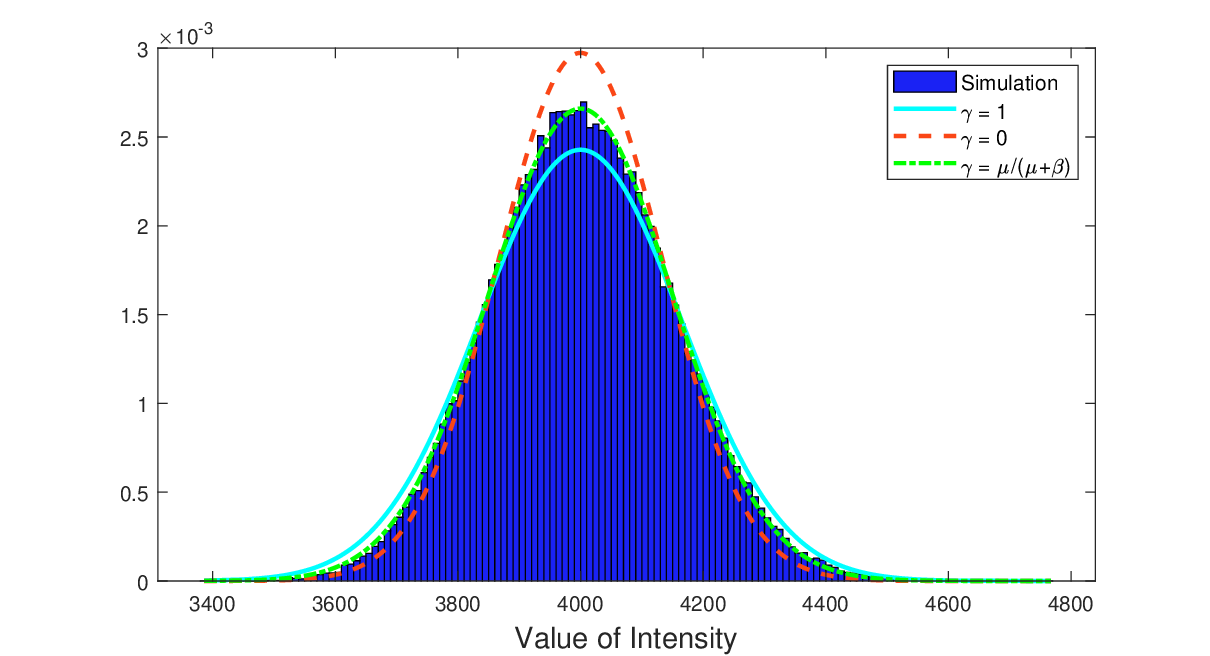}~\hspace{-.35in}~\includegraphics[width=.5\textwidth]{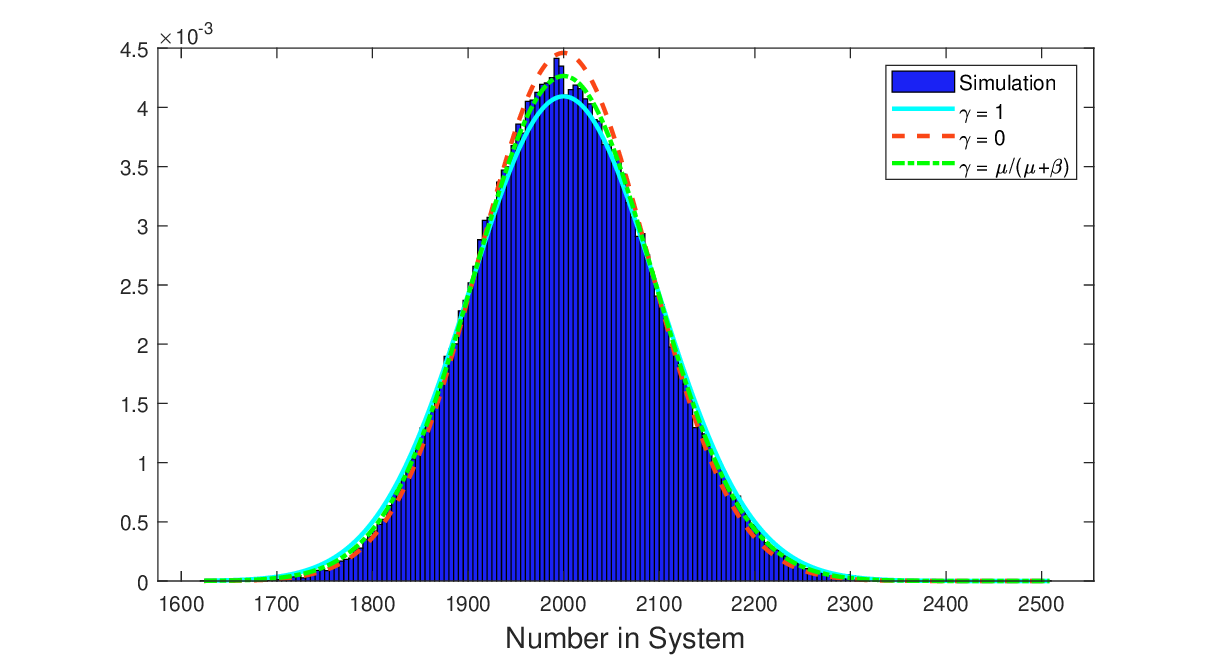}
\caption{Histogram comparing the simulated steady-state HESEP intensity (left) and queue (right) to their diffusion approximations evaluated at multiple values of $\gamma$, where $\nu^* = 1,000$, $\alpha = 3$, $\beta = 2$, and $\mu = 2$.} \label{diffhist}
\end{figure}

\section{An Ephemeral Self-Exciting Process with Finite Capacity and Blocking}\label{subsecBlock}

Drawing inspiration from the works that originated queueing theory, we will now consider the change in the ESEP if there is an upper bound on the total number of active exciters. That is, we suppose that there is a finite capacity and no excess buffer beyond them, so that   any entities that arrive and find the system full are blocked from entry, thus not registering an arrival nor causing any excitement. As an employee of the Copenhagen Telephone company, A.K. Erlang developed these pioneering queueing models to determine the probability that a call would be blocked based on the capacity of the telephone network trunk line. Often referred to as the Erlang-B model, this queueing system remains relevant not just modern telecommunication systems, but broadly across industries as varied as healthcare operations and transportation. For English translations of the seminal Erlang papers and a biography of the author, see \citet{erlang1948life}. In those original works, Erlang supposed that calls arrive perfectly independently, that they have no influence or relationship with one another. In the remainder of this subsection we investigate the scenario where these calls instead exhibit self-excitement, which is a potential explanation for the over-dispersion that has been seen in industrial call center data, as detailed in e.g. \citet{ibrahim2016modeling}. Another potential application for this model is a website that may receive viral traffic but is also liable to crash if there are too many simultaneous visitors. Additionally, this finite capacity model could be used to represent a restaurant that becomes more enticing the more patrons it has in its limited seating area, like we have discussed in the introduction. To begin, we find the steady-state distribution of this process in Proposition~\ref{blockprop}. Drawing further inspiration from Erlang's work, we will refer to this finite capacity ESEP model as the \textit{blocking ephemerally self-exciting process} (ESEP-B).

\begin{proposition}\label{blockprop}
Let $\eta_t^\mathrm{B} = \eta^* + \alpha Q_t^\mathrm{B}$ be a ESEP-B, with baseline intensity $\eta^* > 0$, intensity jump $\alpha > 0$, expiration rate $\beta > \alpha$, and capacity $c \in \mathbb{Z}^+$. That is, if $Q_t^\mathrm{B} = c$ any arrivals that occur will be blocked and not recorded. Then, the steady-state distribution of the active number in system is given by
\begin{align}
\PP{Q_\infty^\mathrm{B} = n}
=
\frac{ \PP{Q_\infty^\eta = n}  }{1 - I_{\frac{\alpha}{\beta}}\left(c+1, \frac{\eta^*}{\alpha}\right)}
=
\frac{
\Gamma\left(n + \frac{\eta^*}{\alpha}\right)
\left(\frac{\beta - \alpha}{\beta}\right)^{\frac{\eta^*}{\alpha}}
 \left(\frac{\alpha}{\beta}\right)^{n} }{\Gamma\left(\frac{\eta^*}{\alpha}\right) n!\left(1 - I_{\frac{\alpha}{\beta}}\left(c+1, \frac{\eta^*}{\alpha}\right)\right)}
,
\end{align}
for $0 \leq n \leq c$ and 0 otherwise, where $\PP{Q_\infty^\eta = n}$ is as stated in Theorem~\ref{negbinthm}. Furthermore, the mean and variance of the number in system are given by
\begin{align}
\E{Q_\infty^\mathrm{B}}
=
 \frac{\eta_\infty}{\beta}
\left(
\frac{
1 - I_{\frac{\alpha}{\beta}}\left(c, \frac{\eta^*+\alpha}{\alpha}\right)
}
{
1 - I_{\frac{\alpha}{\beta}}\left(c+1, \frac{\eta^*}{\alpha}\right)
}
\right)
,
\end{align}
\begin{align}
\Var{Q_\infty^\mathrm{B}}
&
=
 \frac{\eta_\infty}{\beta}
 \left(
 \frac{\eta_\infty}{\beta}
 +
 \frac{\alpha}{\beta - \alpha}
 \right)
 \left(
 \frac{
 1 - I_{\frac{\alpha}{\beta}}\left(c-1, \frac{\eta^*+ 2\alpha}{\alpha}\right)
 }
 {
1 - I_{\frac{\alpha}{\beta}}\left(c+1, \frac{\eta^*}{\alpha}\right)
 }
 \right)
-
 \frac{\eta_\infty^2}{\beta^2}
\left(
\frac{
1 - I_{\frac{\alpha}{\beta}}\left(c, \frac{\eta^*+\alpha}{\alpha}\right)
}
{
1 - I_{\frac{\alpha}{\beta}}\left(c+1, \frac{\eta^*}{\alpha}\right)
}
\right)^2
\nonumber
\\
&
\quad
+
 \frac{\eta_\infty}{\beta}
\left(
\frac{
1 - I_{\frac{\alpha}{\beta}}\left(c, \frac{\eta^*+\alpha}{\alpha}\right)
}
{
1 - I_{\frac{\alpha}{\beta}}\left(c+1, \frac{\eta^*}{\alpha}\right)
}
\right),
\end{align}
where $\eta_\infty = \frac{\beta \eta^*}{\beta - \alpha}$  and $I_z(a,b) = \frac{\Gamma (a+b)}{\Gamma (a) \Gamma (b)} \int_0^z x^{a-1}(1-x)^{b-1} \mathrm{d}x$ for $z \in [0,1]$, $a > 0$ and $b > 0$ is the regularized incomplete beta function.
\begin{proof}
To show each of these, we first note that for $k \in \mathbb{Z}^+$, $x > 0$, and $p \in (0,1)$,
\begin{align}\label{sumEQ}
\sum_{n=0}^k
\frac{\Gamma\left(n + x\right)}{\Gamma\left(x\right) n! }
\left(1 - p\right)^{x}
p^n
=
1 - I_{p}\left(k+1, x\right).
\end{align}
Hence, we can use Equation~\ref{sumEQ} to see that
$$
\sum_{n=0}^c
\PP{Q_\infty^\eta = n}
=
\sum_{n=0}^c
\frac{\Gamma\left(n + \frac{\eta^*}{\alpha}\right)}{\Gamma\left(\frac{\eta^*}{\alpha}\right) n! }
\left(\frac{\beta - \alpha}{\beta}\right)^{\frac{\eta^*}{\alpha}}
\left(\frac{\alpha}{\beta}\right)^n
=
1 - I_{\frac{\alpha}{\beta}}\left(c+1, \frac{\eta^*}{\alpha}\right)
.
$$
Because the ESEP is a birth-death process it is reversible. Thus, by truncation we achieve the steady-state distribution, see e.g. Corollary 1.10 in \citet{kelly2011reversibility}. Then, the steady-state mean of the number in system is given by
\begin{align*}
\E{Q_\infty^B}
&
=
\sum_{n=1}^c
\frac{
n
\Gamma\left(n + \frac{\eta^*}{\alpha}\right)
\left(\frac{\beta - \alpha}{\beta}\right)^{\frac{\eta^*}{\alpha}}
 \left(\frac{\alpha}{\beta}\right)^{n} }{\Gamma\left(\frac{\eta^*}{\alpha}\right) n!\left(1 - I_{\frac{\alpha}{\beta}}\left(c+1, \frac{\eta^*}{\alpha}\right)\right)}
 \\
 &
 =
 \frac{
 \frac{\eta^*}{\beta-\alpha}
 }
 {
1 - I_{\frac{\alpha}{\beta}}\left(c+1, \frac{\eta^*}{\alpha}\right)
 }
\sum_{n=1}^c
\frac{
\Gamma\left(n - 1 + \frac{\eta^* + \alpha}{\alpha}\right)
\left(\frac{\beta - \alpha}{\beta}\right)^{\frac{\eta^* + \alpha}{\alpha}}
 \left(\frac{\alpha}{\beta}\right)^{n-1} }{\Gamma\left(\frac{\eta^*+\alpha}{\alpha}\right) (n-1)!}
  \\
 &
 =
 \frac{\eta_\infty}{\beta}
\left(
\frac{
1 - I_{\frac{\alpha}{\beta}}\left(c, \frac{\eta^*+\alpha}{\alpha}\right)
}
{
1 - I_{\frac{\alpha}{\beta}}\left(c+1, \frac{\eta^*}{\alpha}\right)
}
\right),
\end{align*}
where we have again used Equation~\ref{sumEQ} to simplify the summation. Likewise, the second moment in steady-state can be written
\begin{align*}
\E{\left(Q_\infty^B\right)^2}
&
=
\sum_{n=1}^c
\frac{
n^2
\Gamma\left(n + \frac{\eta^*}{\alpha}\right)
\left(\frac{\beta - \alpha}{\beta}\right)^{\frac{\eta^*}{\alpha}}
 \left(\frac{\alpha}{\beta}\right)^{n} }{\Gamma\left(\frac{\eta^*}{\alpha}\right) n!\left(1 - I_{\frac{\alpha}{\beta}}\left(c+1, \frac{\eta^*}{\alpha}\right)\right)}
 \\
 &
 =
 \sum_{n=2}^c
 \frac{
\Gamma\left(n + \frac{\eta^*}{\alpha}\right)
\left(\frac{\beta - \alpha}{\beta}\right)^{\frac{\eta^*}{\alpha}}
 \left(\frac{\alpha}{\beta}\right)^{n} }{\Gamma\left(\frac{\eta^*}{\alpha}\right) (n-2)!\left(1 - I_{\frac{\alpha}{\beta}}\left(c+1, \frac{\eta^*}{\alpha}\right)\right)}
 +
 \sum_{n=1}^c
 \frac{
\Gamma\left(n + \frac{\eta^*}{\alpha}\right)
\left(\frac{\beta - \alpha}{\beta}\right)^{\frac{\eta^*}{\alpha}}
 \left(\frac{\alpha}{\beta}\right)^{n} }{\Gamma\left(\frac{\eta^*}{\alpha}\right) (n-1)!\left(1 - I_{\frac{\alpha}{\beta}}\left(c+1, \frac{\eta^*}{\alpha}\right)\right)}
  \\
 &
 =
 \frac{\eta^*(\eta^* + \alpha)}{(\beta-\alpha)^2}
\sum_{n=2}^c
\frac{
\Gamma\left(n - 2 + \frac{\eta^*+ 2\alpha}{\alpha}\right)
\left(\frac{\beta - \alpha}{\beta}\right)^{\frac{\eta^* + 2\alpha}{\alpha}}
 \left(\frac{\alpha}{\beta}\right)^{n-2} }{\Gamma\left(\frac{\eta^* + 2 \alpha}{\alpha}\right) (n-2)!\left(1 - I_{\frac{\alpha}{\beta}}\left(c+1, \frac{\eta^*}{\alpha}\right)\right)}
 +
 \E{Q_\infty^\mathrm{B}}
 \\
 &
 =
 \frac{\eta_\infty}{\beta}
 \left(
 \frac{\eta_\infty}{\beta}
 +
 \frac{\alpha}{\beta - \alpha}
 \right)
 \frac{
 1 - I_{\frac{\alpha}{\beta}}\left(c-1, \frac{\eta^*+2\alpha}{\alpha}\right)
 }{
 1 - I_{\frac{\alpha}{\beta}}\left(c+1, \frac{\eta^*}{\alpha}\right)
 }
  +
 \E{Q_\infty^\mathrm{B}}
\end{align*}
where once more these sums have been simplified through Equation~\ref{sumEQ}.
\end{proof}
\end{proposition}

As a demonstration of these findings, we now plot both the steady-state distribution and the mean and variance of this blocking system in Figure~\ref{blockfig}. As can be observed in the figure, this system remains over-dispersed even when truncated. We can observe further that this holds in generality as follows. To observe this, we state two known properties of the regularized incomplete beta function:
\begin{align}
I_z(a,b) = I_z(a+1,b) + \frac{z^a(1-z)^b}{a B(a,b)}, & & I_z(a,b+1) = I_z(a,b) + \frac{z^a(1-z)^b}{b B(a,b)} ,
\label{betaI}
\end{align}
where $B(a,b) = \frac{\Gamma (a+b)}{\Gamma(a)\Gamma(b)}$ is the beta function. Using these together, we can observe that
$$
I_z(a,b) > I_z(a+1,b-1).
$$
Thus, we can see that $I_{\frac{\alpha}{\beta}}\left(c+1, \frac{\eta^*}{\alpha}\right) < I_{\frac{\alpha}{\beta}}\left(c, \frac{\eta^* + \alpha}{\alpha}\right) < I_{\frac{\alpha}{\beta}}\left(c-1, \frac{\eta^*+2\alpha}{\alpha}\right) < 1$, and this implies
$$
1
>
\frac{
1 - I_{\frac{\alpha}{\beta}}\left(c, \frac{\eta^*+\alpha}{\alpha}\right)
}
{
1 - I_{\frac{\alpha}{\beta}}\left(c+1, \frac{\eta^*}{\alpha}\right)
}
>
\frac{
1 - I_{\frac{\alpha}{\beta}}\left(c-1, \frac{\eta^*+2\alpha}{\alpha}\right)
}
{
1 - I_{\frac{\alpha}{\beta}}\left(c+1, \frac{\eta^*}{\alpha}\right)
}
.
$$
We now note that the variance is written as the sum of the mean and a positive term and is thus over-dispersed.

\begin{figure}[h]
\begin{center}	
\includegraphics[width=.5\textwidth]{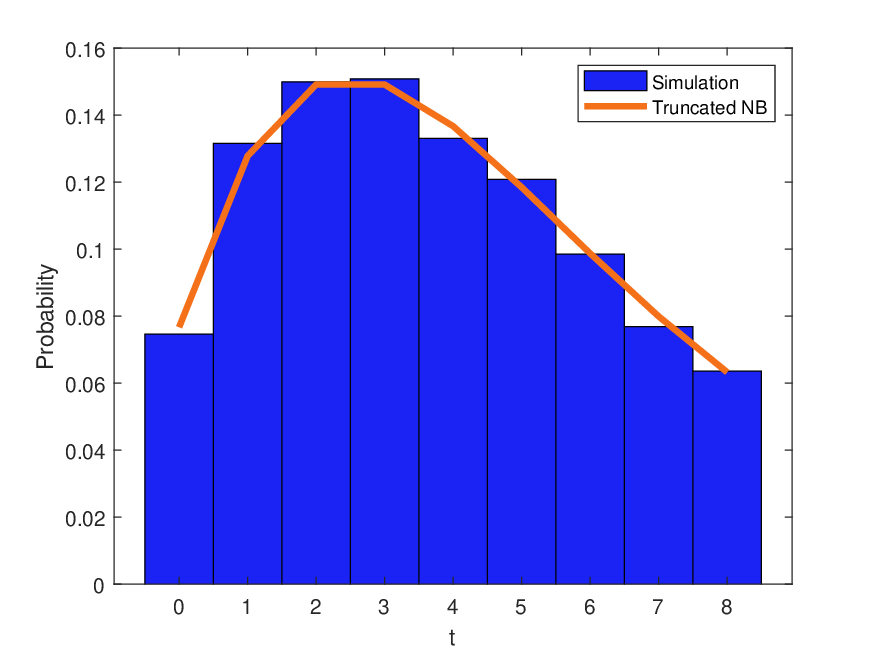}~\hspace{-.1in}~\includegraphics[width=.5\textwidth]{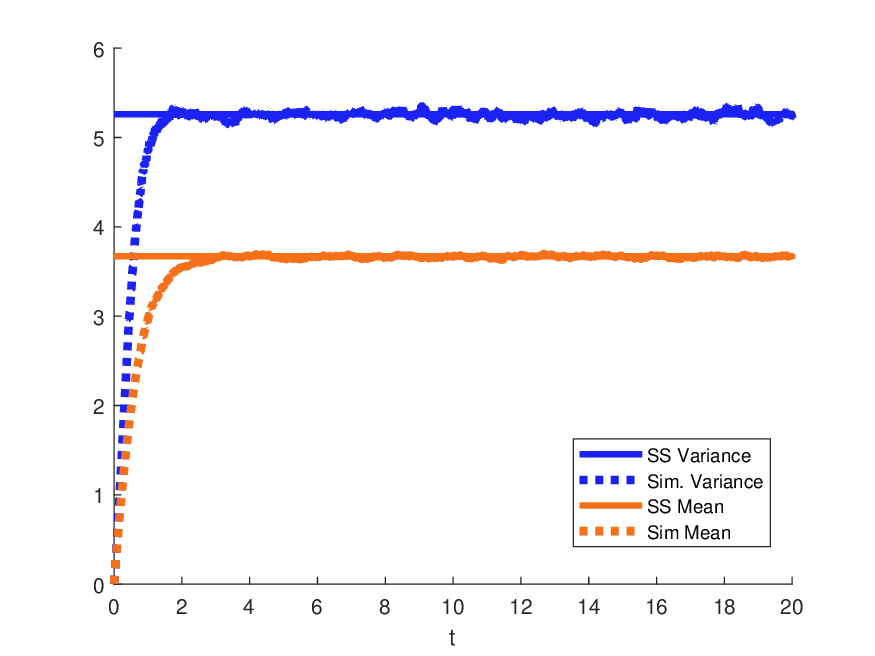}
\caption{Steady-state distribution (left) and mean and variance (right) of the ESEP-B with $\eta^* = 5$, $\alpha = 2$, $\beta = 3$, and $c = 8$ (Right), based on 10,000 replications. } \label{blockfig}
\end{center}
\end{figure}


We can also note that in the classical Erlang-B model, the famous ``Poisson arrivals see time averages'' (PASTA) result implies that the steady-state fraction of arrivals that are blocked is equal to the probability that the queue is at capacity in steady-state, see \citet{wolff1982poisson}. This is not so for the ESEP-B, as the arrival rate is state-dependent and, more specifically, increases with the queue length. However, in Proposition~\ref{hasta} we find that an equivalent result holds asymptotically as the baseline intensity and the capacity grow large simultaneously.  We note that large baseline intensity and capacity are realistic scenarios for many practically relevant applications, including the aforementioned website crashing scenario.

\begin{proposition}\label{hasta}
Let $\eta_t^\mathrm{B} = \eta^* + \alpha Q_t^\mathrm{B}$ be an ESEP-B, with baseline intensity $\eta^* > 0$, intensity jump $\alpha > 0$, exponential service rate $\beta > \alpha$, and capacity $c \in \mathbb{Z}^+$. Then, the fraction of arrivals in steady-state that are blocked $\pi_{\mathrm{B}}$ is given by
\begin{align}
\pi_{\mathrm{B}}
&
=
\frac{(\eta^* + \alpha c)\PP{Q_\infty^\eta = c }}{\sum_{k=0}^c (\eta^* + \alpha k) \PP{Q_\infty^\eta = k}}
=
\frac{(\eta^* + \alpha c) \PP{Q_\infty^\mathrm{B} = c}}{\eta^* + \alpha \E{Q_\infty^\mathrm{B}}}
\label{hastaVal}
,
\end{align}
where $\PP{Q_\infty^\eta = k}$ is as given in Theorem~\ref{negbinthm} and $\PP{Q_\infty^\mathrm{B} = c}$ and $\E{Q_\infty^\mathrm{B}}$ are as given in Proposition~\ref{blockprop}. Moreover, if the baseline intensity and the capacity are redefined to be $\eta^* n$ and $c n$ for $n \in \mathbb{Z}^+$, then
\begin{align}
\frac{\pi_{\mathrm{B}}}{\PP{Q_\infty^\mathrm{B} = c}}
&
\longrightarrow
1
\label{hastaConv}
,
\end{align}
as $n \to \infty$.
\begin{proof}
The expression for steady-state fraction of arrivals blocked $\pi_\mathrm{B}$  in Equation~\ref{hastaVal} follows as a direct consequence from observing that the $\eta^* + \alpha k$ is the arrival rate when the queue is in state $k$. We are thus left to show Equation~\ref{hastaConv}. By use of Equation~\ref{hastaVal}, we have that the ratio of $\pi_{B}$ and $\PP{Q_\infty^\mathrm{B} = c}$ is
\begin{align*}
\frac{\pi_\mathrm{B}}{\PP{Q_\infty^\mathrm{B} = c}}
&
=
\frac{\eta^* + \alpha c }{\eta^* + \alpha \E{Q_\infty^\mathrm{B}}}
=
\frac{\eta^* + \alpha c }{\eta^* +  \frac{\alpha \eta^*}{\beta - \alpha}
\left(
\frac{
1 - I_{\frac{\alpha}{\beta}}\left(c, \frac{\eta^*}{\alpha}+1\right)
}
{
1 - I_{\frac{\alpha}{\beta}}\left(c+1, \frac{\eta^*}{\alpha}\right)
}
\right)}
,
\end{align*}
by use of Proposition~\ref{blockprop}. Substituting in the scaled forms of the baseline intensity and capacity $\eta^* n$ and $c n$ and then dividing the numerator and denominator by $c n$, this is
\begin{align*}
\frac{\eta^*n + \alpha cn }{\eta^*n +  \frac{\alpha \eta^*n}{\beta - \alpha}
\left(
\frac{
1 - I_{\frac{\alpha}{\beta}}\left(cn, \frac{\eta^*n}{\alpha}+1\right)
}
{
1 - I_{\frac{\alpha}{\beta}}\left(cn+1, \frac{\eta^*n}{\alpha}\right)
}
\right)}
&
=
\frac{\frac{\eta^*}{c} + \alpha  }{\frac{\eta^*}{c} +  \frac{\eta^*}{c}\left(\frac{\alpha }{\beta - \alpha}\right)
\left(
\frac{
1 - I_{\frac{\alpha}{\beta}}\left(cn, \frac{\eta^*n}{\alpha}+1\right)
}
{
1 - I_{\frac{\alpha}{\beta}}\left(cn+1, \frac{\eta^*n}{\alpha}\right)
}
\right)}
.
\end{align*}
From the definition and symmetry of the regularized incomplete beta function, we can note that the ratio of these functions is such that
\begin{align*}
\frac{
1 - I_{\frac{\alpha}{\beta}}\left(cn, \frac{\eta^*n}{\alpha}+1\right)
}
{
1 - I_{\frac{\alpha}{\beta}}\left(cn+1, \frac{\eta^*n}{\alpha}\right)
}
&
=
\frac{
I_{1-\frac{\alpha}{\beta}}\left(\frac{\eta^*n}{\alpha}+1, cn\right)
}
{
I_{1-\frac{\alpha}{\beta}}\left(\frac{\eta^*n}{\alpha}, cn + 1\right)
}
=
\frac{\alpha c}{\eta^*}
\left(
\frac{
 \int_0^{1-\frac{\alpha}{\beta}} x^{\frac{n\eta^*}{\alpha}}\left(1-x\right)^{cn-1}\mathrm{d}x
}
{
 \int_0^{1-\frac{\alpha}{\beta}} x^{\frac{n\eta^*}{\alpha}-1}\left(1-x\right)^{cn}\mathrm{d}x
}
\right)
.
\end{align*}
We can now recognize an identity for the hypergeometric function $\,_2F_1(a,b;c;z)$, and thus re-express this ratio as
\begin{align*}
\frac{\alpha c}{\eta^*}
\left(
\frac{
 \int_0^{1-\frac{\alpha}{\beta}} x^{\frac{n\eta^*}{\alpha}}\left(1-x\right)^{cn-1}\mathrm{d}x
}
{
 \int_0^{1-\frac{\alpha}{\beta}} x^{\frac{n\eta^*}{\alpha}-1}\left(1-x\right)^{cn}\mathrm{d}x
}
\right)
&
=
\frac{\alpha c}{\eta^*}
\left(
\frac{
\frac{1}{\frac{\eta^* n}{\alpha} + 1} \left(1-\frac{\alpha}{\beta}\right)^{\frac{\eta^* n}{\alpha}+1}\left(\frac{\alpha}{\beta}\right)^{c n} \,_2F_1 \left(c+\frac{\eta^* n}{\alpha}+1,1;cn+2;1-\frac{\alpha}{\beta}\right)
}
{
\frac{1}{\frac{\eta^* n}{\alpha}} \left(1-\frac{\alpha}{\beta}\right)^{\frac{\eta^* n}{\alpha}}\left(\frac{\alpha}{\beta}\right)^{c n + 1} \,_2F_1 \left(c+\frac{\eta^* n}{\alpha}+1,1;cn+1;1-\frac{\alpha}{\beta}\right)
}
\right)
\\
&
=
\frac{\alpha c}{\eta^*}
\left(
\frac{\eta^* n}{\eta^* n + \alpha}
\right)
\left(
\frac{\beta - \alpha}{\alpha}
\right)
\frac{
\,_2F_1 \left(c+\frac{\eta^* n}{\alpha}+1,1;cn+2;1-\frac{\alpha}{\beta}\right)
}
{
 \,_2F_1 \left(c+\frac{\eta^* n}{\alpha}+1,1;cn+1;1-\frac{\alpha}{\beta}\right)
}
.
\end{align*}
As $n \to \infty$, this yields
$$
\frac{\alpha c}{\eta^*}
\left(
\frac{\eta^* n}{\eta^* n + \alpha}
\right)
\left(
\frac{\beta - \alpha}{\beta}
\right)
\frac{
\,_2F_1 \left(c+\frac{\eta^* n}{\alpha}+1,1;cn+2;1-\frac{\alpha}{\beta}\right)
}
{
 \,_2F_1 \left(c+\frac{\eta^* n}{\alpha}+1,1;cn+1;1-\frac{\alpha}{\beta}\right)
}
\longrightarrow
\frac{\alpha c}{\eta^*}
\left(
\frac{\beta - \alpha}{\alpha}
\right)
,
$$
which thus implies that
$$
\frac{\frac{\eta^*}{c} + \alpha  }{\frac{\eta^*}{c} +  \frac{\eta^*}{c}\left(\frac{\alpha }{\beta - \alpha}\right)
\left(
\frac{
1 - I_{\frac{\alpha}{\beta}}\left(cn, \frac{\eta^*n}{\alpha}+1\right)
}
{
1 - I_{\frac{\alpha}{\beta}}\left(cn+1, \frac{\eta^*n}{\alpha}\right)
}
\right)}
\longrightarrow
\frac{\frac{\eta^*}{c} + \alpha  }{\frac{\eta^*}{c} +  \frac{\eta^*}{c}\left(\frac{\alpha }{\beta - \alpha}\right)
\frac{\alpha c}{\eta^*}
\left(
\frac{\beta - \alpha}{\alpha}
\right)
}
=
1
,
$$
and this completes the proof.
\end{proof}
\end{proposition}

As an example of the convergence stated in Proposition~\ref{hasta}, we compare the probability of the system being at capacity and the fraction of blocked arrivals below in Figure~\ref{blockproblimit}. In this figure, $\eta^*$ and $c$ are increased simultaneously according to a fixed ratio. Although at the initial values it is clear that a PASTA-esque result does not hold, as the baseline intensity and capacity both increase one can see that the two curves tend toward one another in each of the different parameter settings.

\begin{figure}[h]
\begin{center}	
\includegraphics[scale=.85]{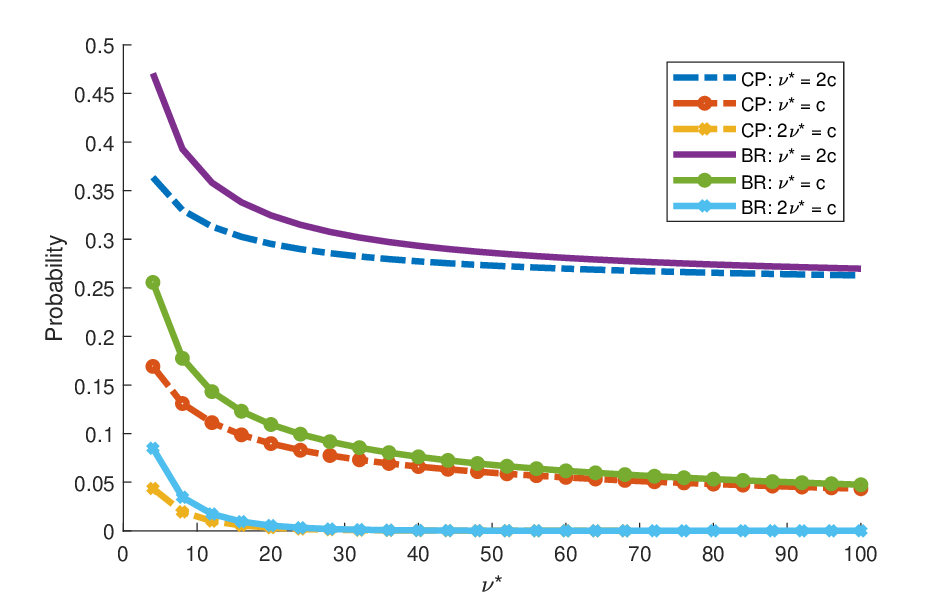}
\caption{Comparison of the ratio of blocked arrivals (BR) and the probability of system being at capacity (CP) when increasing $\eta^*$ and $c$ simultaneously, where $\alpha = 2$ and $\beta = 3$.} \label{blockproblimit}
\end{center}
\end{figure}

\section{Proof of Proposition \ref{Npgf}}\label{pgfProof}
\begin{proof}
Using Proposition~\ref{QDpgf}, we proceed through use of exponential identities for the hyperbolic functions. Specifically, we will make use of the following:
\begin{align}
\tanh(x)
&
=
\frac{e^x - e^{-x}}{e^x + e^{-x}}
,
\label{tanh}
\\
\cosh(x)
&
=
\frac{e^x + e^{-x}}2
,
\label{cosh}
\intertext{and}
\tanh^{-1}(x)
&
=
\frac{1}{2}
\log\left(
\frac{1 + x}{1-x}
\right)
.
\label{tanhinv}
\end{align}
Using these identities we can further observe that
\begin{align*}
\cosh\left(\tanh^{-1}(x)\right)
&
=
\frac{e^{\tanh^{-1}(x)} +  e^{-\tanh^{-1}(x)}}2
=
\frac{\left(
\frac{1 + x}{1-x}
\right)^{\frac{1}{2}}
+
\left(
\frac{1 - x}{1 + x}
\right)^{\frac{1}{2}}
}2
.
\end{align*}
Now, for any time $t \geq 0$ we can note that $N_t = Q_t + D_t$. Thus, we have that
$$
\E{z^{N_{t}}} = \E{z^{Q_t}z^{D_t}} = G(z,z,t)
,
$$
where $G(z_1,z_2,t)$ is as given in Proposition~\ref{QDpgf}. Setting $z_1 = z_2 = z$ and $D_0 = N_0 - Q_0$, this is
\begin{align}
&
G(z, z, t)
=
z^{N_0 - Q_0} e^{\frac{\eta^*(\beta-\alpha)}{2\alpha}t}
\left(
1
-
\left(
\tanh\left(
\frac{t}2 \sqrt{(\beta+\alpha)^2 - 4\alpha\beta z}
+
\tanh^{-1}\left(\frac{\beta+\alpha-2\alpha z }{\sqrt{(\beta+\alpha)^2 - 4\alpha\beta z}} \right)
\right)
\right)^2
\right)^{\frac{\eta^*}{2\alpha}}
\nonumber
\\
&
\quad
\cdot
\left(
\frac{\beta + \alpha}{2\alpha}
-
\frac{\sqrt{(\beta+\alpha)^2 - 4\alpha\beta z}}{2\alpha}
\tanh \left(\frac{t}2 \sqrt{(\beta+\alpha)^2 - 4\alpha\beta z} + \tanh^{-1}\left(\frac{ \beta+\alpha-2\alpha z}{\sqrt{(\beta+\alpha)^2 - 4\alpha\beta z}} \right) \right)
\right)^{Q_0}
\nonumber
\\
&
\quad
\cdot
\left(
\cosh \left(
\tanh^{-1}\left(\frac{2\alpha z - \beta - \alpha}{\sqrt{(\beta+\alpha)^2 - 4\alpha\beta z}} \right)
\right)
\right)^{\frac{\eta^*}{\alpha}}
.
\label{presimp}
\end{align}
Using the hyperbolic identities and simplifying, this is
\begin{align*}
&
G(z, z, t)
=
z^{N_0 - Q_0} e^{\frac{\eta^*\eta^*(\beta-\alpha)}{2\alpha}t}
\left(
\frac{
2
e^{
\frac{t}{2} \sqrt{(\beta+\alpha)^2 - 4\alpha\beta z}}
}
{
1-\frac{ \beta+\alpha-2\alpha z}{\sqrt{(\beta+\alpha)^2 - 4\alpha\beta z}}
+
\left(
1 + \frac{ \beta+\alpha-2\alpha z}{\sqrt{(\beta+\alpha)^2 - 4\alpha\beta z}}
\right)
e^{
t \sqrt{(\beta+\alpha)^2 - 4\alpha\beta z}}
}
\right)^{\frac{\eta^*}{\alpha}}
\nonumber
\\
&
\quad
\cdot
\left(
\frac{\beta + \alpha}{2\alpha}
+
\frac{\sqrt{(\beta+\alpha)^2 - 4\alpha\beta z}}{2\alpha}
\left(
\frac{
1-\frac{ \beta+\alpha-2\alpha z}{\sqrt{(\beta+\alpha)^2 - 4\alpha\beta z}}
-
\left(
1 + \frac{ \beta+\alpha-2\alpha z}{\sqrt{(\beta+\alpha)^2 - 4\alpha\beta z}}
\right)
e^{
t \sqrt{(\beta+\alpha)^2 - 4\alpha\beta z}}
}
{
1-\frac{ \beta+\alpha-2\alpha z}{\sqrt{(\beta+\alpha)^2 - 4\alpha\beta z}}
+
\left(
1 + \frac{ \beta+\alpha-2\alpha z}{\sqrt{(\beta+\alpha)^2 - 4\alpha\beta z}}
\right)
e^{
t \sqrt{(\beta+\alpha)^2 - 4\alpha\beta z}}
}
\right)
\right)^{Q_0}
,
\end{align*}
which is the stated result. However, the simplifications used to reach this form require multiple parts and several steps and so we can these individually now. We start with the hyperbolic tangent function that appears on the first and second lines of Equation~\ref{presimp}. Using Equations~\ref{tanh} and~\ref{tanhinv}, this is
\begin{align*}
&
-
\tanh \left(\frac{t}2 \sqrt{(\beta+\alpha)^2 - 4\alpha\beta z} + \tanh^{-1}\left(\frac{ \beta+\alpha-2\alpha z}{\sqrt{(\beta+\alpha)^2 - 4\alpha\beta z}} \right) \right)
\\
&
=
-
\frac{
e^{\left(
\frac{t}2 \sqrt{(\beta+\alpha)^2 - 4\alpha\beta z}
+
\frac{1}{2}
\log\left(
\frac{1 + \frac{ \beta+\alpha-2\alpha z}{\sqrt{(\beta+\alpha)^2 - 4\alpha\beta z}}}{1-\frac{ \beta+\alpha-2\alpha z}{\sqrt{(\beta+\alpha)^2 - 4\alpha\beta z}}}
\right)
\right)}
-
e^{-\left(
\frac{t}2 \sqrt{(\beta+\alpha)^2 - 4\alpha\beta z}
+
\frac{1}{2}
\log\left(
\frac{1 + \frac{ \beta+\alpha-2\alpha z}{\sqrt{(\beta+\alpha)^2 - 4\alpha\beta z}}}{1-\frac{ \beta+\alpha-2\alpha z}{\sqrt{(\beta+\alpha)^2 - 4\alpha\beta z}}}
\right)
\right)}
}
{
e^{\left(
\frac{t}2 \sqrt{(\beta+\alpha)^2 - 4\alpha\beta z}
+
\frac{1}{2}
\log\left(
\frac{1 + \frac{ \beta+\alpha-2\alpha z}{\sqrt{(\beta+\alpha)^2 - 4\alpha\beta z}}}{1-\frac{ \beta+\alpha-2\alpha z}{\sqrt{(\beta+\alpha)^2 - 4\alpha\beta z}}}
\right)
\right)}
+
e^{-\left(
\frac{t}2 \sqrt{(\beta+\alpha)^2 - 4\alpha\beta z}
+
\frac{1}{2}
\log\left(
\frac{1 + \frac{ \beta+\alpha-2\alpha z}{\sqrt{(\beta+\alpha)^2 - 4\alpha\beta z}}}{1-\frac{ \beta+\alpha-2\alpha z}{\sqrt{(\beta+\alpha)^2 - 4\alpha\beta z}}}
\right)
\right)}
}
\\
&
=
-
\frac{
e^{\left(
t \sqrt{(\beta+\alpha)^2 - 4\alpha\beta z}
+
\log\left(
\frac{1 + \frac{ \beta+\alpha-2\alpha z}{\sqrt{(\beta+\alpha)^2 - 4\alpha\beta z}}}{1-\frac{ \beta+\alpha-2\alpha z}{\sqrt{(\beta+\alpha)^2 - 4\alpha\beta z}}}
\right)
\right)}
-
1
}
{
e^{\left(
t \sqrt{(\beta+\alpha)^2 - 4\alpha\beta z}
+
\log\left(
\frac{1 + \frac{ \beta+\alpha-2\alpha z}{\sqrt{(\beta+\alpha)^2 - 4\alpha\beta z}}}{1-\frac{ \beta+\alpha-2\alpha z}{\sqrt{(\beta+\alpha)^2 - 4\alpha\beta z}}}
\right)
\right)}
+
1
}
\\
&
=
\frac{
1
-
\left(
\frac{1 + \frac{ \beta+\alpha-2\alpha z}{\sqrt{(\beta+\alpha)^2 - 4\alpha\beta z}}}{1-\frac{ \beta+\alpha-2\alpha z}{\sqrt{(\beta+\alpha)^2 - 4\alpha\beta z}}}
\right)
e^{
t \sqrt{(\beta+\alpha)^2 - 4\alpha\beta z}
}
}
{
1
+
\left(
\frac{1 + \frac{ \beta+\alpha-2\alpha z}{\sqrt{(\beta+\alpha)^2 - 4\alpha\beta z}}}{1-\frac{ \beta+\alpha-2\alpha z}{\sqrt{(\beta+\alpha)^2 - 4\alpha\beta z}}}
\right)
e^{
t \sqrt{(\beta+\alpha)^2 - 4\alpha\beta z}}
}
\\
&
=
\frac{
1-\frac{ \beta+\alpha-2\alpha z}{\sqrt{(\beta+\alpha)^2 - 4\alpha\beta z}}
-
\left(
1 + \frac{ \beta+\alpha-2\alpha z}{\sqrt{(\beta+\alpha)^2 - 4\alpha\beta z}}
\right)
e^{
t \sqrt{(\beta+\alpha)^2 - 4\alpha\beta z}}
}
{
1-\frac{ \beta+\alpha-2\alpha z}{\sqrt{(\beta+\alpha)^2 - 4\alpha\beta z}}
+
\left(
1 + \frac{ \beta+\alpha-2\alpha z}{\sqrt{(\beta+\alpha)^2 - 4\alpha\beta z}}
\right)
e^{
t \sqrt{(\beta+\alpha)^2 - 4\alpha\beta z}}
}
.
\end{align*}
Thus, the second line of Equation~\ref{presimp} simplifies as
\begin{align*}
&
\left(
\frac{\beta + \alpha}{2\alpha}
-
\frac{\sqrt{(\beta+\alpha)^2 - 4\alpha\beta z}}{2\alpha}
\tanh \left(\frac{t}2 \sqrt{(\beta+\alpha)^2 - 4\alpha\beta z} + \tanh^{-1}\left(\frac{ \beta+\alpha-2\alpha z}{\sqrt{(\beta+\alpha)^2 - 4\alpha\beta z}} \right) \right)
\right)^{Q_0}
\\
&
=
\left(
\frac{\beta + \alpha}{2\alpha}
+
\frac{\sqrt{(\beta+\alpha)^2 - 4\alpha\beta z}}{2\alpha}
\left(
\frac{
1-\frac{ \beta+\alpha-2\alpha z}{\sqrt{(\beta+\alpha)^2 - 4\alpha\beta z}}
-
\left(
1 + \frac{ \beta+\alpha-2\alpha z}{\sqrt{(\beta+\alpha)^2 - 4\alpha\beta z}}
\right)
e^{
t \sqrt{(\beta+\alpha)^2 - 4\alpha\beta z}}
}
{
1-\frac{ \beta+\alpha-2\alpha z}{\sqrt{(\beta+\alpha)^2 - 4\alpha\beta z}}
+
\left(
1 + \frac{ \beta+\alpha-2\alpha z}{\sqrt{(\beta+\alpha)^2 - 4\alpha\beta z}}
\right)
e^{
t \sqrt{(\beta+\alpha)^2 - 4\alpha\beta z}}
}
\right)
\right)^{Q_0}
\\
&
=
\left(
\frac{\beta + \alpha}{2\alpha}
+
\frac{
\frac{\sqrt{(\beta+\alpha)^2 - 4\alpha\beta z}}{2\alpha} - \frac{ \beta+\alpha-2\alpha z}{2\alpha}
-
\left(
\frac{\sqrt{(\beta+\alpha)^2 - 4\alpha\beta z}}{2\alpha} + \frac{ \beta+\alpha-2\alpha z}{2\alpha}
\right)
e^{
t \sqrt{(\beta+\alpha)^2 - 4\alpha\beta z}}
}
{
1-\frac{ \beta+\alpha-2\alpha z}{\sqrt{(\beta+\alpha)^2 - 4\alpha\beta z}}
+
\left(
1 + \frac{ \beta+\alpha-2\alpha z}{\sqrt{(\beta+\alpha)^2 - 4\alpha\beta z}}
\right)
e^{
t \sqrt{(\beta+\alpha)^2 - 4\alpha\beta z}}
}
\right)^{Q_0}
\\
&
=
\left(
\frac{
\frac{ (\beta+\alpha)^2-2\beta\alpha z-2\alpha^2z}{2\alpha\sqrt{(\beta+\alpha)^2 - 4\alpha\beta z}}
\left(
e^{
t \sqrt{(\beta+\alpha)^2 - 4\alpha\beta z}}
-
1
\right)
+
\frac{\sqrt{(\beta+\alpha)^2 - 4\alpha\beta z}}{2\alpha} + z
-
\left(
\frac{\sqrt{(\beta+\alpha)^2 - 4\alpha\beta z}}{2\alpha} - z
\right)
e^{
t \sqrt{(\beta+\alpha)^2 - 4\alpha\beta z}}
}
{
1-\frac{ \beta+\alpha-2\alpha z}{\sqrt{(\beta+\alpha)^2 - 4\alpha\beta z}}
+
\left(
1 + \frac{ \beta+\alpha-2\alpha z}{\sqrt{(\beta+\alpha)^2 - 4\alpha\beta z}}
\right)
e^{
t \sqrt{(\beta+\alpha)^2 - 4\alpha\beta z}}
}
\right)^{Q_0}
\\
&
=
\left(
\frac{
\left(
\frac{ (\beta -\alpha) z}{\sqrt{(\beta+\alpha)^2 - 4\alpha\beta z}}
\right)
\left(
e^{
t \sqrt{(\beta+\alpha)^2 - 4\alpha\beta z}}
-
1
\right)
+
 z
+
z
e^{
t \sqrt{(\beta+\alpha)^2 - 4\alpha\beta z}}
}
{
1-\frac{ \beta+\alpha-2\alpha z}{\sqrt{(\beta+\alpha)^2 - 4\alpha\beta z}}
+
\left(
1 + \frac{ \beta+\alpha-2\alpha z}{\sqrt{(\beta+\alpha)^2 - 4\alpha\beta z}}
\right)
e^{
t \sqrt{(\beta+\alpha)^2 - 4\alpha\beta z}}
}
\right)^{Q_0}
.
\end{align*}
Following the same approach, the first line of Equation~\ref{presimp} rearranges to
\begin{align*}
&
\left(
1
-
\left(
\frac{
e^{\left(
\frac{t}2 \sqrt{(\beta+\alpha)^2 - 4\alpha\beta z}
+
\frac{1}{2}
\log\left(
\frac{1 + \frac{ \beta+\alpha-2\alpha z}{\sqrt{(\beta+\alpha)^2 - 4\alpha\beta z}}}{1-\frac{ \beta+\alpha-2\alpha z}{\sqrt{(\beta+\alpha)^2 - 4\alpha\beta z}}}
\right)
\right)}
-
e^{-\left(
\frac{t}2 \sqrt{(\beta+\alpha)^2 - 4\alpha\beta z}
+
\frac{1}{2}
\log\left(
\frac{1 + \frac{ \beta+\alpha-2\alpha z}{\sqrt{(\beta+\alpha)^2 - 4\alpha\beta z}}}{1-\frac{ \beta+\alpha-2\alpha z}{\sqrt{(\beta+\alpha)^2 - 4\alpha\beta z}}}
\right)
\right)}
}
{
e^{\left(
\frac{t}2 \sqrt{(\beta+\alpha)^2 - 4\alpha\beta z}
+
\frac{1}{2}
\log\left(
\frac{1 + \frac{ \beta+\alpha-2\alpha z}{\sqrt{(\beta+\alpha)^2 - 4\alpha\beta z}}}{1-\frac{ \beta+\alpha-2\alpha z}{\sqrt{(\beta+\alpha)^2 - 4\alpha\beta z}}}
\right)
\right)}
+
e^{-\left(
\frac{t}2 \sqrt{(\beta+\alpha)^2 - 4\alpha\beta z}
+
\frac{1}{2}
\log\left(
\frac{1 + \frac{ \beta+\alpha-2\alpha z}{\sqrt{(\beta+\alpha)^2 - 4\alpha\beta z}}}{1-\frac{ \beta+\alpha-2\alpha z}{\sqrt{(\beta+\alpha)^2 - 4\alpha\beta z}}}
\right)
\right)}
}
\right)^2
\right)^{\frac{\eta^*}{2\alpha}}
\\
&
=
\left(
1
-
\left(
\frac{
1-\frac{ \beta+\alpha-2\alpha z}{\sqrt{(\beta+\alpha)^2 - 4\alpha\beta z}}
-
\left(
1 + \frac{ \beta+\alpha-2\alpha z}{\sqrt{(\beta+\alpha)^2 - 4\alpha\beta z}}
\right)
e^{
t \sqrt{(\beta+\alpha)^2 - 4\alpha\beta z}}
}
{
1-\frac{ \beta+\alpha-2\alpha z}{\sqrt{(\beta+\alpha)^2 - 4\alpha\beta z}}
+
\left(
1 + \frac{ \beta+\alpha-2\alpha z}{\sqrt{(\beta+\alpha)^2 - 4\alpha\beta z}}
\right)
e^{
t \sqrt{(\beta+\alpha)^2 - 4\alpha\beta z}}
}
\right)^2
\right)^{\frac{\eta^*}{2\alpha}}
\\
&
=
\left(
\frac{
4
\left(
1-\frac{ (\beta+\alpha-2\alpha z)^2}{(\beta+\alpha)^2 - 4\alpha\beta z}
\right)
e^{
t \sqrt{(\beta+\alpha)^2 - 4\alpha\beta z}}
}
{
\left(
1-\frac{ \beta+\alpha-2\alpha z}{\sqrt{(\beta+\alpha)^2 - 4\alpha\beta z}}
+
\left(
1 + \frac{ \beta+\alpha-2\alpha z}{\sqrt{(\beta+\alpha)^2 - 4\alpha\beta z}}
\right)
e^{
t \sqrt{(\beta+\alpha)^2 - 4\alpha\beta z}}
\right)^2
}
\right)^{\frac{\eta^*}{2\alpha}}
\\
&
=
\left(
\frac{
\frac{ 4\alpha \sqrt{z-z^2}}{\sqrt{(\beta+\alpha)^2 - 4\alpha\beta z}}
e^{
\frac{t}{2} \sqrt{(\beta+\alpha)^2 - 4\alpha\beta z}}
}
{
1-\frac{ \beta+\alpha-2\alpha z}{\sqrt{(\beta+\alpha)^2 - 4\alpha\beta z}}
+
\left(
1 + \frac{ \beta+\alpha-2\alpha z}{\sqrt{(\beta+\alpha)^2 - 4\alpha\beta z}}
\right)
e^{
t \sqrt{(\beta+\alpha)^2 - 4\alpha\beta z}}
}
\right)^{\frac{\eta^*}{\alpha}}
.
\end{align*}
Finally, the third line of Equation~\ref{presimp} is simplified through use of Equations~\ref{cosh} and~\ref{tanhinv}. This expression is then given by
\begin{align*}
&
\left(
\frac{\left(
\frac{1 + \frac{2\alpha z - \beta - \alpha}{\sqrt{(\beta+\alpha)^2 - 4\alpha\beta z}}}{1-\frac{2\alpha z - \beta - \alpha}{\sqrt{(\beta+\alpha)^2 - 4\alpha\beta z}}}
\right)^{\frac{1}{2}}
+
\left(
\frac{1 - \frac{2\alpha z - \beta - \alpha}{\sqrt{(\beta+\alpha)^2 - 4\alpha\beta z}}}{1 + \frac{2\alpha z - \beta - \alpha}{\sqrt{(\beta+\alpha)^2 - 4\alpha\beta z}}}
\right)^{\frac{1}{2}}
}2
\right)^{\frac{\eta^*}{\alpha}}
=
\left(
\left(
\frac{\left(
\frac{1 - \frac{\beta+\alpha-2\alpha z}{\sqrt{(\beta+\alpha)^2 - 4\alpha\beta z}}}{1+\frac{\beta+\alpha-2\alpha z}{\sqrt{(\beta+\alpha)^2 - 4\alpha\beta z}}}
\right)^{\frac{1}{2}}
+
\left(
\frac{1 + \frac{\beta+\alpha-2\alpha z}{\sqrt{(\beta+\alpha)^2 - 4\alpha\beta z}}}{1 - \frac{\beta+\alpha-2\alpha z}{\sqrt{(\beta+\alpha)^2 - 4\alpha\beta z}}}
\right)^{\frac{1}{2}}
}2
\right)^2
\right)^{\frac{\eta^*}{2\alpha}}
\\
&
=
\left(
\frac{
\frac{1 - \frac{\beta+\alpha-2\alpha z}{\sqrt{(\beta+\alpha)^2 - 4\alpha\beta z}}}{1+\frac{\beta+\alpha-2\alpha z}{\sqrt{(\beta+\alpha)^2 - 4\alpha\beta z}}}
+
2
+
\frac{1 + \frac{\beta+\alpha-2\alpha z}{\sqrt{(\beta+\alpha)^2 - 4\alpha\beta z}}}{1 - \frac{\beta+\alpha-2\alpha z}{\sqrt{(\beta+\alpha)^2 - 4\alpha\beta z}}}
}4
\right)^{\frac{\eta^*}{2\alpha}}
\\
&
=
\left(
\frac{
\frac{1 - \frac{\beta+\alpha-2\alpha z}{\sqrt{(\beta+\alpha)^2 - 4\alpha\beta z}}}{1+\frac{\beta+\alpha-2\alpha z}{\sqrt{(\beta+\alpha)^2 - 4\alpha\beta z}}}
+
\frac{1 + \frac{\beta+\alpha-2\alpha z}{\sqrt{(\beta+\alpha)^2 - 4\alpha\beta z}}}{1+\frac{\beta+\alpha-2\alpha z}{\sqrt{(\beta+\alpha)^2 - 4\alpha\beta z}}}
+
\frac{1 - \frac{\beta+\alpha-2\alpha z}{\sqrt{(\beta+\alpha)^2 - 4\alpha\beta z}}}{1-\frac{\beta+\alpha-2\alpha z}{\sqrt{(\beta+\alpha)^2 - 4\alpha\beta z}}}
+
\frac{1 + \frac{\beta+\alpha-2\alpha z}{\sqrt{(\beta+\alpha)^2 - 4\alpha\beta z}}}{1 - \frac{\beta+\alpha-2\alpha z}{\sqrt{(\beta+\alpha)^2 - 4\alpha\beta z}}}
}4
\right)^{\frac{\eta^*}{2\alpha}}
\\
&
=
\left(
\frac{
1 - \frac{\beta+\alpha-2\alpha z}{\sqrt{(\beta+\alpha)^2 - 4\alpha\beta z}}
+
1 + \frac{\beta+\alpha-2\alpha z}{\sqrt{(\beta+\alpha)^2 - 4\alpha\beta z}}
}
{
2
\left(
1 - \frac{(\beta+\alpha-2\alpha z)^2}{(\beta+\alpha)^2 - 4\alpha\beta z}
\right)
}
\right)^{\frac{\eta^*}{2\alpha}}
\\
&
=
\left(
\frac{\sqrt{(\beta+\alpha)^2 - 4\alpha\beta z}}{ 2\alpha \sqrt{z-z^2}}
\right)^{\frac{\eta^*}{\alpha}}
.
\end{align*}
Together these forms give the stated result.
\end{proof}

\end{document}